\documentclass[12pt,a4paper]{article}
\usepackage{amsmath,amssymb,amsthm,graphicx,color}
\usepackage[left=1in,right=1in,top=1in,bottom=1in]{geometry}
\usepackage{cite}
\usepackage[british]{babel}

\newtheorem{theorem}{Theorem}[section]
\newtheorem{lemma}[theorem]{Lemma}  
  
\newtheorem{conjecture}[theorem]{Conjecture}  
\newtheorem{corollary}[theorem]{Corollary} 
\newtheorem{proposition}[theorem]{Proposition} 

\newenvironment{rmk}{\medskip \noindent {\bf Remark.}} \medskip

\allowdisplaybreaks

\numberwithin{equation}{section}
 
\newcommand{\Z}{\mathbb{Z}}
\newcommand{\ZP}{\mathbb{Z}_+}
\newcommand{\R}{\mathbb{R}}

\newcommand{\spa}{{\mathrm{span}}}
\newcommand{\ud}{\textup{d}}
\newcommand{\re}{{\mathrm{e}}}
\newcommand{\ldp}{\lim_{n \to \infty} \frac{1}{n} \log}
\renewcommand{\mod}{\mathrm{mod}~}

\newcommand{\la}{{\lambda}}
\newcommand{\eps}{\varepsilon}

\newcommand{\Bin}{{\rm Bin}}
\newcommand{\Po}{{\rm Po}}

\newcommand{\bt}{{\bf t}}
\newcommand{\bh}{{\bf h}}
\newcommand{\bg}{{\bf g}}
\newcommand{\bbf}{{\bf f}}
\newcommand{\bd}{{\bf d}}
\newcommand{\bw}{{\bf w}}

\newcommand{\tod}{\stackrel{d}{\longrightarrow}}
\newcommand{\ubar}[1]{\underline{#1\mkern-4mu}\mkern4mu }

\newcommand{\N}{\mathbb{N}}
\newcommand{\NN}{{\mathcal{N}}}
\newcommand{\VV}{{\mathcal{V}}}
\newcommand{\EE}{{\mathcal{E}}}
\newcommand{\FF}{{\mathcal{F}}}
\newcommand{\DD}{{\mathcal{D}}}
\newcommand{\GG}{{\mathcal{G}}}

\renewcommand{\Pr}{\mathbb{P}}
\newcommand{\Exp}{\mathbb{E}}
\newcommand{\Var}{\mathbb{V} {\rm ar}}

\newcommand{\Prb}{{\Pr^{\mathrm{bin}}_{\rho_n}}}
\newcommand{\Expb}{{\Exp^{\mathrm{bin}}_{\rho_n}}}
\newcommand{\Expn}{{\Exp_{\rho_n}}}
\newcommand{\Prn}{{\Pr_{\rho_n}}}

\newcommand{\1}{{\bf 1}}

\newcommand{\0}{{\bf 0}}

\newcommand{\starg}{g^*}

\begin{document}

\title{Rank deficiency in sparse random GF$[2]$ matrices}
\author{R.W.R.\ Darling\footnote{Mathematics Research Group, National Security Agency}
\and  Mathew D.\ Penrose\footnote{Department of Mathematical Sciences,
University of Bath}
\and  Andrew R.\ Wade\footnote{Department of Mathematical Sciences,
University of Durham}
\and Sandy L.\ Zabell\footnote{Mathematics Department, Northwestern University}}
\date{\today}
\maketitle

\begin{abstract}
Let $M$ be a random $m \times n$ matrix with binary entries
and i.i.d.\ rows.
The  weight (i.e., number of ones) of a row
has a specified probability distribution,
with the row
chosen uniformly at random  given its weight.
 Let $\NN (n,m)$ denote the number of
left null vectors  in $\{0,1\}^m$ for $M$ (including the zero vector), where
addition is mod 2.
We take $n, m \to \infty$,
with $m/n \to \alpha > 0$, while
the weight distribution may vary with $n$ but
converges weakly to  a limiting
 distribution on $\{3, 4, 5, \ldots\}$;
let $W$ denote a variable with this limiting distribution.
Identifying $M$ with a hypergraph on $n$ vertices,
we define the {\em 2-core} of $M$
as the terminal state of an iterative algorithm
that deletes every  row incident to a column of degree 1.

We identify two thresholds $\alpha^*$ and $\ubar \alpha$,
and describe them analytically in terms of
the
 distribution of $W$.
Threshold $\alpha^*$ marks the infimum of values of $\alpha$ at which
$n^{-1} \log{\Exp [\NN (n,m)}]$ converges to a positive limit,
while $\ubar{\alpha}$ marks the infimum
of values of $\alpha$ at which there is
a 2-core of non-negligible size compared to $n$ having more rows
than non-empty columns.

We have $1/2 \leq \alpha^* \leq \ubar{\alpha} \leq 1$,
and typically these inequalities are strict;
for example when
 $W = 3$ almost surely,  numerics give 
 $\alpha^* = 0.88949 \ldots$ and  $\ubar \alpha=  0.91793 \ldots$
(previous work on this model has mainly
been concerned with such cases where $W$ is non-random).
The threshold of values of $\alpha$ for which
 $\NN(n,m) \geq 2$ in probability
lies in $[\alpha^*,\ubar \alpha]$ and is conjectured to
equal $\ubar \alpha$.

The random row weight setting gives rise to interesting new phenomena not present in the
non-random case
that has been the focus of previous work.
 \end{abstract}

\medskip


\medskip


\newpage

\tableofcontents

\newpage

\section{Introduction}
\label{sec:intro}

Suppose that $M:=M(n,m)$ is an $m \times n$ matrix with
entries in $\{0, 1\}$,  each of whose rows contains at least one $1$, for which we seek a left
null vector over GF$[2]$, i.e.\ a row vector
$ a \in \{0, 1\}^m$ such that $ a  M \equiv \0$ (mod 2), where
here and elsewhere $\0$ is
the all-$0$ vector. 
More generally, elements of $M$ might belong to the finite field
GF[$q$] of order $q$. We are interested in the case where $M$ is sparse and random, as specified below.

 Let
$X_1, X_2, \ldots, X_m$ denote the vectors constituting the
rows of $M$, and let $\sigma(n,m)$
denote the co-rank over GF$[2]$, namely
\begin{equation}
\label{sigmadef}
\sigma(n,m):=m - \dim \spa \{X_1, X_2, \ldots, X_m\},
\end{equation}
where here and subsequently `$\spa$' indicates the linear span over GF[2].
Then the number of null vectors of $M$, including the zero vector, is
\begin{equation}
\label{Ndef}
\NN (n,m) = 2^{\sigma(n,m)},
\end{equation}
which counts the number of distinct solutions
in $ \{0,1\}^m$,
including the zero solution, to
\begin{equation}
\label{leftnull}
a_1 X_1 + \cdots  + a_m X_m \equiv \0 ~(\mod 2).
\end{equation}
Note that for a fixed $n$ and a given realization of the sequence of rows $X_1, X_2, \ldots$, the numbers $\NN(n,m)$ are
nondecreasing as $m$ increases.

Suppose that $n, m \to \infty$, with $m/n \to \alpha > 0$.
Our goal is to examine the limiting behaviour of the expected
 number $\Exp [\NN (n,m)]$ of
left null vectors, and the limiting probability
$\Pr [\sigma(n,m) > 0]$ of a mod-$2$ linear dependency of the rows of
$M(n,m)$,   as a function of the parameter $\alpha$,
and especially to derive computable thresholds
at which phase transitions occur. We also
study the rate of exponential decay of the
probability that $\1 := (1,1,\ldots,1)$ is
a null vector.

The probabilistic setting that we consider has the rows $X_1, X_2, \ldots, X_m$
being independent and identically distributed (i.i.d.) with the
law of a random vector $X = X(n) \in \{0,1\}^n$.  The problem
 has   different flavours depending on the underlying law of  $X$,
and several regimes
have received considerable attention in the literature, including:
\begin{itemize}
\item[(a)] The {\em dense} regime in which $X$ has   order $n$ non-zero
components; the standard  model studied in this regime
has $X$ distributed uniformly over $\{0,1\}^n$.
\item[(b)] The classical {\em sparse} regime in which $X$ has
order $\log n$ non-zero components.
\item[(c)] The uniformly (very) sparse regime in which $X$  has $O(1)$
  non-zero components.
\end{itemize}

The main focus of the present paper is   regime (c) (albeit our `$O(1)$' may be random for each row, and might not even have a mean); in Section
\ref{sec:other} below we briefly discuss
other models that have been studied.
 In the simplest case, $X$ contains a fixed
number $r \leq n$ of non-zero components (the cases $r=1$ and $r=2$
have distinct behaviour from the case $r \geq 3$);
in more generality the number of non-zero
components is randomly distributed according to a given {\em weight distribution}.

Before formally describing our model in detail
 and presenting our main results
(in Section \ref{sec:results}),
we make some remarks on motivation, and on the literature.
Note that $\sigma(n,m)=0$
if and only if $M$ has row rank $m$, which occurs if and
only  if $M$ has column rank $m$. Thus the absence   of
non-trivial left null vectors is equivalent to all column vectors
in $\{0,1\}^m$ being expressible as a linear combination
of the columns of $M$ (with addition modulo 2),
 or in other words, to there being
a solution $x \in \{0,1\}^n$ to $Mx \equiv y$ for all
column vectors $y \in \{0,1\}^m$.
In the special case of $r=2$, motivation for considering
this question is discussed at the start of
\cite[Chapter 3]{kolchinbook}. The following interpretations
help to motivate the general case. \\

{\em A scheduling
problem.}
Suppose that a tennis club is organizing its annual schedule. There
 are $n$ playing days, and $m$ potential players.
Each player wants to play on a given subset of the days; if there is not a match available on every one of these days,
they refuse to pay the annual membership. Each day, in order for nobody
 to be left out, an even number of players is required.
Each possible schedule satisfying these requirements is a left null vector mod $2$; the one with the most units achieves the maximal income
for the tennis club. \\

{\em Randomized Lights Out.}
This is a variant of the game `Lights Out' \cite{muetze}.
Each of $m$ lamps can be either on or off,
and there are $n$ switches, each of which is incident
to a specified subset of the lamps, as given by the
random matrix $M$;
Lamp $i$ and Switch $j$   are mutually incident
 if and only if the entry at $(i,j)$ of $M$ is 1.
If a switch is toggled, all of the lamps incident
to it have their status changed from on to off
or off to on. One may ask whether all states
(i.e.\ configurations of on and off lamps) are
accessible from the `all off' state by
using some sequence of switches (or equivalently,
if the `all off' state is accessible from
all possible starting states), and this
is equivalent to the question of whether the
column rank of $M$ is $m$. \\

{\em The XORSAT problem.} This is a variant
of the random satisfiability problem \cite{MPZ}, where
there are $n$ Boolean variables which may be
deemed true or false.
Each row of $M$ represents a clause built as
the logical XOR (exclusive OR) involving those
Boolean variables corresponding
to columns incident to this row, so  the clause
is  true if an odd number of the variables incident
to the row are
deemed true.  Given a vector $y \in \{0,1\}^m$, finding
a solution $x$ to $Mx \equiv y$ corresponds to
finding a truth-assignment for the Boolean
variables so that each clause $i$ is true
if $y_i=1$ and false if $y_i=0$.
Thus the column rank is $m$ if and only if the problem
is satisfiable  for all possible choices of $y$. \\

 {\em A spin-glass model.}
The relationship  between satisfiability
problems and spin glasses has already been noted
in \cite{MPZ}. In the present instance,
consider the following
variant of the well-known Sherington--Kirkpatrick
mean-field spin-glass model (see e.g.\ \cite{talagrand}).
  There is a random collection of hyperedges on $n$ vertices,
represented by the $m$ rows of $M$. Each
hyperedge $i$   has a  sign $g_i$, taking value
$(-1)^{y_i}$.
Each vertex $j$ is assigned a spin $\sigma_j$
taking values in   $\{-1, 1\}$,
and (at zero temperature) the
probability measure on the state-space is
concentrated on states of minimal energy,
i.e.\ with maximal value of  $\sum g_i e_i$,
where here $e_i$ denotes the product of spins
at vertices in hyperedge $i$. The existence of
a configuration with
all terms in the sum equal to  $+1$
is equivalent to the existence of a solution
to $Mx \equiv y$. \\

 {\em The Ehrenfest urn model and the random walk on the hypercube.}
In the Ehrenfest  model of heat exchange, a box contains $n$ particles, some
of which are red and the rest blue. At each step, a particle
is sampled uniformly from the box and changes its colour.
For a sample of the large literature, see
 e.g.\ \cite[p.\ 121]{feller1}, \cite[\S 3.5]{mahmoud}, \cite[\S 3.5]{moran}, or \cite{jk}.
In the case where $X$ has a single
unit entry, we may view each row of $M$ as selecting which particle
is to be changed  at that step.
Then $\1$ is a null  vector for $M$ if and only if
the model returns
to the initial state after $m$ steps.
This may also be interpreted
as a random walk on the
 graph whose vertices are $\{0,1\}^n$
and edges are present between those vertices that differ
in a single component; the event that $\1$ is null
corresponds to the walker being back
in his starting state
after $m$ steps.

The general case, allowing other weight distributions, corresponds to a generalization of the
Ehrenfest model whereby multiple `diffusions'
are allowed, i.e.\  at each step several
 particles may  change colour at once;
 cf \cite[Chapter 10]{mahmoud}.
This can be similarly
interpreted in terms of a walk on a version of the hypercube
with additional edges. \\

 There is a large body of work on
the properties of random matrices over finite fields
and the closely related subject of
random linear equations over finite fields. Surveys are provided
by the book
\cite[Chapter 3]{kolchinbook} as well as the
articles \cite{kl,levitskaya}.
The problems may also be formulated in terms of {\em random hypergraphs}:
each row represents a hyperedge, and each column represents
a vertex; for details see
Section \ref{sec:cores} below.
They
are also related to the {\em XORSAT} problem
in Boolean algebra,
as mentioned above (see also
 Section \ref{between} below).
Generally, such models can be described in the framework of
{\em random allocation} or {\em occupancy} problems:
see the books \cite{kolchinbook,mahmoud,jk,ksc}.

The null-vector problem in the fixed row-weight case   has received several treatments in the literature, and it is not easy to reconcile all of the existing results,
due to differences in presentation and also differences in the underlying probabilistic models.
One contribution of the present paper is to clarify some of these issues, including giving a rigorous justification
that the results are unchanged under small perturbations of the underlying model. Our main contribution, however, is to treat the case of
genuinely {\em random} row weights, which has not previously been studied. We mention that there has recently been renewed interest in
this area in several scientific communities: for example, Alamino and Saad \cite{as} give a statistical physics approach to the null-vector problem;
Ibrahimi {\em et al.} \cite{ikkm} treat the related problem of random XORSAT; Costello and Vu \cite{cv} study the rank of random symmetric (so in particular, square) matrices.

Throughout the paper, we extend the function $x \mapsto x^x$, $x>0$,
continuously to $x=0$, so that $0^0 :=1$.
  We define the {\em weight} of a vector
$v = ( v_1,\ldots,v_n) \in \{0,1\}^n$
to be $w(v) := \sum_{i=1}^n v_i$, i.e., the number of unit entries.
For $n \in \N := \{1,2,\ldots\}$ we
  write $[n]:= \{1,2,\ldots,n\}$.
We write $\tod$ for convergence in distribution.

\section{Results and discussion}
\label{sec:results}

\subsection{Description of the random matrix model}
\label{sec:matrixlaw}

Given $n \in \N $, suppose
 that $X = X(n) \in \{0,1\}^n $ is a random row vector,
selected according to some
 probability law on $\{0,1\}^n$.
Consider a sequence of i.i.d.\ random vectors
$X_1, X_2, \ldots$ with the same law as $X$.  Let $M := M (n,m)$ be the
 $m \times n$ matrix whose rows are $X_1, X_2, \ldots, X_m$.

We will consider $X$ with law of the following form.
Let $W$  be an $\N$-valued random variable (so $\Pr [ W \geq 1 ] =1$)
whose law will be the (limiting) {\em weight distribution}
of our random vector $X$.
Let $W_1, W_2, \ldots$ be a sequence of random variables
with $W_n \in [n]$ such that $W_n \tod W$ as $n \to \infty$.
Let $w(X)$ have the distribution of $W_n$, and for each
$k \in [n]$ let
the conditional distribution of $X$, given
  $w(X)= k$, be uniform over $\{x \in \{0,1\}^n: w(x) =k\}$.

Let $\rho (s) := \Exp [ s^W]$
and $\rho_n (s) := \Exp [ s^{W_n}]$
 denote the probability generating
functions of $W$ and $W_n$, respectively.
 We use $\Prn$ and $\Expn$ for the
probability and expectation for the random matrix
model with $n$ columns and row weight distribution given by $\rho_n$. We shall say $W_n$ are
{\em uniformly bounded} if
there is a finite constant $r_1$ such that
$\Pr[W_n \leq r_1] =1$ for all $n$
(and hence $\Pr[W \leq r_1] =1$ as well).

\subsection{Threshold results in the general setting}
\label{sec:general}

Given the probability generating function $\rho$,
define the threshold
\begin{equation}
\label{alphastar2}
\alpha _{\rho}^* := \inf  \{ \alpha \geq 0: F_\rho (\alpha ) > 0 \},
\end{equation}
where we set
\begin{equation}
\label{Fdef2}
 F_\rho (\alpha)  :=
\log \sup_{\gamma \in [0,1/2]} \left( \frac{ (1 +   \rho (1-2\gamma)  )^\alpha }{2 \gamma^\gamma (1-\gamma)^{1-\gamma} } \right),
~~~ \alpha \geq 0
.
\end{equation}

We state some fundamental properties of $F_\rho(\alpha)$ and $\alpha^*_\rho$ in the next result, which we prove in Section \ref{secprelim}.

\begin{proposition}
\label{alphaprop}
We have
 $F_\rho(\alpha) = 0$ for $ 0 \leq \alpha \leq \alpha^*_\rho$
but $F_\rho(\alpha) >0$ for $\alpha > \alpha^*_\rho$,
and $F_\rho$ is continuous and nondecreasing as a function of $\alpha$.
Moreover:
\begin{itemize}
\item[(i)] $\alpha^*_\rho \in [0,1]$; and $\alpha^*_\rho < 1$ if $\Exp[W]<\infty$.
\item[(ii)] $\alpha^*_\rho = 0$ if and only if $\Pr [ W = 1] >0$.
\item[(iii)] If $\Pr [W=2] =1$, then $\alpha^*_\rho = 1/2$.
\item[(iv)] Suppose that $\tilde W$ is another $\N$-valued random variable, with $\tilde \rho (s) = \Exp [ s^{\tilde W} ]$,
such that $\tilde \rho(s) \leq \rho(s)$ for all $s \in [0,1]$ (which is the case, for example, if $\tilde W$ stochastically
dominates $W$). Then $\alpha^*_{\tilde \rho} \geq \alpha^*_{\rho}$.
\end{itemize}
In particular, if $\Pr [ W \geq 2] =1$ and $\Exp[W]<\infty$, then $\alpha^*_\rho \in [1/2,1)$.
\end{proposition}

Here is our first main result,
 describing the threshold behaviour of the expected number
of null vectors. We shall prove this in Section \ref{mainproofs}.

\begin{theorem}
\label{thm1}
Suppose that $m_n/ n \to \alpha \in (0,\infty)$
  as $n \to \infty$.  Then
\begin{equation}
\label{thm1eq}
 \lim_{n \to \infty }  n^{-1}\log
\Expn [ \NN (n, m_n ) ] =
F_\rho (\alpha )
;\end{equation}
in particular, the limit in (\ref{thm1eq}) is strictly positive for $\alpha > \alpha_\rho^*$ and zero for $\alpha \leq \alpha^*_\rho$.
Moreover, if in addition
 there exist $r_0 \geq 3$ and $r_1 < \infty$ such
 that $\Pr  [ r_0 \leq W_n \leq r_1 ] =1$
for all $n$, and
 $\alpha \in (0, \alpha_\rho^*)$,
then as $n \to \infty$,
\begin{equation}
\label{polybound}
\Expn [ \NN (n, m_n ) ] = 1 + O ( n^{2-r_0} ).
\end{equation}
 \end{theorem}

 \begin{rmk}
 For $\alpha < \alpha^*_\rho$, the  expectation in (\ref{polybound})
 is dominated by low-weight null vectors (short hypercycles), in particular by
  null vectors with only 2 non-zeros. It may be possible, by extending
the argument of Lemma \ref{lem3} below, to expand this expectation
as a power series in $n^{-1}$, but we do not pursue this here.
 For $\alpha > \alpha^*_\rho$ the
  exponential  
 rate in (\ref{thm1eq})
 is dominated by null vectors using a specific positive proportion $\beta_0 = \beta_0 (\alpha)$
 of the (roughly $\alpha n$) available rows
 (and also possibly those using proportion $1-\beta_0$ of the rows, due to parity effects);
 in fact, $\beta_0 = \frac{\rho(1-2\gamma_0)}{1+\rho(1-2\gamma_0)} \in (0,1/2)$,
 where $\gamma_0 = \gamma_0 (\alpha) \in (0, 1/2)$ is the value of $\gamma$ for which the
 supremum in (\ref{Fdef2}) is attained. See Section \ref{mainproofs} for details.
 \end{rmk}

Theorem \ref{thm1}
deals with the {\em expected} number of null vectors.
Also of interest is, for a fixed $n$, the (random) number of rows $m$
at which the first non-zero null vector   appears.
Define
\begin{equation}
\label{Tdef}
 T_n :=\min \left\{m \in \N: X_m \in
\spa \left\{X_1, X_2, \ldots , X_{m-1}\right\}
\right\},
\end{equation}
the first $m$
for which $\mathrm{rank} (M(n,m)) < m$.
 Standard linear algebra implies that $T_n \leq n+1$.

We define another threshold, $\ubar \alpha_\rho$,
through an analytic description that needs more notation
(we shall give a probabilistic
interpretation later on).
For $x \in (0,1)$ set
 \begin{align}
  \label{psidef}
   \psi ( x ) & := x + \left( 1 + \frac{\rho(x)}{\rho'(x)} -     x \right) \log ( 1 - x) ; \\
   \label{hdef}
     h (x) & := - \frac{\log (1-x)}{\rho'(x) } .
 \end{align}
Provided $\Pr [W \geq 1]=1$ we can and do
  extend  $\psi$  continuously to $\psi (0) :=0$,   since
   $\rho(s)/\rho'(s) = O(s)$ as $s \downarrow 0$.
Note that 
$h(x) \to \infty$ as $x \downarrow 0$ provided
$\Pr [ W \geq 3]=1$, and that if $\Exp[W]< \infty$ then
as $x \uparrow 1$ we have
$h(x) \to \infty$  and $\psi(x) \to -\infty$.
Set
  \begin{equation}
  \label{alphasharp}
  \alpha_\rho^\sharp :=
\inf_{x \in (0,1)} h(x),
  \end{equation}
and note that 
$\alpha_\rho^\sharp \rho'(x) \leq -\log(1-x)$ for
all $x \in (0,1)$, so integrating from $0$ to $1$ 
we get $\alpha_\rho^\sharp \leq 1$, 
provided  $\Pr[W \geq 1] =1$.  
For $\alpha \geq 0$, define 
    \begin{equation}
    \label{gs0}
    g^*(\alpha) := \sup \{ x \in ( 0,1) : h (x) \leq \alpha \} ,
    \end{equation}
with the convention $\sup \emptyset =0$ in operation in (\ref{gs0}).
Observe that 
 if $h$ has unbounded range (e.g. if $\Pr[W \geq 3] =1$)
then $h \circ g^*$ is the identity map
on $[\alpha_\rho^\sharp,\infty)$.
 See Figure \ref{fig1} for an example.
Define
  \begin{equation}
\ubar \alpha_\rho := \inf \{ \alpha > \alpha_\rho^\sharp : \psi (g^*( \alpha) ) < 0\} .
\label{ubardef}
\end{equation}
 In (\ref{ubardef}), the set defining $\ubar \alpha_\rho$ is non-empty provided
 $\Pr [ W \geq 3 ]=1$ and $\Exp [W] < \infty$,
since as $\alpha \to \infty$ we have $g^*(\alpha) \to 1$ and $\psi(g^*(\alpha)) \to -\infty$.

The relevance of $\ubar \alpha_\rho$ for the null vector problem is shown by the next result.

\begin{theorem}
\label{thm3}
Suppose $W_n$ are
uniformly bounded and
  $\Pr  [ W_n \geq 3 ] =1$
for all $n$.
Then $\alpha^*_\rho \leq \ubar \alpha_\rho \leq 1$, and
 for any $\eps>0$,
\begin{equation}
\label{thm3eq1}
\lim_{n\to \infty } \Prn \left[( \alpha _{\rho }^* -\eps) n \leq   T_n \leq (
 \ubar \alpha_{\rho }  +\eps) n \right] = 1.\end{equation}
\end{theorem}

Theorem \ref{thm3}  is proved in Section
\ref{secPoisson}.
We exclude the case where $\Pr [ W \in \{1,2\} ] >0$  from the statement of Theorem \ref{thm3}; different phenomena occur in that case
(see 
Proposition \ref{prop3} below).
This case is also discussed in
 \cite{dln}, where 
 the functions $\psi$ and $h$ also play a role.

  Typically, $\ubar \alpha_\rho$ defined by (\ref{ubardef})
 will satisfy $\psi ( g^* ( \ubar \alpha_\rho) ) =0$.
In many cases, there is a single solution in $(0,1)$, denoted $x^*_\rho$, to
$\psi(x)=0$, and $\ubar \alpha_\rho = h(x^*_\rho)$.
However,
  the situation is complicated by the fact that $\alpha \mapsto g^*(\alpha)$
 typically has at least one discontinuity.
We defer a more detailed discussion of the properties of the functions $\psi$, $h$, and $g^*$, and the corresponding thresholds, to
Section \ref{secPoisson} below. Figure \ref{fig1} provides an example.

\begin{figure}[!h]
\center
\includegraphics[width=7cm]{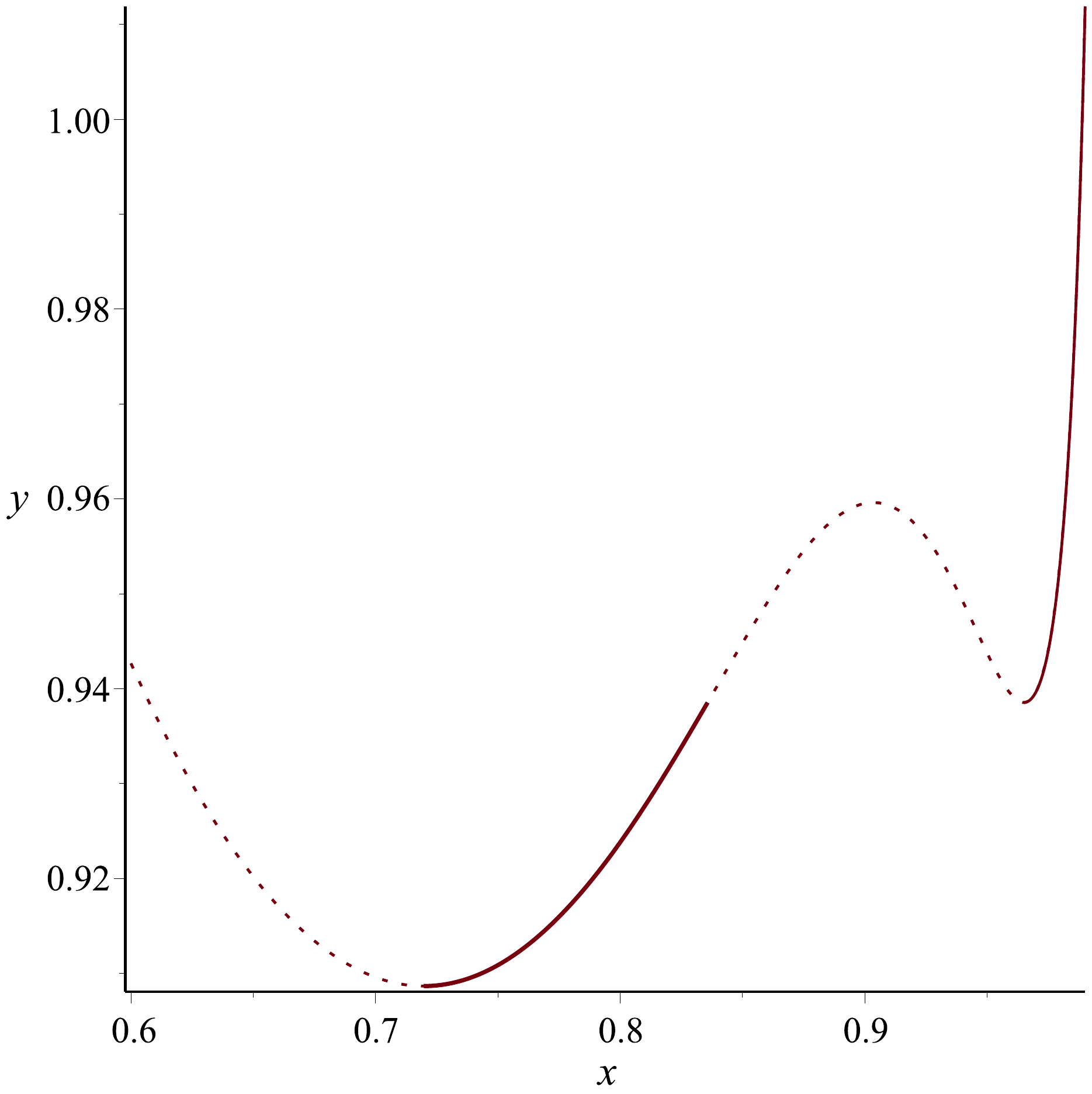} \quad
\includegraphics[width=7cm]{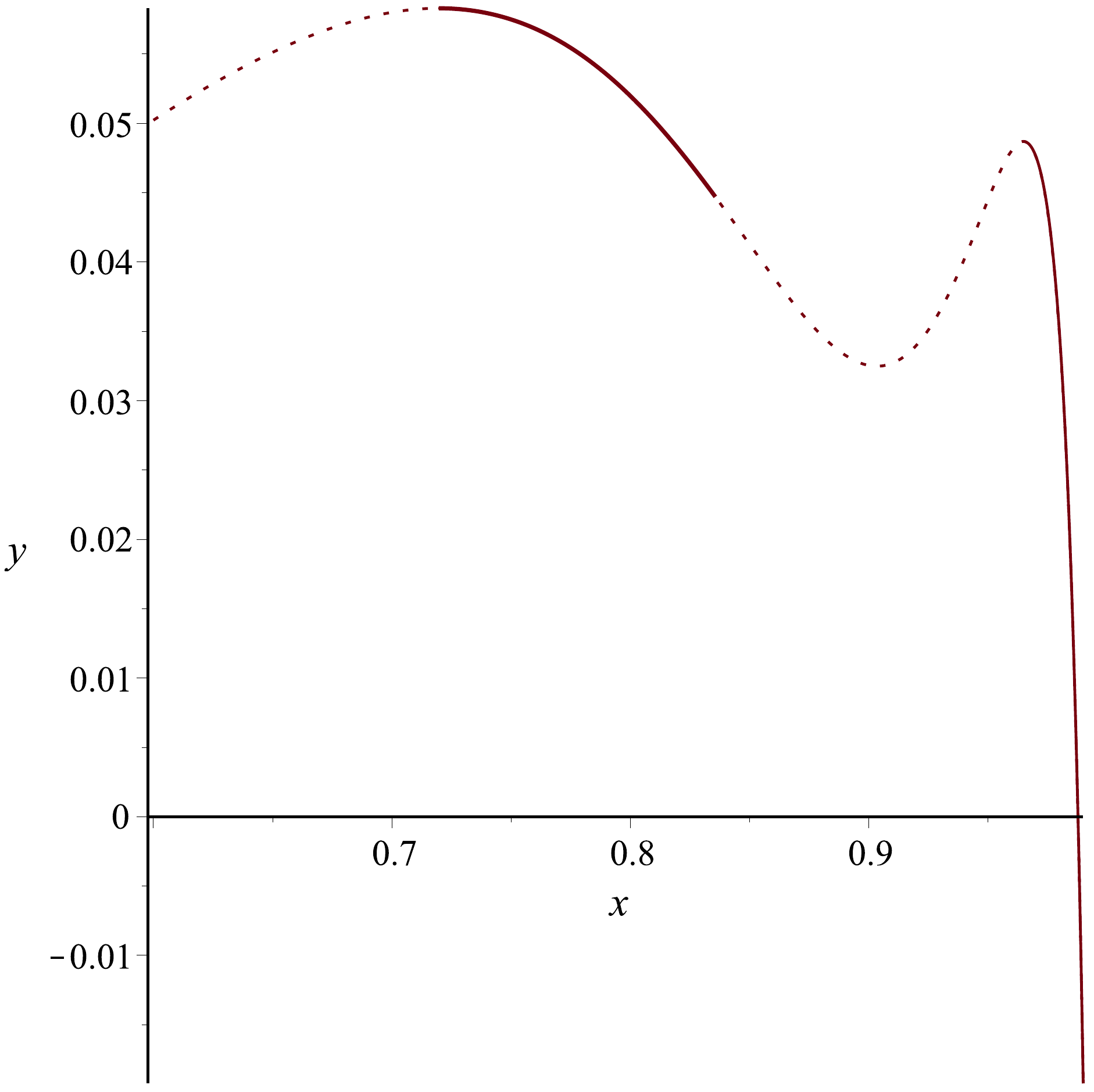}
\caption{Example with $\rho(s) = 0.9 s^3 + 0.1 s^{24}$. The left plot shows parts of the 
curves $y =h(x)$ (all the line) and $x = g^*(y)$ (solid line).
The right plot shows parts of the curves $y= \psi(x)$ (all the line) and
the locus of $(g^*(\alpha),\psi(g^*(\alpha)))$ (solid line).
The left plot shows that $g^*(\alpha)$ has two discontinuities, one at $\alpha = \alpha_\rho^\sharp \approx 0.908654$ and one at
$\alpha \approx 0.938536$, with the first corresponding to a jump from $g^* =0$ to $g^* \approx 0.719682$ and the second
to a jump from $g^* \approx 0.835696$ to $g^* \approx 0.964919$. The right plot shows the single
 positive solution  of $\psi(x)=0$  
at $x = x^*_\rho \approx 0.987817$, so $\ubar \alpha_\rho   =h(x^*_\rho) \approx 0.991613$. 
It is not a coincidence that
the curves $h$ and $\psi$ seem to mirror each other: see Lemma \ref{lem:hpsi} below.}
\label{fig1}
\end{figure}

Theorem \ref{thm3}
leaves open the sharp asymptotics of $T_n/n$: we believe that the upper
bound in Theorem \ref{thm3} (i.e., $\ubar \alpha_\rho$) is sharp:
\begin{conjecture}
\label{conj1}
If $W_n$ are uniformly bounded and
$\Pr  [  W_n \geq 3 ] =1$
for all $n$, then
$T_n/n$ converges in
probability to $\ubar \alpha_\rho$ as $n \to \infty$.
\end{conjecture}

An equivalent statement to the fixed-weight case $W = r \geq 3$ of this conjecture
seems to have been established recently in the random-XORSAT literature: see the comments
in Section \ref{between}.

The probabilistic interpretation
of the thresholds
$\alpha_\rho^\sharp$ and  $\ubar \alpha_\rho$
is in terms of the
{\em 2-core} of $M(n,m)$;
this is the terminal state of an iterative algorithm
that deletes every  row incident to a column of degree 1
 (see Section \ref{sec:cores} below for details).
Let $E(n,m; \eps )$ denote the event that $M(n,m)$ possesses a $2$-core (i) whose number of rows is bounded below by  $\eps n$,
and (ii) which contains more rows
than columns of non-zero weight  (all of which have weight 2 or more).
  In particular, for $\eps>0$, $E(n,m ; \eps)$ implies that $M(n,m)$ has a
 non-empty $2$-core.
If $W_n$ are uniformly bounded,  and $m_n / n \to \alpha > 0$,
 then in Theorem \ref{c:alphabar} below we will
 show that, under certain additional conditions on $\rho$, there exists $\eps>0$ (allowed to depend on $\alpha$) such that
 $\lim_{n \to \infty} \Prn  [ E ( n, m_n ; \eps ) ] = 1$
 for   $\alpha$ in some interval of the form $(\ubar \alpha_\rho , \ubar \alpha_\rho + \delta)$ with 
$\delta >0$.

A more delicate description of the behaviour of the $2$-core
in terms of the function $\psi$ defined at (\ref{psidef})
will be given in Theorem \ref{c:alphabar} below: the limiting aspect ratio of the $2$-core being less than or greater than $1$ depends
on the sign of $\psi(g^*(\alpha))$. In the example shown in Figure \ref{fig1}, and also in the fixed weight setting,
$\psi(g^*(\alpha))$ changes sign only once, but in the the general random weight setting it may change sign {\em multiple} times; see
the example in Figure \ref{fig2} below. Thus the random weight setting gives rise to subtle new phenomena not present in the 
fixed weight case
that has been the focus of previous work.

\begin{rmk} Our techniques may be used to obtain information about the
{\em weight profile} of the left null space. Bayes' Theorem applied to equation (\ref{byweight}) below 
shows that the weight
of a randomly chosen left null vector has distribution given by
\[ \Prn [ w (v) =  k \mid v  M \equiv \0 ] = \frac{ \binom{m}{k} \Prn [ A (n, k)]}{\Expn [ \NN (n,m) ]} ;\]
 the quantities on the right-hand side here are studied in detail in Section \ref{sec:occupancy}.
\end{rmk}

\subsection{Even occupancy in random allocations}
\label{sec:allo}

Let $A (n,m)$ denote the event
that the row vector $\1 = (1,\ldots, 1)$ is  null
for $M$,
i.e.,
\begin{equation}
 A (n,m) := \{ X_1 + \cdots + X_m \equiv \0 ~ (\mod 2) \}.
\label{Adef}
\end{equation}
 One interpretation is in terms of
the {\em random allocation model}.
Suppose we have $n$ urns,
and for each row of $M$ we allocate a collection
of balls to a chosen set of  urns (determined
by the unit entries of that row of $M$).
Event $A(n,m)$ is the event that all
the urns end up with an even number of  balls.
Random allocations have been extensively studied;
see e.g.\
\cite[p.\ 101]{feller1}, and
the monographs
 \cite{kolchinbook,ksc,mahmoud,jk}.

The following theorem, which we prove
in Section \ref{sec:occupancy}, describes the exponential
rate of decay
 for $\Prn [A(n,m_n)]$ where $m_n/n$ has a finite positive limit.
The theorem excludes the case in which {\em both} $W$  and $m_n$ only take odd
values; 
 if $m$ is odd and $W_n$ is odd a.s., then
$\Prn  [ A(n,m ) ] =0$ since the total number of units in the matrix is odd.

\begin{theorem}
\label{allothm}
Suppose that
$m_n/n \to \alpha \in (0,\infty)$ as $n \to \infty$,
and that either  (i) $m_n \in 2 \Z$ for all $n$; or (ii)
$\Pr [ W  \in 2\Z ] >0$.
Then
\begin{align}
\label{allolim}
\lim_{n\to \infty}
  n^{-1} \log \Prn  [ A(n,m_n) ]
 = - R_\rho (\alpha) ,\end{align}
where $R_\rho(\alpha) > 0$ is continuous and nondecreasing in $\alpha >0$ and is defined by
\begin{align}
\label{Rrhodef}
R_\rho (\alpha) :=
- \log
\sup_{\gamma \in [0,1/2]}
\left(
\frac{ \left( \rho (1 - 2\gamma) \right)^\alpha}
{2 \gamma^\gamma (1-\gamma)^{1-\gamma}}
\right).
\end{align}
\end{theorem}

The relevance of Theorem \ref{allothm}   to Theorem \ref{thm1} is clear; see
the formula (\ref{eq1}) below. 
We present an interesting special case of
Theorem \ref{allothm}
  that
can be understood in isolation.
Let $\pi_{n}(m)$ denote the probability that
 all the components $Y_j$ of a multinomial
$(m; n^{-1}, \ldots, n^{-1} )$ random vector
$(Y_1, \ldots, Y_n)$ are even.
Here $Y_j$ can be interpreted as the occupancy of urn $j$
after $m$ balls are independently and uniformly distributed
into $n$ distinct urns: see e.g.\ \cite[p.\ 23]{moran}, \cite[p.\ 11]{mahmoud}, or \cite[p.\ 90]{jk}.

Then
$ \pi_n (m)  = 2^{-n} \sum_{j=0}^n \binom{n}{j}
\left( 1 - (2j/n )  \right)^m ;$
this formula
is known in the Ehrenfest urn literature
 (see \cite[pp.\ 128--129]{mahmoud})
 and can also be obtained from (\ref{pmulti})
below.
 If $m$ is  odd, $\pi_n (m)$ must be zero.

\begin{proposition}
\label{multi}
 Let $\pi _n(m_n)$ denote the
 probability that all the $n$ components
  of a multinomial $(m_n;
n^{-1},\ldots , n^{-1})$ random vector
 are even. Suppose that  $m_n$ is even for each $n$ and
$m_n/n \to \alpha = \lambda \tanh \lambda \in (0,\infty)$ as $n \to \infty$.
Then
\begin{equation}
\label{multilim}
\lim_{n\to \infty}
  n^{-1} \log \pi_n(m_n)  =
\log  \cosh \lambda  -
( \lambda  \tanh \lambda ) (1 - \log  \tanh \lambda  ) .
\end{equation}
\end{proposition}
This result follows from Theorem \ref{allothm}, as we shall
show in
Section \ref{sec:alloproof}.
Proposition \ref{multi} can also be derived from a result
of Kolchin  \cite[Theorem 2, p.\ 141]{kolchin}.

\subsection{The fixed row-weight case}
\label{sec:fixed}

We now describe the special case of
the results in Section \ref{sec:general} when
 $\Pr[W = r]=1 $,  for some fixed $r$.
The existing literature is largely concerned with
this case  (see the discussion in Section \ref{sec:related} below).

Let $r \in \N$.
Define the  thresholds $\alpha^*_r$, $\alpha_r^\sharp$, and $\ubar \alpha_{r}$ 
to be the values of $\alpha^*_\rho$, $\alpha_\rho^\sharp$,  and $\ubar \alpha_\rho$,  
respectively, in the case where
$\Pr[W = r]=1$ (i.e.\ where $\rho(s)=s^r)$.
By (\ref{alphastar2})
we
have
\[
\alpha^*_r := \inf \{ \alpha \geq 0 : F_r (\alpha ) > 0 \} ,
\]
where
\begin{equation}
\label{Fdef}
F_r (\alpha) :=
\log \sup_{ \gamma \in [0,1/2] }   \left( \frac{(1 + (1-2\gamma)^r )^\alpha }
{2 \gamma^\gamma (1-\gamma)^{1-\gamma} } \right) .
\end{equation}
By
 Proposition \ref{alphaprop},
$F_r(\alpha) = 0$ for $\alpha \leq \alpha_r^*$
 but $F_r(\alpha) > 0$ for $\alpha > \alpha_r^*$.
 If $r \geq 3$,
 $\psi(x) =0$ has a single solution in 
 $(0,1)$ (see Proposition \ref{psiprop} below)
  denoted $x^*_r$ and satisfying
 \begin{equation}
 \label{MP3}
 x^*_r = - \left( 1 -   \left( \frac{r-1}{r} \right) x^*_r \right) \log ( 1 - x^*_r),
 \end{equation}
 and
\begin{equation}
\label{fixedbar} \ubar \alpha_r = h (x^*_r) = - \frac{\log (1-  x^*_r )}{r   (x^*_r)^{r-1}} .\end{equation}
For example,  
$\alpha_3^\sharp \approx 0.818469$, $g^* (\alpha_3^\sharp) \approx 0.715332$, and
$x^*_3 \approx 0.883414$ so $\ubar \alpha_3 \approx 0.917935$
(see also Figure \ref{fig8} below).
The next result is a specialization of Theorems \ref{thm1} and \ref{thm3}
together with
 Theorem \ref{c:alphabar} and Proposition \ref{lessthan1}.

\begin{theorem}
\label{thm0}
Let $r\in \N$. Suppose that $W_n \to r$ in probability.
Suppose that $m_n /n \to \alpha \in (0,\infty)$ as $n \to \infty$.
Then
\begin{equation}
\label{meandecay}
 \lim_{n \to \infty }  n^{-1}\log
\Expn [ \NN (n, m_n ) ] =  F_r (\alpha);
\end{equation}
in particular, the limit in (\ref{meandecay}) is strictly positive if and only if $\alpha > \alpha_r^*$.
If also $\Pr [ W_n \geq 3]=1$ (so $r \geq 3$) and $W_n$ are uniformly bounded, then $\alpha^*_r \leq
\ubar \alpha_r \leq 1$ and for any $\eps >0$,
$$
\lim_{n\to \infty } \Prn \left[( \alpha _{r}^* -\eps) n
\leq   T_n \leq ( \ubar{\alpha }_{r}+\eps) n \right] = 1,
$$
and moreover, there exists $\eps_\alpha >0$ such that for all $\eps \in (0,\eps_\alpha)$,
\begin{equation}
\label{zeroone}
\lim_{n \to \infty}
\Prn
 [ E (n, m_n ; \eps ) ] = \begin{cases}
0 & ~\textrm{if}~ \alpha < \ubar \alpha_r  \\
1 & ~\textrm{if}~ \alpha > \ubar \alpha_r .\end{cases}
\end{equation}
\end{theorem}

The sharp monotone transition displayed by (\ref{zeroone})
in the fixed weight case was indicated by
Cooper \cite{cooper2}.
In the general, random weight setting, the picture can be more complicated,
and   there exist examples where the transition is {\em non-monotone}: see the example
in Figure \ref{fig2} below.

\begin{rmk}
In the case $r=2$, compare Theorem \ref{thm0} to
Proposition \ref{prop3}(ii) below: the first cycle in the random graph
appears at $m_n = Zn$, where $Z$ has an asymptotic distribution on $(0,1/2)$.
It is a classical result that for $\alpha \in (0,1/2)$, if $m_n / n \to \alpha$, the number of cycles
$\NN (n, m_n)$ in an Erd\H os--R\'enyi graph has a Poisson limit with finite expectation, but the limiting expectation
is infinite for $\alpha \geq 1/2$ (see e.g.\ \cite[\S2.3]{kolchinbook}); we could not find in the literature an explicit reference
to the fact that the expectation blows up {\em exponentially} with $n$ for $\alpha > 1/2$, at the rate given by  Theorem \ref{thm0}
(a classical result of Erd\H os and R\'enyi  states that at $\alpha = 1/2$, the expected number of cycles grows as $\frac14 \log n$: see
Theorem 5a of \cite[p.\ 41]{er60}).
\end{rmk}

\subsection{Threshold numerics and asymptotics}
\label{sec:numbers}

In this section we discuss numerical and asymptotic evaluation of the thresholds in our results for the fixed row-weight case described in Section \ref{sec:fixed}.
Table \ref{tab1} shows values of $\alpha_r^\sharp$, $\alpha^*_r$,
and $\ubar{\alpha }_r$, for $r \leq 8$. Previous computations of these thresholds
 are reviewed
in Section \ref{sec:related}.
As suggested by the numerical results,  
 it can be shown that, for $r$ large enough, $\alpha_r^\sharp < \alpha^*_r < \ubar \alpha_r < 1$;
this is a consequence of the following result.

\begin{proposition}
\label{threshasym}
As $r \to \infty$,
\begin{equation}
\label{starsim}
\alpha^\sharp_r \to 0 ; ~~~
1- \alpha^*_r \sim  \frac{\re^{-r}}{\log 2} ; ~~~ 1 - \ubar \alpha_r \sim  \re^{-r} .\end{equation}
\end{proposition}

The  asymptotic result for $\alpha^*_r$ in (\ref{starsim}) is due to Calkin \cite{calkin2}; we
prove the other two parts in Section \ref{secPoisson} below.

\begin{table}
\center
\begin{tabular}{|c|c|c|c|c|c|c|c|c|}
\hline
$r$ & 1 & 2 & 3 & 4 & 5 & 6 & 7 & 8 \\
\hline
$\alpha_r^\sharp$  &  0 &  0.5 &  0.818469 & 0.772280 & 0.701780 & 0.637081 & 0.581775 & 0.534997  \\
\hline
$\alpha^*_r$ & 0 & 0.5 & 0.889493 & 0.967147 & 0.989162 & 0.996228 & 0.998650 & 0.999510 \\
\hline
$\ubar{\alpha }_r$ &  --- & ---  & 0.917935 & 0.976770 & 0.992438 & 0.997380 & 0.999064 & 0.999660 \\
\hline
\end{tabular}
\caption{Threshold parameters for $r$-uniform random hypergraphs. Note that $\ubar{\alpha}_r$ is not defined
when $r =1$ or $2$.}
\label{tab1}
\end{table}

One can obtain arbitrarily sharp upper and lower bounds for the solution  $x^*_r \in (0,1)$
of $\psi(x)=0$
 in the case $\rho(s)=s^r$, $r \geq 3$, as follows. In this case, by
(\ref{MP3}) we have that
$x^*_r = i_r(x^*_r)$, where we set
\[ i_r (x) := 1 -\exp \left\{ - \frac{x}{1-(\frac{r-1}{r}) x } \right\} .\]
For $\theta \in [0,1)$, $x \mapsto \frac{x}{1-\theta x}$ is strictly increasing for $x \in [0,1]$. 
Thus if $x^*_r > a_n$, it follows that
$x^*_r > a_{n+1} := i_r (a_n)$. Also, $i'_r(0) = 1$, $i''_r (0) = \frac{2-r}{r} >0$,
 and $i_r (1) = 1 -\re^{-r} <1$, so $i_r(x) > x$ for $x \in (0,x^*_r)$ but 
$i_r(x) < x$ for $x \in (x^*_r,1]$.
Hence starting with $a_0 = \frac{r-2}{r-1} < x^*_r$ (an inequality proved in Proposition \ref{psiprop} below), we can
 iterate to obtain an increasing sequence of lower bounds $a_n$ for $x^*_r$.
Conversely, starting instead with $b_0 =  1 > x^*_r$ and iterating $b_{n+1} := i_r (b_n)$
 gives a decreasing sequence of upper bounds $b_n$ for $x^*_r$. For example,
after one step we get
\[  1 - \exp \left\{ - \frac{r(r-2)}{2(r-1)} \right\} < x^*_r <  1 - \re^{-r}  , ~~~ (r \geq 3).\]
Proceeding up to $b_2$ for the upper bound and $a_4$ for the lower bound is sufficient to obtain the $r \to \infty$ asymptotic expression
\begin{equation}
\label{grsim}
 x^*_r = 1 - \re^{-r} - r^2 \re^{-2r} + O (r^4 \re^{-3r} ) ,\end{equation}
which will be the main ingredient in the proof of the $\ubar \alpha_r$ result in (\ref{starsim}).

In fact, this iterative procedure converges, so $a_n \uparrow x^*_r$ and $b_n \downarrow x^*_r$.
To prove convergence
 it is sufficient to show that
$i'_r(x) < 1$ at $x = x^*_r$. A calculation shows that 
$i'_r(x)$ evaluated at $x = x^*_r$ comes to
$x^{-2} (1-x) (\log (1-x))^2$, so for the required inequality it
suffices to show that
\[ - x^{-1} \log (1-x) < (1-x)^{-1/2} , ~\textrm{for}~ 0 < x < 1 .\]
The coefficient of $x^k$ in the power series expansion of the left-hand side of the last inequality is
$1/(k+1)$, and for the right-hand side it is $4^{-k} \binom{2k}{k}$, and both series are
convergent on the given interval. An induction shows that $1/(k+1) \leq
4^{-k} \binom{2k}{k}$ for all integers $k \geq 0$. So term-by-term comparison
of the two power series gives the inequality.

To end this section we discuss the numerical evaluations  of $\alpha^*_r$ in Table \ref{tab1}.
In this discussion we use some claimed properties of the functions involved that we do not
verify rigorously, since here we are only concerned with numerical estimation.
 Let
 \begin{equation}
 \label{Frdef} F_{r,\alpha} (\gamma) := \log \left( \frac{(1 + (1-2\gamma)^r )^\alpha }
{2 \gamma^\gamma (1-\gamma)^{1-\gamma} } \right) ,\end{equation}
so that $F_r (\gamma) = \sup_{\gamma \in [0,1/2]} F_{r,\alpha} (\gamma)$.
Differentiating, we obtain
\begin{equation}
\label{Fdiff}
\frac{\ud }{\ud \gamma} F_{r,\alpha} (\gamma) = - \frac{2\alpha r (1-2\gamma)^{r-1}}{1+(1-2\gamma)^r} + \log \left( \frac{1-\gamma}{\gamma} \right) .\end{equation}
Thus $\gamma$ is a stationary value for $F_{r,\alpha}$ if $\gamma$ solves
\begin{equation}
\label{gamma0}
\alpha =  \frac{1 + (1 - 2 \gamma )^r}{2 r (1 - 2 \gamma )^{r-1}}
\log \left( \frac{1 - \gamma }{\gamma } \right) =: \alpha_r (\gamma).
\end{equation}
Numerical curve sketching shows that (\ref{gamma0}) generically has at most 2 solutions in $(0,1/2)$;
of such solutions, the smallest will be the local maximum, since $F'_{r,\alpha} (\gamma) \to \infty$ as $\gamma \downarrow 0$,
by (\ref{Fdiff}). If (\ref{gamma0}) has no solutions in $(0,1/2)$, then the supremum in (\ref{Fdef})
is either $F_{r,\alpha} (0) = (\alpha -1) \log 2$ or $F_{r,\alpha}(1/2) = 0$.
Thus setting $\gamma_0 := \gamma_0 (\alpha, r) = 0$ if (\ref{gamma0}) has no solutions in $(0,1/2)$ and
$\gamma_0 := \gamma_0 (\alpha, r)$ to be the  smallest positive solution to (\ref{gamma0}) otherwise,
we have that $F_r(\alpha) = F_{r, \alpha} (\gamma_0)$ whenever $\alpha \in (0,1)$.

For $\alpha \in (0,1)$, $F_r(\alpha) >0$
if and only if $\gamma_0 (\alpha, r) > 0$.
Moreover, Proposition \ref{alphaprop} shows that $\alpha^*_r <1$, so
that for $\alpha <1$ such that $\gamma_0 (\alpha ,r) >0$,
$F_r(\alpha) = \alpha_r (\gamma_0) \log ( 1 + (1-2\gamma_0)^r ) - \log ( 2 \gamma_0^{\gamma_0} (1-\gamma_0)^{1-\gamma_0} )$.
Thus to find $\alpha_r^*$, we solve for $\gamma \in [0,1/2]$ the equation
\begin{equation}
\label{numerical}
\alpha_r (\gamma) -\phi_r (\gamma) = 0 ,
\end{equation}
where  \[
\phi_r (\gamma) = \frac{\log{(2 \gamma^{\gamma} (1 - \gamma)^{1 - \gamma})}}{\log{(1 + (1 - 2 \gamma)^r)}}.
\]
Numerical curve plotting shows that $\gamma \mapsto \alpha_r (\gamma) -\phi_r (\gamma)$ is decreasing
on $[0,1/2]$, so  (\ref{numerical}) can be solved using efficient numerical methods; let $\gamma_r$
denote the solution to (\ref{numerical}). Then we compute $\alpha^*_r$ via
$\alpha^*_r  = \alpha_r (\gamma_r)$.

\subsection{Discussion and related results}
\label{sec:related}

\subsubsection{Previous results on threshold values}
In the simplest case, $W_n = W_n^{\rm hyp} := r \wedge n$ a.s.,  for a fixed $r \in \N$; then $W=r$ a.s.\
This fixed row weight `hypergeometric' model is  studied by Cooper \cite{cooper1}.
A   variation is the model in which $r$ units are assigned to the row
{\em independently} and uniformly at random, with multiplicities reduced mod $2$.
The latter `binomial' model corresponds to $W_n = W_n^{\rm bin}$ distributed as the number of odd
components in a multinomial $(r; n^{-1}, \ldots, n^{-1})$ random vector; then
$W_n \tod r$ (see Lemma \ref{wnconv} below). The $r \geq 3$ binomial model is  studied
by Kolchin \cite{kolchin}. Note that in this model rows of all zeroes may appear,
in which case they are ignored (in other words, empty hyperedges are discounted):
this is a small effect since $\Pr [ W_n^{\rm bin} =0 ] = O ( n^{-r/2} )$,
so a vanishing proportion of rows needs to be discarded.

Phase transitions in the null vector problem for random matrices over finite
 fields with  fixed row weight $r \geq 3$
  have been studied since the early 1990s.
In the case of the binomial model,
 the threshold $\alpha^*_r$, $r \geq 3$, for $\Exp [ \NN (n ,m_n ) ]$,
$m_n / n \to \alpha$, was described by Balakin {\em et al.} \cite{bkk} and Kolchin \cite{kolchin}
(having been announced in \cite{kolchin92}); in these results $\alpha^*_r$
is characterized
 by the fact that
 the expected number of non-trivial null vectors
tends to $0$ ($\infty$) when $\alpha < \alpha^*_r$ ($\alpha > \alpha^*_r$),
but the proofs show that the growth is in fact exponential for
 $\alpha > \alpha^*_r$.
Calkin \cite{calkin2} and Cooper \cite{cooper1} also study $\alpha^*_r$, $r \geq 3$, and in particular Calkin \cite{calkin2} studies
  $\alpha^*_r$ as $r \to \infty$; both \cite{cooper1} and \cite{calkin2} work in the case $W_n = r \wedge n$.
Note that Cooper's \cite{cooper1} expression of the matrix problem
 is transposed compared to ours.
The first part of our Theorem \ref{thm0} represents
a slight generalization  of the results just mentioned above because
it allows for any class of sequences $W_n$ provided
 $W_n \to r$ in probability.
The case of finite fields of order $q \geq 3$ has also been studied: see for instance \cite{kolchin2,cooper1,calkin1}.

In these previous investigations, the
analytic description of the threshold $\alpha^*_r$ varies.
Calkin \cite[\S 4]{calkin2} gives the same description of $\alpha^*_r$ as
our equation (\ref{Fdef}).
Systems of nonlinear equations for computing $\alpha^*_r$
have been proposed in \cite{bkk,kolchin,cooper1}; these descriptions can be shown
to be consistent with ours. Specifically, with $F_{r,\alpha} (\gamma)$ as given at (\ref{Frdef}),
one may characterize $\alpha_r^*$ by the two equations $F_{r,\alpha} (\gamma ) = 0$
and $\frac{\ud}{\ud \gamma} F_{r,\alpha} (\gamma ) =0$. On the substitution
$\lambda = \frac{1}{2} \log ( \frac{1-\gamma}{\gamma} )$,  the first of these equations
becomes, after some calculations along the lines of those in the proof of Proposition
\ref{multi} below,
\begin{equation}
\label{old1}
( 1 + (\tanh \lambda )^r )^\alpha \re^{-\lambda \tanh \lambda} \cosh \lambda = 1 .\end{equation}
The second equation, involving the vanishing of the derivative given at (\ref{Fdiff})
 gives, after the same substitution for $\lambda$,
\begin{equation}
\label{old2}
r \alpha = (1 + (\tanh \lambda)^{-r} ) \lambda \tanh \lambda .\end{equation}
The system of equations (\ref{old1}) and (\ref{old2}) is the same
as that given by Cooper \cite[p.\ 269]{cooper1}, and, after some manipulation, is
seen to coincide
also with that given by Balakin {\em et al.} \cite[p.\ 564]{bkk} and Kolchin \cite[p.\ 139]{kolchin}.
Despite this agreement, there are some small discrepancies in the numerical evaluations
for $\alpha_r^*$ in \cite{bkk,kolchin,calkin2,cooper1}, which can presumably be put down to numerical inaccuracies.

A similar tabulation to our
tabulation of $\ubar \alpha_r$ is given by Cooper \cite[pp.\ 370--371]{cooper2}, who also gives
an equivalent  analytic description of $\ubar \alpha_r$ to  our (\ref{MP3});
see also
Dietzfelbinger {\em et al.} \cite{DGMMPR} which we discuss further
in the next subsection. We note also that $\alpha_r^\sharp$ has received considerable
attention
in its own right:
see e.g.\ \cite{ikkm} for its role in random XORSAT.

\subsubsection{Between the two thresholds}
\label{between}

The following problem arises in the XORSAT literature.
Let $r \in \N$ with $r \geq 3$.
Let $M$ be our $m \times n$ matrix, with
 $m/n \to \alpha >0$, and suppose $W_n = r$ a.s.\ for all $n \geq r$.
Let $N$ denote the number of column vectors $x \in \{0,1\}^n$ such that
$M  x \equiv \omega$, where $\omega \in \{0,1\}^m$ is chosen
  uniformly at random
 (independent of $M$). Thus $N$ is a random variable.

Dubois and Mandler \cite{dubois} show for $r=3$,
and Dietzfelbinger {\em et al.} \cite{DGMMPR}
extend to general $r \in \N$ with $r \geq 3$ (also
providing a more detailed proof)
the following result (see  \cite[Theorem 3.1]{dubois} and
\cite[Theorem 1]{DGMMPR}, and
 also \cite{ikkm}): there is a constant $\tilde{\alpha}_r >0$ such that
 provided $\alpha < \tilde{\alpha}_r$, $\Pr [ N > 0 ] \to 1$
as $n \to \infty$.

The proof of this in \cite{dubois} is based on
 a second moment calculation. The analytical definition of
 $\tilde{\alpha}_r$
in \cite{dubois,DGMMPR} is not obviously the same as our
definition of $\ubar \alpha_r$,
but the definition in terms of cores (see \cite[Proposition 3 and
equation (4)]{DGMMPR}) seems to match our definition of
$\ubar{\alpha}_r$, and
the numerical values in \cite{DGMMPR} are consistent with our
$\ubar{\alpha}_r$.

If we accept that $\tilde{\alpha}_r = \ubar{\alpha}_r$,
this result implies that if $\alpha < \ubar \alpha_r$
 there is, for $n$ large enough,
 no non-zero left null vector for $M$, as follows.
Suppose that a non-zero $y$ satisfies $y \cdot M = \0$. Then $N >0$ implies
$y \cdot \omega = 0$. So $\Pr [ y \cdot \omega = 0 ] \geq \Pr[ N > 0] \to 1$,
which contradicts the easy observation that $\Pr[ y \cdot \omega = 0 ] = 1/2$
for non-zero $y$.

We may then deduce that in the case with $W_n = n \wedge r$,
our Theorem \ref{thm3} may be strengthened to
$n^{-1} T_n \to \ubar \alpha_r$ in probability.
This implies that for $\alpha$ in the interval $(\alpha^*_r, \ubar{\alpha}_r)$,
a form of substantialism occurs;
existence of any left null vector is unlikely, but if there is one, there are lots of them.

\subsection{Results for other random matrix models}
\label{sec:other}

\subsubsection{The case of fixed weight vectors with $r=1$ or $r=2$}

The classical cases of the constant weight model $W_n = n \wedge r$ in which $r \in \{ 1,2\}$
exhibit  different behaviour from the case $r \geq 3$.
  Recall the definition of $T_n$ from (\ref{Tdef}).

\begin{proposition}
\label{prop3}
(i)  For $r=1$,  for any $z>0$,
$\lim_{n\to \infty } \Pr  [T_n>z n^{1/2} ] =
\exp \{ - z^2 / 2 \}$.

(ii)  For $r=2$,  for any $z\in (0,1)$,
\begin{equation}
\label{rgform}
\lim_{n\to \infty } \Pr  [ T_n> z n/2 ]
= (1-z)^{1/2} \exp \left\{ \frac{z}{2} + \frac{z^2}{4} \right\}.
\end{equation}
\end{proposition}
 \begin{proof}
 In the case $r=1$, $X_1, X_2, \ldots , X_m$ correspond
 to $m$ repetitions of the experiment of placing a ball uniformly at random
in one of $n$ urns. Then $T_n$ is the first trial at which
a ball is placed in an urn which is already occupied, and its law is given by solution
to the birthday problem \cite[p.\ 33]{feller1}:
\[\log  \Pr  [T_n>m]
 = \sum _{j=1}^{m-1}  \log \left(1 - \frac{j}{n}\right)
 = -\frac{m (m-1)}{2 n} + O ( m^3/n^2 ),\]
provided $m = o(n)$.
Take $m = z n^{1/2}$ to obtain the result.

In the case $r=2$, $T_n$ is the same as the number of edges which have
to be added in the Erd{\H o}s--R\'enyi random graph process in order
for the first cycle to appear.
The formula (\ref{rgform}) has been derived
by Janson \cite[Theorem 8.1]{janson}. \end{proof}

\subsubsection{Uniform non-zero random vectors over a finite field}

For the sake of comparison with sparse matrix phenomena, the case
where $X$ is selected uniformly at random from
$\{0,1\}^n \setminus \{\0\}$
 is worthy of mention.

In fact we give a result for
the more general finite field GF$[q]$ for arbitrary prime power  $q$.
  Let $E_q^n := \{0,\ldots,q-1\}^n$.
  Suppose (for this section only)
 that $X_1, X_2, \ldots$
are i.i.d.\ random vectors
that are uniformly distributed over
 $E_q^n \setminus \{ \0\}$, and
let $M (n,m)$ be the $m \times n$ matrix
over GF$[q]$ with rows $X_1,\ldots,X_m$.
 We write $\Pr_{*}$ for probability associated with this model.
Define $T_n$ by (\ref{Tdef}) (with addition mod $q$).
The following elementary result
proves that $n+1-T_n$ is likely to be a small integer,
uniformly in $n$.
 Write $\ZP := \{ 0,1,2,\ldots\}$.

\begin{proposition}
\label{prop1}
Suppose that $X_1, X_2, \ldots$ are independent and uniformly
distributed on $E_q^n \setminus \{ \0 \}$.
As $n \to \infty$, $n+1 - T_n$ converges to a
 distribution on $\ZP$: for any $r \in \ZP$,
\begin{equation}
\label{Tlim}
 \Pr_* [ T_n > n +1 -r ] \to \prod_{j=r}^\infty (1 - q^{-j} ) =: p_r \in [0,1] ,
\end{equation}
where $p_0 = 0$ and $p_r \sim \exp \{ - q^{1-r} /(q-1) \}$
as $r \to \infty$.
Moreover, the following lower bounds apply: for $n \in \N$ and $1 \leq r \leq n$,
\[ \Pr_* [ T_n > n+1 -r ] \geq \begin{cases}
\exp \{ - q^{1-r} \} & \textrm{if} ~ q \geq 3 \\
\exp \{ - \frac{4}{3} 2^{1-r} \} & \textrm{if} ~ q = 2
\end{cases}
. \]
\end{proposition}

\begin{rmk}
 Limit results of the form of
(\ref{Tlim}) are classical, and apparently
date back at least to an 1895 paper of
Landsberg (see \cite[p.\ 69]{levitskaya}).
The $q=2$ case of (\ref{Tlim})
corresponds to the $T=n-s$, $m+s=0$ case of
\cite[Theorem 3.2.1, p.\ 126]{kolchinbook},
but with a slightly different probabilistic model:
there the $X_i$ are uniform on the whole of $E_2^n$, not just
$E_2^n \setminus \{ \0 \}$; comparing the results shows that this
difference in the probability measures used is negligible in the limit.\end{rmk}

\begin{proof}[Proof of Proposition \ref{prop1}.]
 Let $A_k$ denote the event that
$\{X_1, X_2, \ldots , X_k\}$ is linearly independent.
If $A_k$ occurs, then the span of $X_1, X_2, \ldots , X_k$
is a subspace with $q^k-1$ non-zero elements.
Since $X_{k+1}$ is  (statistically)
 independent  of $X_1, X_2, \ldots , X_k$ and  uniform on $E_q^n \setminus \{\0\}$,
which has $q^n-1$ non-zero elements,
\[ \Pr_* \left[A_{k+1} \mid A_k\right] =
\frac{q^n-q^k}{q^n-1} ,\]
so that, since $\Pr_* [ A_1 ] =1$, for $k \in \N$,
\begin{equation}
\Pr_* [T_n>k] = \Pr_* [A_k] = \prod _{j=1}^{k-1} \Pr_* [A_{j+1} \mid A_j ]
= \prod _{j=1}^{k-1} \frac{q^n - q^j}{q^n - 1}
= (1 - q^{-n} )^{1-k} \prod_{j=1}^{k-1} ( 1 - q^{j-n} )
; \nonumber \end{equation}
with the usual
convention that an empty product is $1$.
Taking $k = n +1 -r$
we obtain
\begin{equation}
\label{pak}
 \Pr_* [ T_n > n+1 -r ] = (1 - q^{-n} )^{r - n}
\prod_{j=r}^{n-1} ( 1 - q^{-j})
\to \prod_{j=r}^\infty (1 - q^{-j} ) ,\end{equation}
on letting $n \to \infty$, establishing (\ref{Tlim}) with $p_r$ as stated there.
The asymptotic form given for $p_r$ follows from writing
$p_r = \exp \sum_{j=r}^\infty \log ( 1 - q^{-j})$ and applying Taylor's
theorem.

Moreover, it follows from   the fact that $(1-q^{-n})^{r-n} \geq 1$ provided $r \leq n$ that
the convergence in (\ref{pak}) is in fact monotone for $n \geq r$, i.e.,
\[ \log   \Pr_* [T_n > n+1 - r]
\geq \sum _{j=r}^{n-1} \log \left( 1 - q^{-j}\right)
\geq \sum_{j=r}^\infty \log \left( 1 - q^{-j} \right) .\]
With $f(x) = \log (1 -x ) +x +x^2$,
we have $f(0)=0$ and $f'(x) = \frac{x(1-2x)}{1-x} \geq 0$
for $|x| \leq 1/2$,
so $\log (1-x) \geq -x -x^2$ for all $x$ with $|x| \leq 1/2$.
Thus for $q \geq 2$ and $r \geq 1$,
\[ \log   \Pr_* [ T_n > n+1 - r]  
\geq - \sum_{j=r}^\infty q^{-j}
-  \sum_{j=r}^\infty q^{-2j}
= - \frac{q^{1-r}}{q-1} \left( 1 + \frac{  q^{1-r}}{q +1} \right) .\]
The stated lower bounds follow.
\end{proof}

\subsubsection{Random vectors of weight $O(\log n)$}
\label{sec:notsosparse}

There is a separate class of results, initiated by the classical work of Kovalenko \cite{kov} and Balakin
\cite{balakin}, in which   the weight is
random with a law which changes as $n$ increases;
in the classical case it is $\Bin (n, (a+\log n)/n)$
 for  $a \in \R$. Much more on this model
is given in \cite{kolchinbook, kl,cooper3, levitskaya}, for example.

\section{Multinomial parities and random allocations}
\label{sec:occupancy}

\subsection{Overview and terminology}

In this section we work towards proving Theorem \ref{allothm}
and Proposition \ref{multi}.
Our null vector problem can be naturally formulated in terms
of classical occupancy problems of random allocations of balls into urns.

We shall use the following terminology.
Suppose $W$ is a random variable taking values
in $\ZP$, and $k \in \N$, and $p,p_1,p_2,\ldots,p_k$
are  numbers in $[0,1]$ such that $\sum_{i=1}^k p_i =1$.
(In most of the rest of the paper we assume $W \geq 1$, but for this section
we can allow $W$ to take value $0$.)
Let us say the random variable $X$ has the
${\rm Bin}(W,p)$ distribution if
 for each $n \in \ZP$
 the conditional distribution of $X$, given that $W=n$,
is binomial with parameters $(n,p)$. Let us say
that a random vector $(Z_1,\ldots,Z_k)$ has
the multinomial $(W;p_1,\ldots,p_k)$
distribution if
 for each $n \in \ZP$
 the conditional distribution of $(Z_1,\ldots,Z_k)$, given that $W=n$,
is multinomial with parameters $(n;p_1,\ldots,p_k)$.

Recall from Section \ref{sec:matrixlaw}
that we assume $W_n$ (having the distribution of row weights for
our matrix with $n$ columns) is chosen
to converge in  distribution to a limiting
random variable $W$. An important special case
is the so-called
{\em binomial} model.
In the binomial scheme
 take $W_n = W_n^{\rm bin}$ to
be distributed as the number of odd components
in a multinomial $(W ; n^{-1} , \ldots, n^{-1} )$ random vector.
By Lemma \ref{wnconv} below,
 $W^{\rm bin}_n \tod W$ as $n \to \infty$,
so this is indeed a special case.

We write $\Prb$
for 
probability associated with the binomial
allocation scheme.
For the general model we write $\Prn$ as
before.

\subsection{Exact formulae for the allocation problem}
\label{sec:exact}

In this subsection $n$ is  fixed.
Let $X_{ij}$ denote the $j$th component of $X_i$.
Define the column sums  $Y_j$ and partial row sums  $S_{i,J}$
 of the matrix $(X_{ij})$ as follows
(in this case the addition does not need
to be mod 2):
\begin{align*}
Y_j      :=  \sum_{i=1}^m X_{ij}, ~ j \in [n]; ~~~\textrm{and}~~~
S_{i,J}  := \sum_{j \in J} X_{ij},  ~ J \subseteq [n] .
\end{align*}

Recall from (\ref{Adef}) that $A(n,m)$ denotes
the event that $\1$ is a null (row) vector for $M$.

\begin{lemma}
\label{lem:nullp}
In the binomial allocation scheme,
 we have the exact formula
\begin{equation}
\label{pmulti}
 \Prb [ A(n,m) ] = 2^{-n} \sum_{j=0}^n \binom{n}{j}
\left( \rho
(1 - (2j/n ) ) \right)^m .
\end{equation}
In the general allocation scheme,
\begin{align}
\label{pAnm}
 \Prn [ A(n,m) ] & = 2^{-n} \sum_{J \subseteq [n]}
\left( \Expn [ (-1)^{S_{1,J} } ] \right)^m
\\
\label{phyp}
& = 2^{-n} \sum_{j=0}^n  \binom{n}{j} (2 p^{(n)}_{j} - 1)^m,
\end{align}
where $p^{(n)}_j := \sum_{r =0}^n p_{j,r} \Pr [ W_n =r]$ and $p_{j,r}$ is given by
\begin{align}
\label{pjk}
p_{j,r} =
\frac{1}{\binom{n}{j}} \left(
\binom{n-r}{j} + \binom{r}{2} \binom{n-r}{j-2} +
\binom{r}{4} \binom{n-r}{j-4} + \cdots \right).
\end{align}
\end{lemma}
\begin{proof}
Event $A(n,m)$ occurs if and only if all the $Y_j$ are even, so
\begin{align*}
\Pr[A (n,m) ]  =
\Exp \prod_{j=1}^n \left( \frac{1 + (-1)^{Y_j} }{2} \right)
  = 2^{-n}
\sum_{J \subseteq [n] } \Exp \left[ (-1)^{\sum_{j \in J} Y_j} \right] ,
\end{align*}
where the latter sum is over subsets $J$ of $[n]$, including the empty set.
Since $\sum_{j \in J} Y_j = \sum_{i=1}^m S_{i,J}$ and $S_{1,J}, S_{2,J}, \ldots$ are i.i.d.,
 (\ref{pAnm}) follows.

Consider  the binomial allocation scheme.
In the binomial model,
$$
S_{1,J} =  \sum_{j \in J} X_{1j} \equiv \sum_{j \in J} Z_j  ~ ({\rm mod ~ 2}),
$$
where $(Z_1,\ldots,Z_n) $ has
 a multinomial $(W ; n^{-1}, \ldots, n^{-1})$ distribution
so that $\sum_{j \in J} Z_j$ has
 a $\Bin (W , |J|/n)$ distribution.
Recalling that if $\xi \sim {\rm Bin}(n,p)$ then
$\Exp [s^\xi]=(sp + (1-p))^n$,
we then obtain (\ref{pmulti}) from (\ref{pAnm}).

In the general
 scheme,
conditional on
$\sum_{j=1}^n X_{1j}=r$, the distribution of
$S_{1,J}$ is hypergeometric with parameters
$(n;|J|,r)$. We do not use the
 generating
function (see e.g.\ \cite[p.\ 17]{moran})
explicitly, but
apply the hypergeometric probability mass function
directly to
(\ref{pAnm}).

Let $H \subseteq [n]$
denote the set of values
of $j$ for which $X_{1 j} = 1$.
For $ r \in [n]$, we
write $\Exp_r$ for expectation in
the case where
 $\Pr [ W_n=r]=1$.
Instead of fixing
$J \subseteq [n]$
and choosing $H$ as a uniform random $r$-subset,
we obtain an exact formula
for $\Exp_r [ (-1)^{S_{1,J}} ]$ by fixing $j = |J|$
and a $r$-subset $H$, and selecting
$J$ uniformly from the $j$-subsets of $[n]$.
The probability $p_{j,r}$ that $S_{1,J}:= | H \cap J|$
is even is given by summing probabilities for
$|H \cap J| \in \{0, 2, 4, \ldots\}$,
giving the expression in (\ref{pjk}).
It follows that $\Exp_r [ (-1)^{S_{1,J}} ] = 2 p_{j,r} - 1$,
and hence
$$
\Expn [ (-1)^{S_{1,J}}] =
\sum_{r=1}^n
 (2 p_{j,r} -1)
 \Pr [ W_n=r] =
 2 p_{j}^{(n)} -1.
$$
Substitution of this into (\ref{pAnm}) gives (\ref{phyp}).
\end{proof}

\subsection{Asymptotics in the binomial model}
\label{sec:binlim}

The remaining parts of Section \ref{sec:occupancy}  are concerned with asymptotic analysis of the quantities in Lemma \ref{lem:nullp}. First we state a
result that will enable us to work primarily with even $m$, which has technical advantages.

\begin{lemma}
\label{oddlem}
Suppose that $W_n \tod W$ and $\Pr [ W \in 2\Z] >0$.  
Then for any $m > 3$,
\[ \log \Prn [ A (n ,m-3) ] + O (\log n) \leq \log \Prn [ A (n, m) ] \leq \log \Prn [ A (n ,m+3) ] + O (\log n) .\]
\end{lemma}
\begin{proof}
The fact that $\Pr[ W \in 2\Z] > 0$ and $W_n \tod W$
implies that there exist $\eps>0$ and $r \in 2\Z$ such that $\Pr[ W_n = r]>\eps$ for all $n$ large enough.
For any  $m  > 3$, suppose that $A(n,m -3)$
occurs. Then $A(n,m )$ will occur if the $3$ additional rows themselves constitute a hypercycle.
With probability at least $\eps^3$, these new rows each have $r$ units, and given this,
there is a probability at least $n^{-2r}$, say, that these units form a hypercycle.
In other words, $\log \Prn [ A(n,m ) ] \geq \log \Prn[ A(n,m -3) ] + O (\log n)$.
Applying this inequality twice,
once with  $m  +3$ in place of $m$,
gives the result.
\end{proof}

Recall that in general we assume $W_n \tod W$. Next we give an elementary lemma that confirms the binomial
model's place in this framework.

\begin{lemma}
\label{wnconv}
For $W$ a $\ZP$-valued random variable,   let $W^{\rm bin}_n$ be the number
of odd components in a multinomial $(W  ; n^{-1}, \ldots, n^{-1})$ random vector.
 Then $W^{\rm bin}_n \tod W$ as $n \to \infty$.
 \end{lemma}
\begin{proof}
Let $W$  and $W^{\rm bin}_n$ be coupled in the natural way.
Then for each $k \in \N$, by the union bound
$$
\Pr[ W_n^{\rm bin} \neq W \mid W=k] \leq \binom{k}{2} \frac{1}{n},
$$
which tends to zero as $n \to \infty$, and the result
follows easily from this.
\end{proof}

We will prove Theorem \ref{allothm} (in Section \ref{sec:alloproof})
by first showing that (\ref{allolim}) holds in the binomial setting,
using (\ref{pmulti})
and the Stirling approximation for the binomial coefficients as discussed in
Section \ref{sec:appendix}. Then we will extend this to the general setting using an approximation argument
described in Section \ref{sec:approx1}.
We start by proving a slightly more general statement than (\ref{allolim})  in the binomial case, which we will also need later in the proof of Theorem \ref{thm1}.

\begin{lemma}
\label{binlem}
Recall the definition of $R_\rho$ from (\ref{Rrhodef}). Then $R_\rho(\alpha)$ is continuous and nondecreasing as a function of $\alpha$.
Suppose that  
either (i) $m_n \in 2\Z$ for all $n$; or (ii) $\Pr [ W \in 2 \Z ] > 0$.
Suppose that there exist $\alpha_1, \alpha_2$ with $0 < \alpha_1 < \alpha_2 < \infty$ such that,
for all $n$ sufficiently large, $\alpha_1 < m_n /n < \alpha_2$. Then,
\begin{align}
\label{infsup} \limsup_{n \to \infty} n^{-1} \log \Prb [ A (n ,m_n) ] & \leq - R_\rho (\alpha_1) ; \nonumber\\
\liminf_{n \to \infty} n^{-1} \log \Prb [ A (n ,m_n) ] & \geq - R_\rho (\alpha_2) .\end{align}
In particular, if $m_n / n \to \alpha >0$, then
\begin{equation}
\label{binlim2}
 \lim_{n \to \infty} n^{-1} \log \Prb [ A (n ,m_n) ] = - R_\rho (\alpha) .\end{equation}
\end{lemma}
\begin{proof}
Suppose that $m_n/ n \in (\alpha_1, \alpha_2)$. First assume that the $m_n$ are even.
By (\ref{pmulti}) and Lemma \ref{pgf}(iii),
\begin{align*}
\Prb [A(n,m) ] & \leq (n+1) 2^{-n} \max_{0 \leq j \leq n}  \binom{n}{j} | \rho ( 1 -(2j/n) ) |^{m} \\
& \leq (n+1) 2^{-n} \sup_{\gamma \in [0,1/2]}  \binom{n}{\gamma n} ( \rho( 1 -2\gamma ) )^{m},\end{align*}
where we set $\binom{n}{x} = 0$ if $x$ is a not an integer in $\{0,1,\ldots, n\}$.
Using the upper bound on binomial coefficients  from the first inequality in (\ref{chooseub2}), we get
\[ \Prb [A(n,m_n) ] \leq (n+1) \sup_{\gamma \in [0,1/2]}  \left( 2 \gamma^\gamma (1-\gamma)^{1-\gamma} \right)^{-n}
 ( \rho( 1 -2\gamma ) )^{m_n} .\]
 By monotonicity, we then obtain
 \begin{equation}
 \label{eq45}
  n^{-1} \log \Prb [A(n,m_n) ] \leq
n^{-1} \log(n+1)
+ \log \sup_{\gamma \in [0,1/2]} g_{m_n/n} (\gamma),\end{equation}
where we have set
\[ g_\alpha ( \gamma) :=  \frac{ (
\rho ( 1-2\gamma ) )^\alpha }{
2 \gamma^\gamma (1-\gamma)^{1-\gamma}} ;\]
so with $R_\rho$ as defined at (\ref{Rrhodef}), $R_\rho (\alpha) = - \log \sup_{\gamma \in [0,1/2]} g_\alpha (\gamma)$.
Note that $g_\alpha(\gamma)$ is continuous in  $\gamma$ for $\alpha \geq 0$ and $\gamma \in [0,1/2]$,
and $g_\alpha (\gamma)$ is nonincreasing in $\alpha$;
this monotonicity implies, by Dini's theorem, that if $\alpha' \to \alpha$ monotonically then $g_{\alpha'}$ converges  uniformly to $g_\alpha$ on the compact
interval $[0,1/2]$. It follows that $\alpha \mapsto \sup_{\gamma \in [0,1/2]} g_\alpha (\gamma)$
is continuous as a function of $\alpha >0$, and is also nonincreasing in $\alpha$.
In particular, this shows that $R_\rho(\alpha)$ is continuous   and nondecreasing in $\alpha$,
as claimed in the lemma.
Moreover, we obtain from (\ref{eq45}) and the fact that $m_n /n > \alpha_1$  that
 \[ \limsup_{n \to \infty}  n^{-1} \log \Prb [A(n,m_n) ] \leq \log \sup_{\gamma \in [0,1/2]} g_{\alpha_1} (\gamma) , \]
which gives the first inequality in  
 (\ref{infsup}). 

For the second inequality,
we have from (\ref{pmulti}) and (\ref{chooselb}) that for any integer $i_n \leq n/2$,
\begin{align*}
\Prb [A (n,m_n) ]
  \geq 2^{-n} \binom{n}{i_n} ( \rho( 1 - (2i_n /n)) )^{m_n}
  \geq \re^{-1/6} \left( \frac{n}{2\pi i_n (n-i_n)} \right)^{1/2} \!\!
 ( g_{m_n/n} (i_n /n) )^{n}  ,\end{align*}
 using the fact that $m_n$ is even.
 Then, since $m_n /n < \alpha_2$,
 \[ \Prb [A (n,m_n) ]
  \geq  \re^{-1/6} \left( \frac{n}{2\pi i_n (n-i_n)} \right)^{1/2}
 ( g_{\alpha_2} (i_n /n) )^{n} .\]
Now use continuity of $g_\alpha$
to choose a
 sequence of integers $i_n \leq n/2$, $n \in \N$,
such that
$g_{\alpha_2} ( i_n /n) \to \sup_{\gamma \in [0,1/2]} g_{\alpha_2} (\gamma)$,
with 
$i_n \to \infty$ and $n-i_n \to \infty$ as $n \to \infty$.
The lower bound in (\ref{infsup}) follows, for   $m_n$ even.

The results in (\ref{infsup}) extend to the case of odd $m_n$ with $\Pr [ W \in 2\Z ] >0$ by Lemma \ref{oddlem}, which is applicable here by Lemma \ref{wnconv}.
Finally, (\ref{binlim2}) follows from (\ref{infsup}) on taking $\alpha_1 = \alpha - \eps$ and $\alpha_2 = \alpha + \eps$,
for arbitrary $\eps>0$, and using the continuity of $R_\rho$.
\end{proof}

\subsection{Approximation by the binomial model}
\label{sec:approx1}

The exact formula (\ref{pmulti}) is simpler
to work with than the more complicated exact formula (\ref{phyp}),
but intuition suggests that the asymptotics of any
of the models in the class with $W_n \tod W$ should be
similar. In this section we quantify this intuition.

\begin{lemma}
\label{approxlem1}
Suppose that $W_n \tod W$ and either (i) $m_n \in 2\Z$ for all $n$; or (ii) $\Pr [ W \in 2 \Z ] > 0$.
Suppose that there exist $\alpha_1, \alpha_2$ with $0 < \alpha_1 < \alpha_2 < \infty$ and  $n_0 \in \N$
such that $\alpha_1 < m_n /n < \alpha_2$ for all $n \geq n_0$.
Then, uniformly over sequences $m_n$ satisfying the given conditions,
\begin{align}
\lim_{n \to \infty} n^{-1} \left|  \log  \Prn
[ A(n, m_n) ] -
\log \Prb  [ A(n, m_n)] \right| = 0.
\label{approxA}
\end{align}
In particular, if $m_n /n \to \alpha >0$,
\[ \lim_{n \to \infty} n^{-1}   \log  \Prn
[ A(n, m_n) ]  = \lim_{n \to \infty} n^{-1} \log \Prb  [ A(n, m_n)]  . \]
\end{lemma}
\begin{proof}
Denote the weight of row 1 by $W_n$ in the general allocation scheme,
and by 
  $W^{\rm bin}_n$  in the binomial scheme.
Then
$W_n \tod W$ and $W^{\rm bin}_n \tod W$, and we can work in a probability space
where $\Pr [ W_{n } \neq W^{\rm bin}_{n } ] \to 0$.
 As in Section \ref{sec:exact}, set $S_{1,J} = \sum_{j \in J} X_{1j}$;
 on $\{W_{n } = W^{\rm bin}_{n }\}$, the (conditional) law
 of $S_{1,J}$ is the same as in the binomial model.
 Thus
\begin{equation}
\label{delta}
\sup_{J \subseteq [n]}
 \left| \Expn [ (-1)^{S_{1,J}} ] - \Expb [ (-1)^{S_{1,J}} ] \right|
\leq 2 \Pr  [ W_{n} \neq W^{\rm bin}_{n} ] \to 0.\end{equation}

First suppose that the $m_n$ are even.
By (\ref{pAnm}) with (\ref{pmulti}) and (\ref{delta}),
there exists a triangular array of numbers
$(\delta_{j,n}, j \in [n] \cup\{0\}, n \in \N)$
satisfying $\max_{0 \leq j \leq n} | \delta_{j,n}| \to 0$
as $n \to \infty$, and
\begin{align}
\label{pdiff}
 \Prn  [A (n,m_n) ] = 2^{-n} \sum_{j=0}^n \binom{n}{j}
 \left( \rho (1 - (2j/n)) + \delta_{j,n}
\right)^{m_n}  .
\end{align}

Let $\eps >0$ and choose $K>1$ large enough so that
$\log(1-K^{-1}) > - \eps / \alpha_2$ and
$\log(1+K^{-1}) <  \eps / \alpha_2$.
Then choose $\delta >0$ such that
$(K+1) \delta < \exp \{- 1/(\alpha_1 \eps)\}$.
Finally assume $n$ is large enough so that
$\sup_{j \in [ n] \cup \{0\}} |\delta_{j,n}| \leq \delta$
and $\alpha_1 < (m_n/n) <   \alpha_2$.

We split the sum in (\ref{pdiff})
into two parts, depending on
the size of $\rho (1 - (2j/n))$.
First suppose that $|\rho (1 - (2j/n))| \leq K \delta$.
In this case
\begin{equation}
\label{smallj}
\left| \left( \rho (1 - (2j/n)) +\delta_{j,n}
\right)^{m_n} \right|
\leq ((K+1)\delta )^{m_n}
\leq
\exp \{- m_n/(\alpha_1 \eps)\} \leq \exp \{ -n/\eps\},
\end{equation}
and similarly,
\begin{equation}
\label{smallj2}
 \left( \rho (1 - (2j/n)) \right)^{m_n}
\leq \exp \{ -n/\eps \}.
\end{equation}
It follows from (\ref{pdiff}) and (\ref{smallj}) that
\begin{align}
\label{diff1}
 \Prn  [ A(n,m_n) ] & =
 2^{-n} \sum_{j : | \rho ( 1 - (2j/n)) | > K \delta } \binom{n}{j}
 \left( \rho (1 - (2j/n)) +\delta_{j,n}
\right)^{m_n}  \nonumber\\
& ~~~ +
O \left( \exp \{-n/\eps
\} \right) .
\end{align}
Now suppose that $| \rho (1 - (2j/n)) | > K \delta$.
In this case
\begin{align*}
 \left( \rho (1 - (2j/n)) +\delta_{j,n}
\right)^{m_n}
& = \left( \rho (1- (2j/n)) \right)^{m_n}
\left(1 +  \theta_{j,n} K^{-1} \right)^{m_n} ,
\end{align*}
where $|\theta_{j,n}| \leq 1$.
By the choice of $K$,
$ \re^{-\eps/\alpha_2} < 1 + \theta_{j,n} K^{-1} < \re^{\eps/\alpha_2}$,
and hence
$$
\exp \{ -\eps n \} < (1 + \theta_{j,n} K^{-1})^{m_n} < \exp \{ \eps n \}.
$$
Therefore
\begin{align*}
\left( \rho (1 - (2j/n)) +\delta_{j,n} \right)^{m_n}
& = \left( \rho (1- (2j/n)) \right)^{m_n} \exp \{ \eps_{j,n} n \} ,
 \end{align*}
where $|\eps_{j,n}| < \eps$.
Hence
for the sum on the right-hand side of (\ref{diff1}),
there exists $\eps_n$ with $|\eps_n|< \eps$ such that
\begin{align*}
 ~~ &  2^{-n} \sum_{j :  | \rho ( 1 - (2j/n)) | > K \delta } \binom{n}{j}
 \left( \rho (1 - (2j/n)) +\delta_j \right)^{m_n} \\
 & = 2^{-n} \exp \{ \eps_n n \} \sum_{j : |\rho ( 1 - (2j/n)) | > K \delta } \binom{n}{j}
 \left( \rho (1 - (2j/n)) \right)^{m_n} ,
 \end{align*}
 using the assumption that $m_n$ is even so
all the terms in the sum are nonnegative.
 Then by (\ref{smallj2}) and a similar argument to (\ref{diff1}), the last displayed quantity
 is equal to
$\Prb [ A(n,m_n)] \exp \{ \eps_n n \} +   O \left( \exp \{ (\eps - \eps^{-1}) n   \} \right)$.
Combining this with (\ref{diff1}) we obtain
\begin{equation}
\label{eq0}
 \Prn  [ A(n,m_n) ] = \Prb  [ A(n,m_n) ] \exp \{ \eps_n n \} +  O \left( \exp \{ (\eps -  \eps^{-1} ) n   \} \right) ,\end{equation}
uniformly in $n$ (and $m_n$), the implicit constants depending on $\alpha_1$ and $\alpha_2$.

It follows from (\ref{eq0})
that
\[ \log  \Prn  [ A(n,m_n) ] = \log \Prb  [ A(n,m_n) ] + \eps_n n + \log \left( 1 + \frac{\Delta_n}{ \Prb  [ A(n,m_n) ] \exp \{ \eps_n n \} } \right) ,\]
where $\Delta_n = O \left( \exp \{ (\eps -  \eps^{-1} ) n   \} \right)$ is the final term in (\ref{eq0}).
By Lemma \ref{binlem}, we have that $\Prb  [ A(n,m_n) ] \geq \exp \{ - n R_\rho (\alpha_2) -\eps n \}$, for all $n$ large enough.
So we may take $\eps >0$ small enough so that
\[ \log \left( 1 + \frac{\Delta_n}{ \Prb  [ A(n,m_n) ] \exp \{ \eps_n n \} } \right)  = O ( \exp \{ - n \} ) ,\]
say. Hence
\[  n^{-1} \log  \Prn  [ A(n,m_n) ] = n^{-1} \log \Prb  [ A(n,m_n) ] + \eps_n   + o(1) .\]
Since $|\eps_n| \leq \eps$ and $\eps >0$ was arbitrary, (\ref{approxA}) follows
in the case of even $m_n$. In the other case, Lemma \ref{oddlem} yields the same conclusion.
The final statement in the lemma then follows from Lemma \ref{binlem}, which says that $\lim_{n\to\infty} n^{-1} \Prb [A(n,m_n)]$ exists
in $(0,\infty)$
when $m_n /n \to \alpha >0$.
\end{proof}

\subsection{Proofs of Theorem \ref{allothm} and Proposition \ref{multi}}
\label{sec:alloproof}

Now we can complete the proofs  of   Theorem \ref{allothm} and Proposition \ref{multi}.

\begin{proof}[Proof of Theorem \ref{allothm}.]
The theorem is now a consequence of Lemmas \ref{binlem} and \ref{approxlem1}.
\end{proof}

\begin{proof}[Proof of Proposition \ref{multi}.]
Set
\begin{equation}
\label{fdef} f_\alpha (\gamma ) :=  \log
\left(  \frac{(1-2 \gamma)^\alpha}{2 \gamma^{\gamma}
(1-\gamma)^{1-\gamma}}  \right)
.\end{equation}
In the case $\rho(s) = s$, it follows from Theorem \ref{allothm} that
$n^{-1} \log   \pi_n(m_n) \to \sup_{\gamma \in [0,1/2]} f_\alpha (\gamma)$.
 Proposition \ref{multi}  will follow once we prove that, setting $\alpha = \lambda \tanh \lambda$,
\begin{align}
\sup_{\gamma \in [0,1/2]}  f_\alpha (\gamma)
=  -( \lambda \tanh \lambda ) ( 1 - \log (\tanh \lambda ))
+ \log ( \cosh \lambda ).
\label{sup1}
\end{align}
Note that $f_\alpha (\gamma) \to -\infty$
as $\gamma \uparrow 1/2$.
 Differentiating (\ref{fdef}) gives,
for $\gamma \in (0,1/2)$,
\[ \frac{\ud }{\ud \gamma} f_\alpha (\gamma) = - \frac{2\alpha}{1-2\gamma} + \log \left( \frac{1-\gamma}{\gamma} \right) ,\]
which is zero at $\gamma_1 := \gamma_1 (\alpha) \in (0,1/2)$ defined implicitly
in terms of $\alpha$ by
\begin{equation}
\label{gamma1}
 \alpha = \frac{1}{2} (1 -2\gamma_1) \log \left( \frac{1-\gamma_1}{\gamma_1} \right).\end{equation}
 One can verify that
  for $\alpha >0$ (\ref{gamma1})  defines a unique stationary value
$\gamma_1 \in (0,1/2)$ which is a local maximum, since
 for $\gamma_1 \in (0,1/2)$ the right-hand side
of (\ref{gamma1}) is positive, continuous, and strictly decreasing as a function of $\gamma_1$,
vanishing at $\gamma_1 =1/2$; this local maximum is indeed the   maximum
of $f_\alpha (\gamma)$ for $\gamma \in [0,1/2]$ since $f'_\alpha (\gamma) \to \infty$ as $\gamma \downarrow 0$ (and also $f''_\alpha(\gamma_1) < 0$).

Setting $\lambda = \frac{1}{2} \log \left( \frac{1-\gamma_1}{\gamma_1} \right)$ we see
that
$\lambda \tanh \lambda = \alpha$ as given by (\ref{gamma1}),
since we get $\tanh \lambda = 1 -2\gamma_1$.
 To verify (\ref{sup1})
we need to express $f_\alpha (\gamma_1)$ in terms of $\lambda$ to get the
expression on the right-hand side of (\ref{sup1}). We have
\begin{align*} f_\alpha (\gamma_1) & = \alpha \log (1 -2\gamma_1)
- \log 2 - \gamma_1 \log \gamma_1 - (1-\gamma_1) \log (1-\gamma_1) \\
& = ( \lambda \tanh \lambda ) \log \tanh \lambda + \log \cosh \lambda
+ ((1/2) - \gamma_1 ) \log \gamma_1 + (\gamma_1 - (1/2)) \log (1-\gamma_1) ,\end{align*}
where we have used the fact that
 $\log \tanh \lambda = \log (1 -2\gamma_1)$
and $\log \cosh \lambda = -\frac{1}{2} \log ( 1 - \tanh^2 \lambda )
= - \log 2 - \frac{1}{2} \log \gamma_1 - \frac{1}{2} \log (1-\gamma_1)$.
Collecting the terms involving $\gamma_1$ in the last displayed equation,
we see that they simplify to $-\lambda \tanh \lambda = -\alpha$ as given
by (\ref{gamma1}), so
we verify (\ref{sup1}).
 \end{proof}

\subsection{Alternative proof of Proposition \ref{multi} via Poissonization}
\label{sec:poisson}

We give an alternative proof of Proposition \ref{multi}
based on a Poissonization device (as used
by Kolchin in his proof of Theorem 2 in \cite{kolchin})
and large deviations
arguments of a slightly different flavour from those in the proof above.
The proof in this section is direct, avoiding the general Theorem \ref{allothm},
but does use instead some relatively deep local limit theory.
The  following result can be found
for example in  \cite{kolchin}.

\begin{lemma}
\label{poissoneven}
Suppose that $Z_1,Z_2, \ldots$ are independent Poisson random variables with mean $\mu >0$.
Let $Z_1^E, Z_2^E,\ldots$ be i.i.d.,
where the law of $Z_1^E$ is the same as the conditional law of $Z_1$, given that $Z_1$ is even:
$\Pr [ Z_1^E = k ] = \Pr [ Z_1 = k \mid Z_1 \in 2 \Z ]$.
Then
\begin{equation}
\label{pi2}
\pi_n (m) =
\frac{\Pr \left[Z_1^E + \cdots  +Z_n^E = m\right]}
{\Pr \left[Z_1 + \cdots  +Z_n = m\right]}
\left( \frac{1 + \re^{-2 \mu }}{2} \right)^n.
\end{equation}
\end{lemma}
\begin{proof}
By the well-known relationship
between the Poisson and multinomial distributions (see e.g.\ \cite[p.\ 140]{kolchin} or \cite[p.\ 15]{kolchinbook}), and
Bayes' Theorem,
\begin{align*}
\pi_n(m) & = \Pr \left[Z_1\in 2 \Z, \ldots , Z_n\in 2 \Z \mid Z_1 + \cdots  +Z_n = m\right]
\\
 & = \frac{ \Pr [   Z_1 + \cdots + Z_n = m \mid Z_1 \in 2\Z, \ldots, Z_n \in 2\Z  ] }
{\Pr [ Z_1 + \cdots + Z_n = m] }  \cdot \Pr [ Z_1 \in 2\Z ]^n ,
\end{align*}
which with the expression for $\Pr [ Z_1 \in 2\Z ]$ in Lemma \ref{lem4}(ii) yields (\ref{pi2}).
\end{proof}

\begin{lemma}
\label{local}
Let $X_1, X_2,\ldots$ be an i.i.d.\ sequence of $\Z$-valued random
  variables, $S_n := X_1 + \cdots + X_n$, and $(x_n)_{n \in \N}$
   a sequence
  of even integers.  Suppose that
$\Exp [\re^{tX_1}] < \infty$ for some $t>0$,
$\Pr[X_1 =0 ] \wedge \Pr[X_1 =2] >0$,
  and  $x_n = n\Exp [ X_1] + o(n)$. Then
\[
\lim_{n \to \infty} n^{-1} \log \Pr [S_n = x_n ] = 0.
\]
\end{lemma}
\begin{proof}
Write $\mu = \Exp [X_1]$ and $\sigma^2 = \Var [X_1]$,  which is finite and positive
 by the conditions in the lemma.
Write $Y_i = X_i - n\mu$ and $y_n = n^{-1/2} \sigma^{-1} (x_n - n\mu)$, so
$\Exp [ Y_i ] =0$, $\Exp [ Y_i^2] =\sigma^2$, and $y_n = o(n^{1/2})$.
If $Z_n = n^{-1/2} \sigma^{-1} \sum_{i=1}^n Y_i$,
Richter's
 local central limit theorem \cite[Chapter 7, \S\S 1 and 4]{il} tells us that
\[ \Pr [S_n = x_n] = \Pr [ Z_n = y_n  ] = \Theta \left( n^{-1/2} \exp \left\{ - \frac{1}{2} y_n^2 \left(1 + O (n^{-1/2} y_n ) \right) \right\} \right) ,\]
which is $\exp \{ o (n) \}$, since $y_n = o(n^{-1/2})$.
\end{proof}

\begin{proof}[Second proof of Proposition \ref{multi}.]
From Lemma \ref{poissoneven},  $n^{-1} \log \pi_n (m)$ can be expressed as
\begin{align}
\label{eq50}
 n^{-1} \log \Pr [Z_1^E + \cdots + Z_n^E = m]
 - n^{-1} \log \Pr [Z_1 + \cdots + Z_n = m]
+ \log \cosh \mu - \mu .\end{align}
By assumption, $m = m_n$ is such that $m_n/n \to \alpha >0$.
The proof proceeds by choosing $\mu$ so that
   $\Exp[Z_1^E] = \alpha$;
then the first term   in (\ref{eq50})
vanishes in the limit
by Lemma \ref{local}, and the proof of the theorem
 then reduces to evaluating the other logarithmic rate.

We choose $\mu$ so that  $\Exp [Z_1^E] = \alpha$; by Lemma
\ref{lem4}(ii) this means $\mu \tanh \mu = \alpha$ so that $\mu = \lambda$.
Since $Z_1 + \cdots + Z_n$ is Poisson with mean $n \lambda$,
\begin{align*}
\lim_{n\to\infty} n^{-1} \log \Pr [Z_1 + \cdots + Z_n = m_n ] &= \ldp \frac{(n\la)^{m_n} \re^{-n \la}}{m_n !}\\
&= -\la + \lim_{n \to \infty} \frac{m_n}{n}
\left[\log n\la - \frac{1}{m_n} \log m_n ! \right] .\end{align*}
Stirling's formula
implies that $n^{-1} \log n! = \log ( n /\re) + o(1)$,
so that
\begin{align}
\label{eq51} \lim_{n\to\infty} n^{-1} \log  \Pr [Z_1 + \cdots + Z_n = m_n ]
 & = -\la + \lim_{n \to \infty} \frac{m_n}{n}
\left[\log \left( \frac{n\la \re}{m_n}\right) + o(1)\right] \nonumber
\\
& =-\la + \alpha \left(\log \la - \log \alpha + 1\right)
.\end{align}
Combining (\ref{eq51}) with the $\mu = \lambda$ case of (\ref{eq50})
 we complete  the proof.
\end{proof}

\section{Proofs of main results}
\label{sec:proofs}

\subsection{Exact formula for the expected number of null vectors}

Let $\NN ( n ,m ; \ell )$ denote the number of left null vectors
 of weight $\ell$,
so that
\begin{equation}
\label{wtsum}
 \NN ( n, m) =  \sum_{\ell =0 }^m \NN ( n ,m ; \ell )
.\end{equation}
The value of $\NN(n,m)$
 is the number of collections
  of rows of $M(n,m)$
which sum to $\0$ (mod $2$), and for each set of $\ell$ rows the
probability that it sums to $\0$ is $\Prn [A (n,\ell) ]$. Hence
\begin{align}
\label{byweight}
\Expn
 [ \NN (n, m ; \ell) ] &  = \binom{m}{\ell} \Prn
 [ A (n,\ell) ] , ~\textrm{and} \\
\label{eq1}
\Expn
[\NN(n,m)]  & = \sum_{\ell=0}^m \binom{m}{\ell} \Prn
 [ A (n,\ell) ] .
\end{align}
We can thus express $\Expn [ \NN (n,m) ]$ using our exact formulae for $\Prn [ A (n,\ell) ]$
given in Lemma \ref{lem:nullp}. The proofs of our main results, Theorems \ref{thm0}, \ref{thm1}, and \ref{thm3},
will be based on an asymptotic analysis of (\ref{eq1}). As in the proof of Theorem \ref{allothm} (see Section \ref{sec:allo})
it is most convenient to work in the binomial model, for which $W^{\rm bin}_n$ is the number of odd components in a multinomial
$(W ; n^{-1},\ldots,n^{-1})$ vector. Thus a key step in the proof will be showing that, in the general case of $W_n \tod W$,
 the expression in (\ref{eq1})
can be well approximated by the binomial case. First, in the next section, we make some preliminary computations.

\subsection{Preliminaries}
\label{secprelim}

Before embarking on the main proof, we study the rate functions that will appear.
Define
\begin{equation}
\label{FFdef}
 F_{\rho,\alpha} (  \gamma ) := \log \left(
 \frac{( 1 + \rho ( 1 -2 \gamma ))^\alpha}{2
\gamma^\gamma (1-\gamma)^{1-\gamma}} \right) ,
\end{equation}
and recall from (\ref{alphastar2}) and (\ref{Fdef2}) that
$F_\rho(\alpha) = \sup_{\gamma \in [0,1/2]} F_{\rho,\alpha}(\gamma)$
and $\alpha_\rho^* = \inf\{\alpha \geq 0: F_\rho(\alpha) >0\}$.
Note that for $\gamma \in [0,1/2]$, $\rho(1-2\gamma) \geq 0$.
By continuity, $F_{\rho ,\alpha} (  \gamma)$ attains its supremum over $\gamma \in [0,1/2]$;
we denote by $\gamma_0 := \gamma_0 (\alpha) \in [0,1/2]$
the {\em smallest} point at which the supremum is attained.

We collect results on $F_\rho (\alpha)$ and $\alpha^*_\rho$ in the next lemma,
which will enable us to complete the proof of Proposition \ref{alphaprop}.

\begin{lemma}
\label{alphalem}
Suppose that $\Pr [ W = 0 ] =0$.
For any $\alpha \geq 0$,
$F_\rho(\alpha) \geq 0$, and $F_\rho$ is continuous and nondecreasing.  
 The threshold $\alpha_\rho^*$
enjoys the following properties.
\begin{itemize}
\item[(i)]
$\alpha^*_\rho \in [0,1]$, and $F_\rho(\alpha) = 0$ for $\alpha \leq \alpha^*_\rho$
but $F_\rho(\alpha) >0$ for $\alpha > \alpha^*_\rho$.
\item[(ii)] If $\alpha < \alpha^*_\rho$, then for any $\eps>0$, $\sup_{\gamma \in [0, (1/2)-\eps]} F_{\rho ,\alpha} (  \gamma) <0$.
\item[(iii)] If $\alpha > \alpha^*_\rho$, then $\gamma_0 (\alpha) \in [0,1/2)$.
\item[(iv)]  Suppose that $\tilde W$ is another $\N$-valued random variable, with $\tilde \rho (s) = \Exp [ s^{\tilde W} ]$,
such that $\tilde \rho(s) \leq \rho(s)$ for all $s \in [0,1]$.
 Then $\alpha^*_{\tilde \rho} \geq \alpha^*_{\rho}$.
\item[(v)] $\alpha^*_\rho = 0$ if and only if $\Pr [ W = 1] >0$.
\item[(vi)] If $\Pr [W=2] =1$, then $\alpha^*_\rho = 1/2$.
\item[(vii)] If $\Exp [W] < \infty$, then $\alpha^*_\rho < 1$.
\end{itemize}
\end{lemma}
\begin{proof}
By Lemma \ref{pgf}, $\rho(0) = \Pr [ W=0]=0$ and $\rho(1)=1$;
hence   $F_{\rho ,\alpha} (  1/2) = 0$ and
$F_{\rho ,\alpha} (0) = (\alpha -1 ) \log 2$,
so that
$F_\rho(\alpha) \geq (\alpha -1)^+ \log 2 \geq 0$.
Since $F_{\rho ,\alpha} (  \gamma)$ is nondecreasing as a function of
 $\alpha \geq 0$, Dini's theorem implies continuity of $F_\rho (\alpha)$ as a function of $\alpha \geq 0$.

 For part (i),  $F_\rho(\alpha) \geq (\alpha -1)^+ \log 2$ implies that $F_\rho (\alpha ) > 0$ for $\alpha >1$, so that
 $\alpha^*_\rho \leq 1$.
On the other hand, for $\alpha <1$, $\gamma_0 (\alpha) \in (0,1/2]$, since by continuity there
is some neighbourhood of $0$ for which $F_{\rho,\alpha} (  \gamma ) < 0$.

Since $F_{\rho ,\alpha} (  \gamma)$ is nondecreasing as a function of
 $\alpha$, for $\alpha' \geq \alpha$, $F_\rho (\alpha') \geq F_{\rho,\alpha'} (\gamma_0 (\alpha) ) \geq F_{\rho} (\alpha)$, i.e.,
 $F_\rho$ is also nondecreasing.
 Hence
$F_\rho (\alpha) >0$
 for $\alpha > \alpha^*_\rho$.
Also, the fact that $F_\rho (\alpha ) =0$ for $\alpha < \alpha^*_\rho$
 is immediate from
the definition of $\alpha^*_\rho$ and the fact that $F_\rho (\alpha ) \geq 0$.
Then $F_\rho (\alpha^*_\rho) = 0$ by the continuity of $F_\rho$ established above.
Thus we obtain part (i).

For part (ii), suppose that $\alpha < \alpha^*_\rho$.
Suppose for some $\gamma_0 \in [0,1/2)$
 that $F_{\rho,\alpha} ( \gamma_0    ) \geq 0$.
Since $\rho (1-2\gamma) > 0$ for $\gamma < 1/2$,
 $F_{\rho,\alpha} (  \gamma_0 )$ is  strictly increasing in $\alpha$,
so there exists $\alpha' \in (\alpha, \alpha^*_\rho)$
for which $F_{\rho ,\alpha'} ( \gamma_0   ) >0$,
 contradicting the definition
of $\alpha^*_\rho$.  This gives (ii).

For part (iii), suppose that $\alpha > \alpha^*_\rho$. Then $F_\rho (\alpha) >0$
by part (i) of the lemma; since $F_{\rho,\alpha}  (  1/2) = 0$, the supremum is attained in $[0,1/2)$.

For part (iv), we have that for any $\gamma \in [0,1/2]$,
$F_{\rho , \alpha } (\gamma) \geq F_{\tilde \rho, \alpha} (\gamma)$,
  since $\tilde \rho( 1-2\gamma) \geq \rho (1-2\gamma)$.
 So $F_\rho (\alpha) \leq F_{\tilde \rho} (\alpha)$
  for all $\alpha \geq 0$, and hence
$\alpha^*_{\tilde \rho} \geq \alpha^*_\rho$.

For the remaining parts of the lemma we use more detailed properties of the generating function
$\rho (s)$ (see Lemma \ref{pgf}).
For part (v),
differentiating in (\ref{FFdef})  we obtain
\begin{equation}
\label{FFdiff}
 \frac{\ud}{\ud \gamma} F_{\rho,\alpha} (  \gamma)
 = - \frac{2\alpha \rho' (1-2\gamma)}{1 + \rho (1-2\gamma) } + \log \left( \frac{1-\gamma}{\gamma} \right) ;\end{equation}
 this is well defined at least for $\gamma \in (0,1)$.
At $\gamma =1/2$ this equates to
 $-2\alpha \Pr [ W =1]$, since by Lemma \ref{pgf}
$\rho'(0) = \Pr [ W=1]$ and $\rho(0)=0$.
So if $\Pr [ W=1] >0$, $F_{\rho,\alpha} (  \gamma)$ is equal to $0$
 at $\gamma=1/2$ and its derivative there is negative for any $\alpha >0$,
so that, for any $\alpha >0$, $F_{\rho,\alpha} (  \gamma) >0$ for some $\gamma < 1/2$. The `if' part of part (v) follows.

Conversely, suppose that $\Pr [ W=1] =0$.
Then the previous argument shows that $F_{\rho,\alpha}(1/2)
=F'_{\rho,\alpha}(1/2) =0$, while a calculation shows
that $F''_{\rho, \alpha} (1/2) = 4 \alpha  \rho''(0)
-4$. Hence by continuity there exists $\delta >0$ such that
for $\alpha < \delta $ and $(1/2)- \delta \leq \gamma \leq 1/2$
we have $F''_{\rho,\alpha}(\gamma) \leq -3$.  Hence by Taylor's theorem,
 $F_{\rho,\alpha}(\gamma) \leq 0$ for $\alpha < \delta $ and
$(1/2)- \delta \leq \gamma \leq 1/2$.  Also,
$F_{\rho,\alpha} (\gamma ) \to
 - \log ( 2 \gamma^\gamma (1-\gamma)^{1-\gamma})$ as $\alpha \to 0$,
which is strictly negative apart from at $\gamma = 1/2$. Thus
by Dini's theorem,
for all $\alpha$ small enough we have
 $F_{\rho,\alpha}(\gamma) \leq 0$
for $\gamma \leq (1/2) -\delta$.
So all together we have shown that $F_{\rho,\alpha} (\gamma) \leq 0$ for all $\alpha$ sufficiently small.
Hence $\alpha_\rho^* >0$ in this case,
 giving the   `only if' part of (v).

For part (vi), suppose that $\Pr[W=2]=1$, i.e., $\rho(s)=s^2$. In this case, (\ref{FFdiff}) has a zero at $\gamma \in [0,1/2)$ if $\alpha = s( \gamma)$
where
\[ s(\gamma ) = \frac{1+(1-2\gamma)^2}{4(1-2\gamma)} \log \left( \frac{1-\gamma}{\gamma} \right) .\]
We claim that $s(\gamma)$ is decreasing on $[0,1/2)$, with a unique minimum of $s (1/2) =1/2$.
To verify this, we show $s'(\gamma) < 0$ for $\gamma \in [0,1/2)$, which, after simplification,
amounts to
\[ \frac{(1+(1-2\gamma)^2)(1-2\gamma)}{8 \gamma^2 (1-\gamma)^2} > \log \left( \frac{1-\gamma}{\gamma} \right)   .\]
Setting $z = 1-2\gamma$, it suffices to show that $\frac{z}{(1-z^2)^{2}} > \frac{1}{2} \log \left( \frac{1+z}{1-z} \right)$
for $z \in (0,1]$, which can be verified by term-by-term comparison of the corresponding power series, namely
$z + 2 z^3 + 3 z^5 + \cdots > z + \frac{z^3}{3} + \frac{z^5}{5} + \cdots$.
 Hence $s(\gamma) = \alpha$ has no solution
for $\alpha <1/2$, in which case the only stationary value of $F_{\rho,\alpha}$ is at $\gamma =1/2$,
necessarily the maximum.
Hence $\alpha^*_\rho \geq 1/2$. On the other hand,
if $\alpha > 1/2$ then
$F''_{\rho,\alpha}(1/2) = 8 \alpha -4 >0$, while
$F_{\rho,\alpha}(1/2) = F'_{\rho,\alpha}(1/2) = 0$,
so by Taylor's theorem and continuity
there exists $\gamma < 1/2$ with $F_{\rho,\alpha}(\gamma) >0$.
Hence $\alpha_\rho^* =1/2$,
 proving (vi).

Finally we prove part (vii). If $\Exp[W]< \infty$, Lemma \ref{pgf}(ii) implies that, as $\gamma \downarrow 0$,
$\rho' (1-2\gamma) =  \Exp [ W] + o(1)$.
Thus  the final term on the right-hand side of (\ref{FFdiff}) dominates in the $\gamma \downarrow 0$ limit,
and there exists $\delta >0$ such that
$\frac{\ud}{\ud \gamma} F_{\rho ,\alpha} (  \gamma) \geq \delta$ for all $\gamma \in [0,\delta]$
and all $\alpha \in [0,1]$.
Then by an application of  the mean value theorem, $F_{\rho,\alpha} (  \delta) \geq (\alpha-1)\log 2 +\delta^2$
for all $\alpha \in [0,1]$. Thus taking $\alpha <1$ close enough to $1$ we see that $F_{\rho,\alpha} (  \delta) > 1$,
which implies that $\alpha^*_\rho < 1$.
\end{proof}

\begin{proof}[Proof of Proposition \ref{alphaprop}.]
Extract the relevant parts of Lemma \ref{alphalem}.
\end{proof}

\subsection{Approximation by the binomial model}
\label{sec:approx2}

In Section \ref{sec:approx1} we showed (in Lemma \ref{approxlem1}) that $\Prn [ A (n, m_n)]$ can be well approximated
by $\Prb [ A(n,m_n)]$  on the logarithmic scale,  provided that $m_n / n \to \alpha$.
The following result is an analogous approximation lemma for $\Expn [ \NN ( n, m_n) ]$.
One could obtain such a result from   Lemma \ref{approxlem1} applied to (\ref{eq1}), with some work
(including dealing separately with terms with $\ell = o(n)$: cf Section \ref{sec:few} below).
However, it is more convenient to proceed directly, albeit using similar ideas to the proof of Lemma \ref{approxlem1};
in this case we are helped
 by the fact that $\Expn [ \NN (n,m)]$
possesses monotonicity properties absent for $\Prn [ A (n,m) ]$.

\begin{lemma}
\label{approxlem2}
Suppose that $W_n \tod W$ and $m_n / n \to \alpha > 0$.
Then
\begin{align*}
\lim_{n \to \infty}
 n^{-1} |\log  \Expn  [ \NN (n, m_n) ] -
\log \Expb [ \NN(n, m_n)] | = 0.
\end{align*}
\end{lemma}
\begin{proof}
We use a coupling argument, constructing the general model with row weights distributed as $W_n \tod W$ on the same probability space
as the binomial model with row weights distributed as $W_n^{\rm bin} \tod W$ (by Lemma \ref{wnconv}).
For any $n$, we can use a probability space in which, for each row, the weight in each model
converges almost surely to a copy of $W$. Indeed, let $W (1), W(2), \ldots$ be
 independent copies of $W$.
Using the Skorokhod representation theorem,
we may take
$W_{n,1}, W_{n,2}, \ldots$ as independent copies of $W_n$, being the weights of the rows in the general model,
such that $W_{n,i} \to W(i)$ almost surely.
Also, take $W^{\rm bin}_{n,i}$ to be the number of odd components in a multinomial $(W(i); n^{-1},\ldots, n^{-1})$ distribution,
so that $W^{\rm bin}_{n,1}, W^{\rm bin}_{n,2}, \ldots$ are independent copies of $W_n^{\rm bin}$ and the weights of the rows
in the binomial model.

Let $A_n (i) := \{ W_{n,i} \neq W^{\rm bin}_{n,i} \}$. Then
for any $\delta >0$, we may take $n$ large enough so that $\Pr [ A_n (i) ] \leq \delta$, uniformly in $i$.
Let $K ( n, m) = \sum_{i=1}^{m} \1_{A_n (i) }$ denote the number of `bad' rows.
Then $K(n,m)$ is stochastically dominated by a $\Bin (m, \delta)$ variable. In particular, for any fixed $\eps>0$ and
any $C < \infty$, standard binomial tail bounds imply that
we may take $\delta$ small enough, and hence
 $n$ sufficiently large, so that
\begin{equation}
\label{badbound}
 \Pr [ K(n,m_n) \geq \eps n ] \leq \Pr [ \Bin ( 2 \alpha n , \delta ) \geq \eps n ]  \leq \exp \{ - C n \} .
\end{equation}
We claim that  each row added to a matrix can increase
 the number of null vectors
by at most a factor of 2; this follows from (\ref{sigmadef})
and (\ref{Ndef}).
Hence
\begin{equation}
\label{difference}
\left| \log  \NN (n,m_n) - \log \NN' (n,m_n) \right| \leq K(n,m_n) ,\end{equation}
where $\NN (n,m_n)$ is the number of null vectors in the matrix with the $W_{n,i}$
and $\NN' (n,m_n)$ is the number of null vectors in the matrix with the $W^{\rm bin}_{n,i}$.
In particular, on $\{ K(n,m_n ) \leq \eps n \}$,  the bound in (\ref{difference}) is $\eps n$.
The statement in the lemma follows from (\ref{badbound}) and (\ref{difference}), since $\eps>0$ and $C<\infty$ were arbitrary.
 \end{proof}

\subsection{Null vectors consisting of few rows}
\label{sec:few}

In the asymptotics of $\Expn [ \NN (n ,m) ]$, it turns out that   null vectors
of {\em low weight} play a distinct and important role.
Recall (\ref{wtsum}). The main result of this section is the following lemma,
which exhibits a polynomial growth rate for null vectors of few rows.

\begin{lemma}
\label{lem3}
Suppose that there exist $r_0 \geq 3$ and $r_1 < \infty$ such that $\Pr  [ r_0 \leq W_n \leq r_1 ] =1$
for all $n$.
 Suppose that $m_n / n \to \alpha >0$. Then there exists $\delta > 0$ such that
 \begin{equation}
\label{lowwt}
\sum_{2\leq \ell \leq \delta  n} \Expn [ \NN ( n, m_n ; \ell )  ] = O ( n^{2-r_0} ).
\end{equation}
\end{lemma}

\begin{rmk}
 The exponent $2-r_0$ in (\ref{lowwt}) cannot be improved when $\Pr[ W_n = r_0 ] >0$, because
 $\Exp [ \NN ( n , m_n ; 2)]$
is itself of order $n^{2-r_0}$. Indeed, there are of order $m_n^2$ weight-2 candidate
vectors, and each is null if each of the two corresponding rows have $r_0$ non-zeros in matching
positions, an event of probability of order $n^{-r_0}$.
\end{rmk}

\begin{proof}[Proof of Lemma \ref{lem3}.]
Let $n, \ell \in \N$.
Let $R = R(n,\ell)$ denote the `column range' of the matrix
 $M(n,\ell)$, that is,
the number of columns of degree at least 1.
Let us take $k= k ( \ell ) \in \N$, to be chosen later.
 We shall estimate $\Prn [A(n,\ell)]$
by considering separately the events  $R \leq k$ and
$R >k$.

We interpret $M(n,\ell)$ as arising from a random allocation scheme,
where for each row we throw balls at random into $n$ urns (columns).
If $R \leq k$ then there is
some set of $k$ columns, such that
all the balls land in these $k$ columns.
For each ball, the probability
that it lands
in one of the first $k$ columns, given that the
other balls cast so far for that row all land in
the first $k$ columns, is at most
$k/n$. Hence since for each row at least $r_0$ balls
are cast, and we consider $\ell $ rows here,
\begin{equation}
\Prn [ R \leq k] \leq \binom{n}{k} \left( \frac{k}{n} \right)^{\ell r_0}
\leq \frac{n^{k- \ell r_0} k^{\ell r_0}}{k!} .
\label{120501a}
\end{equation}
If $R >k$ then to have $A(n,\ell)$ occur we need
to have each of the columns in the range get hit
at least twice
(i.e., have degree at least 2). Thus if $R >k$ and $A(n,\ell)$ occurs
there  is a collection  of $k +1$ columns such
that each column in the collection gets hit at least twice.
Let $B(i)$ be the event
that the column $i$ gets hit at least twice.
The probability that a particular entry is
1, given the values of up to
$k$ other entries in the same row,
is at most $r_1/(n-k)$. Hence
 the union bound yields for $1 \leq j \leq k+1$ that
$$
\Prn [B(j) \mid \cap_{i=1}^{j-1} B(i)] \leq
 \binom{\ell}{2} \left( \frac{r_1}{n-k} \right)^{2},
$$
and hence we have
$$
\Prn [ \cap_{i=1}^{k+1} B(i) ] \leq
 \left( \binom{\ell}{2} \left( \frac{r_1}{n-k} \right)^{2} \right)^{k+1}
$$
so that by the union bound, provided $k \leq n/2$ we have
\begin{align*}
\Prn [ \{R >k\} \cap A(n,\ell) ]
\leq  \binom{n}{k+1}
 \left( \binom{\ell}{2} \left( \frac{2 r_1}{n} \right)^{2} \right)^{k+1}
\leq
 \frac{ n^{-(k+1)} \ell^{2(k+1)} c_1^{k+1} }{(k+1)!},
\end{align*}
where we put $c_1 = 2 r_1^2$.
Combined with (\ref{120501a}) this gives
\begin{align*}
\Prn [ A(n,\ell) ]
\leq
 \frac{ n^{k- \ell r_0} k^{\ell r_0}}{k!}
+
 \frac{ n^{-(k+1)} \ell^{2(k+1)} c_1^{k+1} }{(k+1)!}.
\end{align*}
We assume $m_n/n \to \alpha \in (0,\infty)$,
so for $n$ large enough so that $m_n \leq (1+\alpha)n$ we
 have
for all $\ell$, and for $k \leq n/2$, that
\begin{align}
\Expn [ \NN (n, m_n; \ell) ]
& = \binom{m_n}{\ell}
\Prn [ A(\ell,n) ]
\nonumber \\
& \leq \left( \frac{ ((\alpha +1 )n)^{\ell} }{
\ell ! } \right)
\left(
 \frac{ n^{k- \ell r_0} k^{\ell r_0}}{k!}
+
 \frac{ n^{-(k+1)} \ell^{2(k+1)} c_1^{k+1} }{(k+1)!}
\right).
\label{120502a}
\end{align}
Taking $k = \ell  + r_0 -2$,
 we obtain for each  fixed $\ell$ that for some constant
$c(\ell)$ we have
\begin{align}
\Expn [ \NN (n, m_n; \ell) ] & \leq
c(\ell) (n^{\ell (1-r_0) + \ell + r_0 -2 }
+
 n^{\ell - k -1} )
\nonumber \\
& =
c(\ell) (  n^{(\ell -1)(2 -r_0)}
+  n^{1 -r_0}),
\label{120501c}
\end{align}
which is $O(n^{2 -r_0})$ for any fixed $\ell \geq 2$.

Fix an integer $K \geq 2$, to be chosen later, and
consider $K \leq \ell \leq \delta n$.
Now put $k = \ell( r_0 - 1) - \lceil \ell/2 \rceil$.
Assume $\delta \leq 1/(2 r_0)$; then for $\ell \leq \delta n$
this choice  of $k$ satisfies $k \leq n/2$, so that (\ref{120502a})
remains valid. Also note that, since $r_0 \geq 3$,
 $k \geq \frac{3 \ell}{2} -1 \geq \ell$ provided $\ell \geq 2$.

By the bound $\re^{\ell} \geq \frac{ \ell^{\ell} }{\ell !}$
and similar for $k$,
there are constants $c_2, c_3, c_4$ such that
 the first term in the
right side of (\ref{120502a})
(i.e. the product of the first factor with the
first term in the second factor) is bounded by
a constant times
\begin{align}
\frac{n^{\ell(1-r_0) +k} k^{\ell r_0} c_2^\ell}{\ell^\ell k^k}
& \leq  \frac{n^{-\lceil \ell/2 \rceil} \ell^{r_0 \ell} c_3^\ell }{ \ell^\ell \ell^{(r_0- 1)\ell - \lceil \ell/2 \rceil} }
\nonumber \\
& = \left( \frac{c_4 \ell }{n} \right)^{\lceil \ell/2 \rceil},
\label{120502b}
\end{align}
where for the inequality
 we used the fact that $\ell \leq k$ to replace $k^k$ by $\ell^\ell$ in the denominator.

Similarly, there are constants $c_5, c_6, c_7$ such that
 the second term in the
right side of (\ref{120502a}) is bounded by
a constant times
\begin{align*}
\frac{n^{\ell -k  - 1 } \ell^{2 \ell (r_0 -1) - 2 \lceil \ell /2 \rceil + 2 } c_5^\ell }{
\ell^\ell (k+1)^{k+1} }
& \leq \frac{ n^{\ell (2 -r_0 ) + \lceil \ell/2 \rceil   } \ell^{  \ell (2 r_0 -2) - 2 \lceil \ell/2 \rceil} c_6^\ell \ell^2 }{
\ell^{\ell r_0 - \lceil \ell/2 \rceil +1} n }
\\
& \leq \left( \frac{c_7 \ell }{n}  \right)^{\ell (r_0 - 2) -\lceil \ell/2 \rceil } (\ell/n).
\end{align*}
Combining with (\ref{120502b}),
since $r_0 \geq 3$ so $r_0 - 2 \geq 1$
and $\lceil \ell/2 \rceil \leq (\ell/2) +1$,
we can find a constant $c_8$ such that
for $2 \leq \ell \leq n$ we have
$$
\Expn [ \NN (n,m_n;\ell) ]
\leq c_8
\left(\frac{c_8 \ell}{n} \right)^{\ell/2} .
$$

By calculus we have that
$\left(\frac{c_8x}{n} \right)^x$ is decreasing in $x \leq  n/(c_8 \re)$,
so provided $\delta \leq 1/(c_8 \re)$ the last bound is maximized,
over $K \leq \ell \leq \delta n$,  at $\ell =K$, so that
$$
\sum_{K \leq \ell \leq \delta n }
\Expn [ \NN (n,m_n;\ell) ]
\leq
c_8 \delta n \left( \frac{c_8 K}{n} \right)^{K/2},
$$
which is $O(n^{2-r_0})$ provided we choose $K $ so that
$K/2 \geq r_0 -1$.
\end{proof}

\subsection{Proof of Theorem
 \ref{thm1} }
\label{mainproofs}

First we prove  Theorem \ref{thm1} for the binomial model, i.e.,
for $\Prb$ on the right-hand side of (\ref{eq1}),
 and then use Lemma \ref{approxlem2}.
Specifically, we prove the following result.

\begin{lemma}
\label{binthm}
Suppose that  $\Pr [ W \geq 1 ] =1$.
  Suppose that $m_n/ n \to \alpha \in (0,\infty)$
  as $n \to \infty$.  Then with $F_\rho (\alpha)$ as defined
  by (\ref{Fdef2}),
  \begin{align}
\lim_{n \to \infty} n^{-1} \log  \Expb [\NN(n,m_n)]
= F_\rho (\alpha).
\label{suplim}
\end{align}
\end{lemma}
\begin{proof}
From (\ref{pmulti}),
\begin{align}
\label{ub1}
& \sum_{\ell=0}^m \binom{m}{\ell} \Prb[ A (n,\ell) ]
\leq (m+1) (n+1) 2^{-n} \sup_{0 \leq \ell \leq m} \sup_{0 \leq j \leq n}  \binom{m}{\ell} \binom{n}{j} | \rho (1 - (2j/n) ) |^\ell  \nonumber\\
  &   \phantom{mmmmmmm} \leq (m+1) (n+1) 2^{-n} \sup_{ \beta \in [0,1]} \sup_{\gamma \in [0,1]}  \binom{m}{\beta m} \binom{n}{\gamma n} | \rho ( 1-2\gamma) |^{\beta m},
\end{align}
setting $\binom{n}{x} = 0$ for $x \notin \{0,1,\ldots,n\}$. Write
\begin{align*}
S_\alpha(\beta,\gamma) :=
  \left( \frac{|\rho(1-2\gamma)|^\beta}{ \beta^\beta
 (1-\beta)^{1-\beta}} \right)^\alpha
  \left( \frac{1}{2 \gamma^\gamma (1-\gamma)^{1-\gamma}} \right).
\end{align*}
Taking $m=m_n = O(n)$ in (\ref{ub1})
and using the first inequality in (\ref{chooseub2}),
we obtain
\begin{align}
\label{eq2}
n^{-1} \log \sum_{\ell=0}^{m_n} \binom{m_n}{\ell} \Prb[ A (n,\ell) ]
 \leq O( n^{-1} \log n )
 + \log
 \sup_{ \beta \in [0,1]} \sup_{\gamma \in [0,1]}
 S_{m_n/n} ( \beta, \gamma ) .
 \end{align}
 For any  $B \geq 0$, routine calculus (with a separate argument for $B=0$) shows that
\[
\sup_{  \beta \in [0,1]} \left(
 \frac{B^\beta}{\beta^{\beta}(1-\beta)^{1-\beta}} \right)
= B+1,
\]
with the supremum attained at $\beta = B/(1+B)$, so that from (\ref{eq2}) we have
\begin{align}
\label{eq3}
n^{-1} \log \Expb [ \NN (n, m_n ) ]
 \leq O( n^{-1} \log n )
 +
 \sup_{\gamma \in [0,1]}
\log  \left( \frac{ ( 1 + | \rho ( 1 - 2\gamma ) | )^{m_n/n}}{2 \gamma^\gamma (1-\gamma)^{1-\gamma}} \right) . \end{align}
Considering the transformation $\gamma \mapsto 1 -\gamma$, we see that
\[  \sup_{\gamma \in [1/2,1] } \! \left( \frac{ ( 1 + | \rho ( 1 - 2\gamma ) | )^{\alpha}}{2 \gamma^\gamma (1-\gamma)^{1-\gamma}} \right)
= \! \sup_{\gamma \in [0,1/2]} \! \left( \frac{ ( 1 + | \rho ( 2\gamma- 1 ) | )^{\alpha}}{2 \gamma^\gamma (1-\gamma)^{1-\gamma}} \right)
\leq \! \sup_{\gamma \in [0,1/2]} \! \left( \frac{ ( 1 +   \rho ( 1- 2\gamma )  )^{\alpha}}{2 \gamma^\gamma (1-\gamma)^{1-\gamma}} \right) ,\]
  since, for $\gamma \in [0, 1/2]$, $| \rho (2\gamma-1) | \leq \rho (1-2\gamma)$,
  by Lemma \ref{pgf}(iii).
  Hence from (\ref{eq3}) we have, with $F_\rho (\alpha )$ as defined
 at (\ref{Fdef2}),
$ n^{-1} \log \Expb [ \NN (n, m_n ) ]
 \leq O( n^{-1} \log n )
 +  F_\rho ( m_n /n  )$.
 Since $m_n /n \to \alpha$ and $\alpha \mapsto F_\rho (\alpha)$ is continuous  (see Lemma \ref{alphalem}),
 \[ \limsup_{n \to \infty}   n^{-1} \log \Expb [ \NN (n, m_n ) ] \leq   F_\rho (\alpha) . \]

For the lower bound, we use the fact that
\begin{align*}
\Expb [ \NN (n, m_n) ]  & \geq
\sum_{\ell = 0}^{m_n  } \binom{m_n}{\ell} \sum_{j=0}^{ \lfloor n/2 \rfloor } 2^{-n} \binom{n}{j} ( \rho (1-(2j/n)) )^\ell \\
& \geq 2^{-n} \sup_{\beta \in [0, 1] } \sup_{\gamma \in [0,1/2]} \binom{m_n}{\beta m_n } \binom{n}{\gamma n}
| \rho ( 1- 2 \gamma ) |^{\beta m_n } ,\end{align*}
using the nonnegativity of the appropriate terms for both inequalities.
Using the lower bound  in (\ref{chooselb}), similarly to   above,
we obtain that $\liminf_{n \to \infty}   n^{-1} \log \Expb [ \NN (n, m_n ) ] \geq   F_\rho (\alpha)$.
Hence combining the upper and lower bounds,   we obtain
(\ref{suplim}).
\end{proof}

Now we can give the proof of our main result.

\begin{proof}[Proof of Theorem \ref{thm1}.]
Lemma \ref{binthm} shows that (\ref{thm1eq}) holds for the case where $W_n = W_n^{\rm bin}$, and Lemma \ref{approxlem2} shows that
the result carries over to the general case.
For the final statement of the theorem, suppose that $\alpha < \alpha^*_\rho$ and that $\Pr [ r_0 \leq W_n \leq r_1 ] =1$ for some $r_0 \geq 3$
and $r_1 < \infty$.
 Lemma \ref{lem3} shows that, for a suitable $\delta >0$,
\begin{equation}
\label{ub3} \sum_{\ell=1}^{  \delta n} \binom{m_n}{\ell}
\Prn  [ A (n,\ell) ] = O (n^{2-r_0} ) .
\end{equation}
For $\ell \geq \delta n$, we first restrict to the binomial model.
Choose $\eps >0 $ so that $(3 \eps)^{\delta} < 2^{-2\alpha}$.
By a similar argument to (\ref{ub1}), but splitting the supremum
 over $j$ into two parts,
\begin{align}
\label{ub2}
& \!\!\!\! \sum_{\ell=   \delta n}^{m_n}  \! \scalebox{0.9}{$\dbinom{m_n}{\ell}$}  \Prb[ A (n,\ell) ]
\leq (m_n+1) (n+1) 2^{-n} \!\! \sup_{\delta n \leq \ell \leq m_n} \sup_{j : | j - (n/2) | \leq \eps n}  \! \scalebox{0.9}{$\dbinom{m_n}{\ell}$}  \scalebox{0.9}{$\dbinom{n}{j}$} | \rho (1 - (2j/n) ) |^\ell \nonumber\\
& \phantom{mmmmmm} {}
+ (m_n+1) (n+1) 2^{-n} \sup_{0 \leq \ell \leq m_n} \sup_{j : | j - (n/2) | > \eps n}  \scalebox{0.9}{$\dbinom{m_n}{\ell}$}  \scalebox{0.9}{$\dbinom{n}{j}$} | \rho (1 - (2j/n) ) |^\ell . \end{align}
Similarly to (\ref{eq3}), the second term on the right-hand side of (\ref{ub2})
is bounded above by
\[ \exp \left\{ o(1) + \sup_{\gamma \in [0, (1/2) - \eps ]} F_{\rho, m_n/n} (\gamma)  \right\} ,\]
which decays to $0$ exponentially fast, by Lemma \ref{alphalem}(ii), since $m_n /n \to \alpha \in (0, \alpha^*_\rho)$.
On the other hand, for $| j - (n/2) | \leq \eps n$, we have from Lemma \ref{pgf}(i) that
$| \rho (1 - (2j/n) ) | \leq 3 \eps$, for $\eps >0$ small enough, so that,
 since $\ell \geq   \delta n$,
$| \rho (1 - (2j/n) ) |^\ell \leq (3\eps)^{  \delta n}$,
so by the choice of $\eps$ the first term in
the right hand side of (\ref{ub2})
 tends to $0$ exponentially fast.
Hence
\begin{equation}
\label{266}
 \limsup_{n \to \infty} n^{-1} \log \sum_{\ell= \delta n}^{m_n} \binom{m_n}{\ell} \Prb [ A (n,\ell) ] < 0 .\end{equation}

We next deduce a version of (\ref{266}) with $\Prn$ in place of the special case $\Prb$, using Lemma \ref{approxlem1}
once more.
To this end, observe first that the $\Prn$-analogue of the sum in (\ref{266})
consists of $O(n)$ nonnegative terms, so is bounded between the largest term and $O(n)$ times that same term, so that
\[  n^{-1} \log \sum_{\ell= \delta n}^{m_n} \binom{m_n}{\ell} \Prn [ A (n,\ell) ]  = n^{-1} \log \max_{ \delta n \leq \ell \leq m_n } \binom{m_n}{\ell} \Prn [ A (n,\ell) ]
+ O ( n^{-1} \log n ) .\]
By Lemma \ref{approxlem1}, for any $\eps >0$, there exist $\eps_{n,\ell}$ with $| \eps_{n,\ell} | \leq \eps$,
uniformly for $\ell$ with  $\delta n \leq \ell \leq m_n$ and $n$ sufficiently large, such that
\[   \binom{m_n}{\ell} \Prn [ A (n,\ell) ] =  \binom{m_n}{\ell} \Prb [ A (n,\ell) ] \exp \{ \eps_{n,\ell} n \} .\]
So we obtain
\[
 \limsup_{n \to \infty} n^{-1} \log \sum_{\ell= \delta n}^{m_n} \binom{m_n}{\ell} \Prn [ A (n,\ell) ] < 0  ,\]
 which, combined with (\ref{ub3}), yields (\ref{polybound}).
\end{proof}

\section{Cores of sparse random hypergraphs}
\label{sec:cores}

\subsection{Hypergraphs and 2-cores}
\label{sec:hypermodel}

Given a set $\VV= \{v_1,\ldots,v_n\}$,
 whose elements we call {\em vertices},
a non-empty subset of $\mathcal{V}$ is called a {\em hyperedge}. Given a
collection $\EE:=\left(E_i\right)$ of $m$ hyperedges,
 we refer to the pair $(\VV,\EE)$ as a {\em hypergraph}.
This hypergraph may be identified with an $m \times n$
 matrix $A$ with entries in $\{0,1\}$ (the {\em incidence matrix} of
the hypergraph),
  having no zero rows, as follows.
The entry
$a_{i,j} $ of $A$ takes the value $1$ if and only if
 $v_j \in E_i$, in which case we say
 row $i$ is {\em incident} to column $j$, and that
hyperedge $E_i$ is incident to vertex $v_j$,
and refer to $(E_i,v_j)$  as an {\em incidence}
of the hypergraph.

The number of hyperedges incident to a vertex $v$ is
 the {\em degree} of $v$. Fix a hypergraph $(\VV,\EE)$.
 For   $\FF \subseteq \EE$,
 the set $V(\FF) \subseteq \VV$
of vertices which are incident to at least one of the hyperedges in $\FF$ is called the {\em vertex span}
of $\FF$. We identify the hypergraph $(\FF, V (\FF))$  by the edge subset $\FF$  that induces it,
and call $\FF \subseteq \EE$
 a {\em partial hypergraph}.
A partial hypergraph $\FF \neq \emptyset$ is   a {\em hypercycle}
if every vertex $v$ has even degree with respect to $\FF$.
For an incidence matrix $A$, a left null vector is
 the indicator of a hypercycle in the corresponding hypergraph.

Given a hypergraph $(\VV,\EE)$,
the {\em 2-core} is
defined via
 the following algorithm:
\begin{enumerate}
\item If there exists no vertex 
  of degree one, stop.
\item Otherwise, select an arbitrary vertex of degree one, and delete the unique incident hyperedge; then return to Step 1.
\end{enumerate}
The algorithm terminates, because the partial hypergraphs are decreasing;
 the terminal partial hypergraph, which
 does not depend on the arbitrary choices
made in Step 2 (see \cite[pp.\ 127--128]{dn1}), is called the {\em 2-core} of $\EE$,
denoted $\mbox{Core}(\EE)$. Possibly
$\mbox{Core}(\EE) $ has no hyperedges.
Figure
\ref{fig6} illustrates a hypergraph with 19 vertices and 12 hyperedges, of which 4 hyperedges are in the 2-core.

\begin{figure}[!h]
\center
\includegraphics[width=8cm]{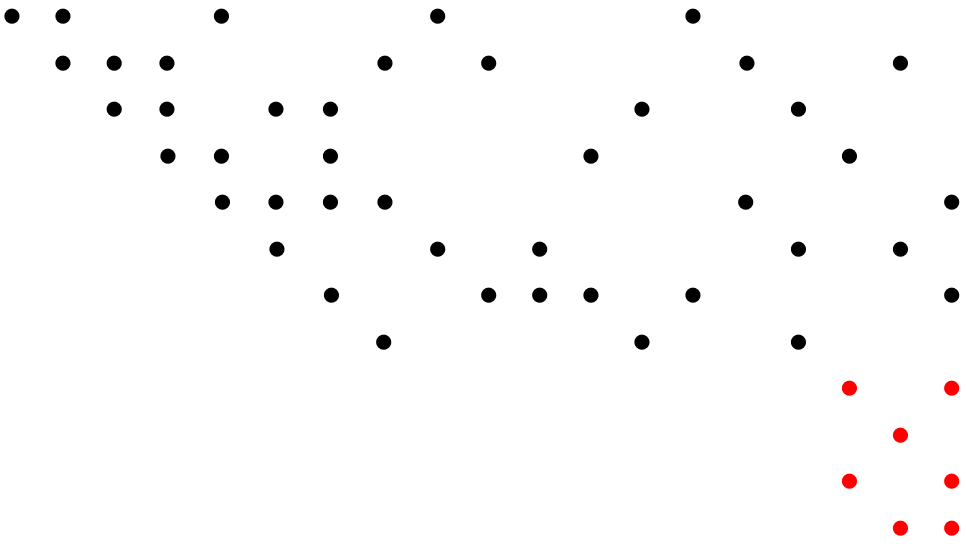}
\caption{Pictorial representation of the adjacency matrix for a  hypergraph with 19 vertices and 12 hyperedges, whose incidences are shown as $\bullet $. The first 8 hyperedges are not in the
 2-core.
The last 4 hyperedges form a linear system of rank 3 over GF[2].
The entire $19 \times  12$ system has rank
11, and so contains a hypercycle (hyperedges 9 and 11).}
\label{fig6}
\end{figure}

The connection between the $2$-core and hypercycles was exploited by Cooper \cite[p.\ 371]{cooper2},
following an idea that he attributes to Molloy (see \cite[p.\ 268]{cooper1}).
The connection is demonstrated by the following useful observation.

\begin{lemma}
\label{hypercycle-2core}
Suppose that the 2-core $\mathcal{C}:=\mbox{Core}(\EE)$ of a hypergraph
$(\VV,\EE)$ has vertex span $V(\mathcal{C}) \subseteq \mathcal{V}$ and size (number of hyperedges) $|\mathcal{C}|$.
\begin{itemize}
\item[(i)] Any hyperedge $E \notin \mathcal{C}$ cannot belong to a hypercycle of $(\VV,\EE)$.
\item[(ii)] If $\mathcal{C} = \emptyset$, then $(\VV,\EE)$ contains no hypercycle.
\item[(iii)] If $|V(\mathcal{C})|<|\mathcal{C}|$, then $(\VV,\EE)$   contains a hypercycle.
\end{itemize}
\end{lemma}
\begin{proof}
If there are $s$ hyperedges not in the 2-core $\mathcal{C}$, there exists a labelling
of them as $E_1, E_2, \ldots, E_s$ with the property that, for every $j$, $E_j$ has some vertex with
degree one after hyperedges $E_1, E_2, \ldots, E_{j-1}$ are removed.
Suppose $(\VV,\EE)$ has some hypercycle $\FF \neq \emptyset$.
None of $E_1, E_2, \ldots, E_s$ can belong to $\FF$:  otherwise,
there would be some minimum $j$ for which $E_j \in \FF$, and this $E_j$ has some
vertex $v$ of degree one in the partial hypergraph from which $E_1, E_2, \ldots, E_{j-1}$
have been removed, which contains $\FF$; so $v$ cannot have even degree in $\FF$,
which is a contradiction. This proves (i), and (ii) follows.
For (iii), say $c:=|V(\mathcal{C})|<|\mathcal{C}|=:r$. Then there are $2^r - 1$ non-empty
partial hypergraphs, but only $2^c < 2^r - 1$ possible indicator vectors for a set of vertices
of odd degree. By the pigeonhole principle, there must be two distinct partial hypergraphs
$\FF, \FF' \subseteq \EE$ for which the sets of vertices of odd degree
are the same. Then $\FF \triangle \FF'$ is a hypercycle.
 \end{proof}

\subsection{The $2$-core in uniform random hypergraphs}
\label{sec:fluid}

In this subsection we consider a certain {\em uniform random}  hypergraph model,
which is different from (but related to) the hypergraph model induced by our random matrix $M(n,m_n)$; in Section \ref{secPoisson}
we will connect the two models.
The incidences (non-zero entries)
 in a random incidence matrix may be viewed as edges of a random bipartite
graph, whose left and right nodes are the row labels and column labels, respectively.
A standard probability model for a random bipartite graph is to fix the degrees of all
nodes in advance (subject to a consistency condition),
and sample uniformly from the bipartite graphs with this set of left and right node degrees.

By interpreting an incidence matrix as a hypergraph, as above,
this automatically gives a {\em uniform random hypergraph} model, where the degree of a left node is the
{\em weight} of a hyperedge (meaning its number of incident vertices), and the
degree of a right node is the vertex degree defined above.
Darling and Norris \cite{dn2} analyse
the statistical properties of the 2-core for such random hypergraphs
under suitable conditions on the hyperedge-weight and vertex-degree distributions.
In unpublished work of the same authors, a generalization to unbounded vertex degrees and row weights is given, under  finite third moments assumptions.

For the purposes of the present paper, we require a more modest relaxation
 of the conditions of \cite{dn2}, to cover the case
where the row weights remain uniformly bounded but the vertex
 degrees are approximately Poisson distributed.

For each $n$, define vectors of nonnegative integers $\bd_n := ( d_n ( k ) : k \in \ZP)$
and
$\bw_n := ( w_n (k) : k \in \N)$ with $\sum_{k \geq 0} d_n (k) =n$
and $m_n := \sum_{k \geq 1} w_n (k)$; we assume that $\bd_n$ and $\bw_n$ are compatible
in the sense that $\sum_{k \geq 1} k w_n (k) = \sum_{k \geq 0} k d_n (k) < \infty$. We also assume that $m_n \to \infty$.
Suppose that for each $i \in \N$ and $j \in \ZP$,
\begin{equation}
\label{uniflimits}
 \lim_{n \to \infty} \frac{w_n (i)}{\sum_{k \geq 1} w_n (k) } = \rho_i; ~~~ \lim_{n \to \infty} \frac{d_n (j)}{n} = \nu_j .\end{equation}
Define generating functions $\rho (s) := \sum_{k \geq 1} \rho_k s^k$ and $\nu (s) := \sum_{k \geq 0} \nu_k s^k$.
We assume that the weights are uniformly bounded, i.e., $\rho_k = 0$ for all $w$ sufficiently large,
and the degree distribution has all moments, i.e., $\sum_{k \geq 1} \nu_k k^\beta < \infty$ for all $\beta >0$. Under these
conditions, $\rho'(1)$ and $\nu'(1)$ are the (finite) means corresponding to these distributions.

Consider a sequence of random hypergraphs with $n$ vertices and $m_n$ hyperedges, selected uniformly
from those hypergraphs with edge weight multiplicities $\bw_n$ and vertex degree multiplicities $\bd_n$.

  To present asymptotic results for the 2-cores of a sequence of such uniform random hypergraphs,
it helps to introduce the notion of sampling a single incidence uniformly at
 random from all incidences in a hypergraph. Denote such an
  incidence $(E,v)$.
Denote the weight of $E$ by $S+1$, and the degree of $v$
by $L+1$; thus $S$ is the number of other vertices in this hyperedge,
and $L$ is the number of other hyperedges incident to this vertex.
Size bias occurs here: the event that $E$ has weight $k$ occurs
with probability
proportional to $k$ times the number of rows of weight $k$,
 and similarly
the probability that the degree of $v$ is $d$ is proportional to $d$ times the
number of degree $d$ vertices. Given the $\rho_w$ and $\nu_d$ describing the limiting
row weight and vertex degree distributions, we may thus compute
 a pair of limiting probability generating functions for $L$ and $S$, respectively:
  \begin{equation} \label{e:pgfs}
\lambda (s) := \Exp[s^L] = \sum _{d=0}^{\infty } \lambda_d s^d; ~~~
\sigma (s) := \Exp[s^S] = \sum _{w=0}^{\infty } \sigma_w s^w,
\end{equation}
 where, due to the size biasing, the coefficients in (\ref{e:pgfs}) are given by
\[ \lambda_d = \frac{(d+1) \nu_{d+1}}{\nu' (1)} ; ~~~ \sigma_w = \frac{(w+1) \rho_{w+1}}{\rho'(1)} .\]
Hence the generating functions themselves become:
\begin{equation}
\label{gf2}
\lambda (s) = \frac{\nu'(s)}{\nu'(1)}; ~~~
\sigma (s) = \frac{\rho'(s) }{\rho'(1) }. \end{equation}

To avoid triviality, we assume that $\sigma _0 = 0$ (equivalently, $\rho_1 =0$),
i.e., there are no 1-edges,
 and $\lambda _0\notin \{0,1\}$
(otherwise the 2-core is of no interest). Define
\begin{equation}
\label{varphidef}
\varphi (s) :=1-\lambda (1-\sigma (s)).
\end{equation}

From the conditions $\sigma _0\neq 1$ and $\lambda _0\neq 1$, we
deduce that  $\varphi
:[0,1]\to \R$ is strictly increasing; moreover $\varphi (0) = 1 -\lambda (1-\sigma (0 ) ) =0$ (since $\sigma_0 = 0$)
and $\varphi (1) = 1 -\lambda_0 \in (0,1)$,
so $\varphi $ takes values in $[0,1)$,
and there exists a largest solution $g^*$ in $[0,1)$ of the equation $\varphi (s)=s$.
That is, 
\begin{equation} \label{sstar}
\starg := \sup\{s \in [0,1): \varphi (s) = s\}.
\end{equation}
In the case where $\starg >0$ and
the curve $y = \varphi (s)$  crosses the curve $y = s$ (rather
than just touching it) at
$s=\starg$,  we also have
\begin{equation} \label{e:crossing}
\starg = \sup\{s \in (0,1): \varphi (s) > s\}.
\end{equation}

Now we can state the result on the $2$-core that we shall use, which
   amounts to a variant of Theorem 7.1 of \cite{dn2}.

\begin{theorem} \label{t:2core}
Consider  a sequence of uniform random hypergraphs associated with
 sequences $\bw_n$ and $\bd_n$
satisfying (\ref{uniflimits}) with $\rho_w =0$ for all $w$ large enough and $\sum_{d \geq 1} \nu_d d^\beta < \infty$ for all $\beta>0$.
Suppose that the corresponding pair (\ref{e:pgfs}) of random-incidence generating functions
has $\sigma_0 = 0$, $\lambda_0 \notin \{0,1\}$, and is such that $\starg$, given
by (\ref{sstar}),
has either $\starg = 0$ or $\starg$ satisfying (\ref{e:crossing}).
Then the following hold a.s.\ in the limit as $n  \to \infty$.
\begin{itemize}
\item[(i)] If  $\starg = 0$, the proportion of hyperedges which survive in the 2-core converges to zero.
\item[(ii)] If $\starg > 0$, then for
any $k \in \ZP$ with $\rho_k >0$, 
 the proportion of weight-$k$ hyperedges which
survive in the 2-core is asymptotically $(\starg)^k$; overall, a proportion $\rho(\starg)$ of
 hyperedges survive, and a proportion $s^*\sigma(\starg)$ of incidences.
\item[(iii)] 
If $\starg > 0$, then for any $d,k \in \N$ with $2 \leq d \leq k$
and $\nu_k >0$,
 the proportion of vertices of degree $k$ whose degree in the 2-core is $d$
converges to
\[
\binom{k}{d} \sigma(\starg)^d (1 - \sigma(\starg))^{k - d}  .
\]
\item[(iv)] If $\starg > 0$, the 2-core is again a uniform random hypergraph, given its hyperedge weights
and vertex degrees, whose distributions are determined by the previous assertions.
\end{itemize}
\end{theorem}
 As mentioned above,
in \cite{dn2}  all but finitely many coefficients of the generating functions (\ref{e:pgfs}) were taken to be zero, but the
methods admit the modest extension of this section, and indeed
 can be extended to the case where $\lambda''(1)$ and $\sigma''(1)$ are
finite, corresponding to
finite third moments for hyperedge weight and vertex degree distributions.
Because of its proximity to
the result in \cite{dn2},
 we do not prove Theorem \ref{t:2core} here.

\subsection{Application to $M(n,m)$}
\label{secPoisson}

In the previous subsection we assumed the row and column weights
of our random matrix
were specified in advance, but now we return to
the random matrix model used in the rest of the paper,
so   our random $m_n \times n$ incidence matrix $A$
will be precisely the matrix $M(n,m_n)$,
described in Section \ref{sec:matrixlaw}, i.e.,
with i.i.d.\ rows with weights having the distribution of $W_n$, and corresponding
generating function $\rho_n(s)$ having limit $\rho(s)$.

To justify being able to apply Theorem \ref{t:2core} in this setting,
we  give the following strong law of large numbers
for the empirical distributions of the row and column weights
of $M$.
\begin{lemma}
\label{emplem}
Suppose $m_n \in \N$ with $m_n/n \to \alpha $ as $n \to \infty$,
with $\alpha >0$, and the $W_n$ are uniformly bounded.
Let $k \in \ZP$,
 let $N_k(n)$ be the number of rows of $M(n,m_n)$ of
weight $k$, and let $\tilde{N}_k (n)$
 be the number of columns  of $M(n,m_n)$ of degree $k$.
Then a.s.,
\begin{align}
\label{emprow}
\lim_{n \to \infty} m_n^{-1} N_k(n)  & = \Pr[W=k] , ~~~\textrm{and} \\
\label{empcol}
\lim_{n \to \infty} n^{-1} \tilde{N}_k(n) &  = \frac{\re^{-\mu} \mu^k}{k!} ,\end{align}
where we set $\mu:= \alpha \Exp [W] = \alpha \rho'(1) $.
Moreover,
the total number of incidences   satisfies the law of large numbers
\begin{align}
\label{incidences}
 \lim_{n \to \infty} n^{-1} \sum_{k \geq 0}  k N_k (n) =  \lim_{n \to \infty} n^{-1} \sum_{k \geq 0} k \tilde N_k (n) = \mu, \ \mathrm{a.s.} 
\end{align}
\end{lemma}
\begin{proof}
First note that ${m_n}^{-1} \Exp [N_k(n)] =
 \Pr [ W_n =k]$, which converges to $\Pr[W=k]$ by assumption.
To deduce almost sure convergence from this convergence in means,
we use the Azuma--Hoeffding inequality in a standard way, as follows.
Fix $n$ and for $1 \leq i \leq m_n$
let $\FF_i$ be the $\sigma$-algebra generated
by the rows $X_1,\ldots,X_i$ of $M(n,m_n)$.
Define  $\xi_i = \Exp [ N_k(n) \mid \FF_i]$, with $\xi_0 = \Exp[N_k(n)]$.
Since resampling a single row changes the number of rows of weight
$k$ by at most 1, we have for $1 \leq i \leq m$ that
\[
| \xi_i - \xi_{i-1} | = |\Exp [ N_k(n) - N_k(n,i) \mid \FF_i]| \leq 1,
\]
where $N_k(n,i)$ is defined like $N_k(n)$ but based on a matrix
with the $i$th row resampled.
By the Azuma--Hoeffding
inequality  applied to the martingale $(\xi_0,\ldots,\xi_{m_n})$ we
obtain  for any $\eps >0$ that
\[
\Pr [ | N_k(n) - \Exp[N_k(n) ]| > \eps n ]
\leq 2 \exp( - \eps^2 n/2),
\]
so by the first Borel--Cantelli lemma, we have
$ | N_k(n) - \Exp[N_k(n) ]| \leq \eps n $ for
all but finitely many $n$ almost surely. Combined with
the convergence of the mean, this gives us
(\ref{emprow}).

The remaining two parts of the lemma use the assumption $\Pr [ W \leq r_1 ] =1$ for $r_1 < \infty$.
To prove (\ref{empcol}) note that
the weight of the first column (or any other column) of
$M(n,m_n)$ is binomially distributed
with parameters $m_n $ (number of trials)
 and $\Exp[W_n]/n$ (probability of success).
Hence
$\Exp[\tilde{N_k}(n)/n] =
\Pr[{\rm Bin}(m_n,\Exp[W_n]/n) =k]$, and
by binomial-Poisson convergence this tends to
$\re^{-\mu} \mu^k/k!$
as $n \to \infty$. Given this convergence of means, we
may prove (\ref{empcol}) by a similar argument (based on
 the Azuma--Hoeffding inequality)  to the one used to prove
(\ref{emprow}), since resampling a single row changes
the number of columns of degree $k$ by at most $r_1$.

For the final statement in the lemma, we have that
\begin{align*} n^{-1} \sum_{k \geq 0} k N_k (n)   = (m_n/n) \sum_{k = 0}^{r_1} k m_n^{-1} N_k (n)
  \to \alpha \sum_{k=0}^{r_1} k \Pr [ W = k] , \ \mathrm{a.s.}, \end{align*}
by (\ref{emprow}), and then  (\ref{incidences}) follows.
\end{proof}

\begin{corollary}
\label{transfer}
Consider the random matrix model $M(n,m_n)$ with row weight distribution
$W_n \tod W$, where the $W_n$ are uniformly bounded.
Suppose that $m_n / n\to \alpha>0$.
Then, a.s.,
taken as   hypergraph incidence matrices the sequence $M(n,m_n)$ defines a sequence of uniform random hypergraphs
whose row weight   and vertex degree distributions
satisfy  (\ref{uniflimits})
with $\rho_w$ and $\nu_d$ given by $\rho_w = \Pr [ W =k]$
and $\nu_d = \re^{-\mu} \mu^d /d!$   respectively,
where  $\mu:=\alpha \Exp[W]$.
\end{corollary}
\begin{proof}
Since the distribution of $M(n,m)$ is
invariant under permutations of the rows or columns,
conditional on the empirical distribution of row and column
weights, all possible outcomes with those row and column
weight distributions are equally likely, so this conditional
distribution is indeed uniform. Moreover by
Lemma \ref{emplem} the limiting proportion of rows of weight  $k$ is given by $\Pr[W=k]$ and
the limiting proportion of columns of
 degree $k$ is given by $\Pr[D=k]$
where
$D\sim$ Po$(\mu)$,
 with $\mu:=\alpha \rho'(1)$.
Hence conditionally on this sequence of empirical distributions, almost surely we have a sequence of random matrices satisfying the
hypotheses of
 Section \ref{sec:fluid}.
\end{proof}

In the notation of  Section \ref{sec:fluid}, in this case $\nu (s) = \sum_{d=0}^\infty \re^{-\mu} \frac{(s\mu)^d}{d!} = \re^{\mu (s-1)}$
is the generating function of a Po$(\mu)$ random variable, so, by (\ref{gf2}),
the pair (\ref{e:pgfs})
becomes
\[
\lambda (s)
=  \re^{\mu (s - 1)}; ~~~
\sigma (s) = \frac{\rho'(s) }{\rho'(1) }.
\]

In this case, we
have from (\ref{varphidef}) that
\[ \varphi (s) := 1-\lambda (1-\sigma (s)) = 1 - \re^{- \mu \sigma (s)} = 1 - \re^{- \alpha \rho'(s)} ,\]
and to emphasize the dependence on $\alpha$ we will use the notation $\varphi_\alpha$ for $\varphi$ from now on. 
Recall from (\ref{sstar})  that $\starg$ was defined as the largest
 $s \in [0,1)$ for which $\varphi_\alpha (s) =s$.
In order for the model of this section to fit into the setting discussed in
 Section \ref{sec:fluid}, we need to assume that $\sigma_0 =0$ and $\lambda_0 \notin \{0,1\}$.
Here $\lambda_0 = \re^{-\mu} = \re^{-\alpha \Exp[W]}$ and $\sigma_0 = \frac{\Pr[ W=1]}{\Exp[W]}$.
So it   suffices to assume that $\alpha >0$,
 $\Pr [ W \geq 2] =1$, and $\Exp[ W] < \infty$; in this case the argument in Section \ref{sec:fluid} 
shows that $\starg$ is well defined.

Note that $g^*$ depends both on $\rho$ and on $\alpha$; in this section we
 write $\starg = g^*(\alpha)$ to emphasize the dependence on $\alpha$;
we will show (see Lemma \ref{gstarlem}) that the present definition is equivalent to that at  (\ref{gs0}) given in Section \ref{sec:general}.
For any solution $s \in [0,1)$ to $\varphi_\alpha (s) = s$,  so in particular for $s = \starg(\alpha)$,
 provided $\rho'(s) \neq 0$, we have $\alpha = h(s)$ as given by (\ref{hdef}).

We note some facts about $g^*(\alpha)$; recall the definition of $\alpha_\rho^\sharp$ from (\ref{alphasharp}).

\begin{lemma}
\label{gstarlem}
Suppose that $\Pr [ W \geq 2] = 1$ and $\Exp[W] < \infty$.
 With the convention $\sup \emptyset = 0$, the definition (\ref{gs0}) is equivalent to
 the definition (\ref{sstar}) of  $g^*(\alpha)$ 
as the largest solution of $\varphi_\alpha (s) = s$.
Also,
 $\alpha_\rho^\sharp \in [0,1]$,
and  $g^*(\alpha ) = 0$ for all $\alpha \in [0, \alpha_\rho^\sharp)$,
and for $\alpha > \alpha_\rho^\sharp$, the function $g^*(\alpha)$ is positive and strictly increasing,
with $g^* (\alpha) \uparrow 1$ as $\alpha \to \infty$.

Now assume also that
  $\Pr [ W \geq 3 ] =1$ and $\Exp[W^2] < \infty$.
  Then   the following hold.
\begin{itemize}
\item[(i)] We have $g^*(\alpha_\rho^\sharp ) \in (0,1)$ and
$\alpha_\rho^\sharp =
h ( g^* ( \alpha_\rho^\sharp ) ) \in (0,\infty)$.
\item[(ii)] The function
$g^*$ is right continuous, and
there is a finite set $\DD_\rho \subset (0,\infty)$, with $\alpha_\rho^\sharp = \inf \DD_\rho$,
such that $g^*$ is continuous apart from jumps at points of $\DD_\rho$. For each $\alpha \in \DD_\rho$,
 $h(g^*(\alpha)) = \alpha$ is a local minimum for $h$.
  \item[(iii)] If $\alpha \notin \DD_\rho$, then $g^*(\alpha)$ satisfies the crossing condition (\ref{e:crossing}).
\end{itemize}
\end{lemma}
\begin{proof}
Since $\Pr [ W \geq 2 ]=1$ and $\Exp[W]< \infty$,
 we have $\varphi_\alpha(1) <1$ and
$\varphi_\alpha(0)=0$.  Therefore by continuity 
we may rewrite
(\ref{sstar}) as
$$
g^*(\alpha) = \sup \{s \in (0,1): \varphi_\alpha(s) \ge s\},
$$
using the convention $\sup \emptyset =0$.
By the definition  (\ref{hdef}) of $h$,
for $s \in (0,1)$ it is easy to check 
that
$\varphi_\alpha(s) \geq s$ if and only if $h(s) \leq \alpha$,
and this shows that (\ref{gs0}) and (\ref{sstar}) give equivalent
definitions of $g^*(\alpha)$.

By (\ref{hdef}) and subsequent remarks, 
$h(x)$ is positive, continuous in $x$, and tends to infinity 
as $x \uparrow 1$. By 
the definition
(\ref{alphasharp}), and the subsequent remark,
 $\alpha_\rho^\sharp \in [0,1]$. 
By the definition
 (\ref{gs0}), it is clear that $g^*(\alpha) =0$ for
$\alpha \in [0,\alpha_\rho^\sharp)$, and
the fact that $g^*(\alpha) $ is positive and strictly
increasing for
$\alpha \in (\alpha_\rho^\sharp,\infty)$
is easily deduced
from the continuity of $h$.
Also, given $\varepsilon \in (0,1)$ we can choose
$\alpha $ with $h (1 -\varepsilon) < \alpha$  so that
$g^*(\alpha) > 1-\varepsilon$, and together with
the monotonicity of $g^*$ this shows $g^* (\alpha) \to 1$ as
$\alpha \to \infty$.

For part (i), under the extra assumption $\Pr[W \geq 3]=1$
 we have $ h $ going to infinity at 0 and at 1,
and by continuity $h$ attains its infimum on $(0,1)$,
so  using (\ref{alphasharp})
and (\ref{gs0}) we have that $g^*(\alpha^\sharp_\rho)$ is the supremum
of a non-empty compact set contained in $(0,1)$,
and so lies in (0,1). The last part of (i) also follows
from the continuity of $h$.

For part (ii), under the extra assumption
$\Exp[W^2]< \infty$, note first that if $ 0  \leq y <
 \alpha_\rho^\sharp$ then $g^*(y) =0 $. Hence
$g^*$ is continuous at $y$ for all $y < \alpha_ \rho^\sharp$.

 Now let  $y \geq \alpha_\rho^\sharp$;
 note that by (\ref{gs0}) and continuity of $h$, we have $h(g^*(y)) =y$.
 Take a monotonic sequence $y_n $ tending to $y$; set 
$x_n = g^*(y_n)$. 

Suppose first that $y_n \downarrow y$.
Then the sequence $x_n$ is nonincreasing;
denoting the limit by $x_\infty$ we have  $h(x_n) =y_n$ so
$h(x_\infty) =y$ by continuity, and therefore $x_\infty \leq g^*(y)$
by (2.9).
Since also $x_n \geq g^*(y)$ by monotonicity we have
$x_\infty = g^*(y)$; hence $g^* $ is right-continuous at $y$.

Now suppose instead  that  $y_n \uparrow y$.
Set $x = g^*(y)$. If $h$ does not have a local minimum at
$x$ then $\liminf g^*(y_n) \geq x$, so that 
$x_n \to x$, and hence
$g^*$ is left-continuous at $y$. Hence, if
 $g^*$ is discontinuous at $y$ then
$h$ has a local
minimum at $g^*(y)$.

The function $h'$ is analytic and non-constant on $(0,1)$
so its zeros do not accumulate except possibly at 0 or
1. However  $h'(x) =0$ implies $\rho'(x)/\rho''(x) = - (1-x) \log (1-x)$,
so by the assumption $\Exp[W^2] < \infty$ there
exists $\varepsilon > 0$ such that $h'(x) \neq 0$ for
  $ 1-\varepsilon < x < 1$
and for $0  < x < \varepsilon $;
 for the latter case we use the fact that, as $x \downarrow 0$,
\[  \frac{\rho''(x)}{\rho'(x)} (1-x) \log (1-x) \to r_0 -1 > 1, \]
if $r_0 \geq 3$ is the smallest possible value of $W$.
 Thus $h$ has only finitely
many local minima in $(0,1)$, and hence $h$ has
 a local minimum at  
$g^*(y)$ for at most finitely many $y$. This completes the proof
of (ii).

For part (iii) note that, for $s \in (0,1)$,
 $\varphi_\alpha(s) > s$ if and only if
$h(s) < \alpha$, so (\ref{e:crossing})
gives $g^*(\alpha) = \sup \{s \in (0,1): h(s) < \alpha\}$,
 which for $\alpha \notin {\cal D}_\rho$ 
 agrees with the definition (\ref{gs0}).
\end{proof}

To apply the results in Section \ref{sec:fluid} to $M(n,m_n)$,
 we need to assume that  
  $g^*(\alpha)$ either is  zero or  satisfies (\ref{e:crossing}).
Lemma \ref{gstarlem} shows that a sufficient condition  for this is that
  $\alpha \in (0,\infty)$, $\alpha \notin \DD_\rho$.

By  Theorem \ref{t:2core}(ii) and (\ref{incidences}), $n^{-1}$ times the
number of incidences which survive in the 2-core converges a.s.\ to
\begin{equation}
\label{2coreinc}
\mu g^* \sigma( g^* ) = (\alpha \rho'(1)) g^* \frac{\rho'( g^* ) }{\rho'(1) }
= \alpha g^* \rho'(g^*) = - g^* \log (1 - g^* ).
\end{equation}
By Theorem \ref{t:2core}(iii) and (\ref{empcol}),
for $d \geq 2$
 the proportion of original vertices whose degree in the 2-core is $d$ is
asymptotically
\begin{align}
\label{2coredegs}
\sum_{k \geq d} \re^{-\mu} \frac{\mu^k}{k!} \binom{k}{d} \sigma( g^* )^d (1 - \sigma( g^* ))^{k - d}
& = \re^{-\mu} \frac{ (\mu \sigma (g^*))^d}{d!} \sum_{j \geq 0} \frac{ \mu^j (1- \sigma(g^*))^j}{j!} 
\nonumber\\
& = \re^{-\mu \sigma (g^*)} \frac{(\mu \sigma (g^*))^d}{d!} ,
\end{align}
the remainder having degree $0$ in the $2$-core (some columns in the core have degree
 zero because in the algorithm of Section \ref{sec:hypermodel}, we never delete any columns). In other words,
 the 2-core vertex degrees have the
distribution of a random variable $D \1 \{ D \neq 1\}$, where
 $D \sim {\rm Po} (\mu \sigma( g^* ))$;
by (\ref{2coreinc}),
$\mu \sigma( g^* ) = \alpha \rho'( g^* )$.
As a check on the previous calculation (\ref{2coreinc}) of the number of surviving incidences,
$n^{-1}$ times the total number of incidences
in the 2-core should converge to the mean of the vertex-degree distribution, which is
$
\alpha \rho'( g^* ) (1 - \re^{-\alpha \rho'( g^* )}) = \alpha g^* \rho'(g^* )$,
as   in (\ref{2coreinc}).

The key issue  in predicting hypercycles and
rank deficiency is whether
the number of rows in the 2-core exceeds the number of
  occupied columns in the 2-core,
which is treated in Theorem \ref{c:alphabar};
 a related result appears in \cite{cooper2}. Recall the definition
 of $\psi(g)$ from (\ref{psidef}) and $\ubar{\alpha}_\rho$ from (\ref{ubardef}).

\begin{theorem} \label{c:alphabar}
Suppose $W_n$ are uniformly bounded and $\Pr [ W \geq 3] = 1$. Let $\alpha \in (0,\infty)$.
 Consider the 2-core of the random incidence matrix $M(n,m_n)$
where $m_n/n \to \alpha$ as $n \to \infty$.
Then
if $\alpha < \alpha_\rho^\sharp$, the number
 of rows   in the 2-core is $o(n)$, a.s.

Now suppose $\alpha > \alpha_\rho^\sharp$, so $g^* = g^*(\alpha) >0$, and suppose that
 $\alpha \notin \DD_\rho$.
 Then:
\begin{itemize}
\item[(i)]
$n^{-1}$ times the number of rows in the 2-core converges a.s.\ to $\alpha \rho(g^*)$.
\item[(ii)] $n^{-1}$ times the number of occupied columns in the 2-core converges a.s.\ to
$1 - \re^{-\nu}(1 + \nu)$, where $\nu:=\alpha \rho'(g^* )$.
\item[(iii)]
Almost surely, for all $n$ large enough, the $2$-core has more rows than occupied
columns if $\psi (g^*(\alpha)) <0$ but has fewer rows than occupied columns if $\psi (g^*(\alpha)) >0$.
Moreover, there exists $\delta >0$ such that if $\alpha \in (\ubar \alpha_\rho, \ubar \alpha_\rho +\delta)$,
for all $n$ large enough,
 the $2$-core has more rows than columns
 and so the corresponding hypergraph has a hypercycle.
 \end{itemize}
\end{theorem}

Figure \ref{fig2} shows an example of some of the more exotic behvaiour that can occur in the random weight setting. In the case where $\rho(s) = 0.9183 s^3+ 0.04 s^{19} + 0.0417 s^{41}$,
$\psi ( g^*(\alpha))$ changes sign several times, and so Theorem \ref{c:alphabar} shows that as $\alpha$ increases from $0$  the $2$-core switches
from having asymptotically more columns than rows to having more rows than columns not just once (at $\ubar \alpha_\rho$), but {\em twice}. This non-monotone behaviour
does not occur in the fixed weight case. Proposition \ref{psiprop} below, and the subsequent discussion, explains some of the features in the figure.

\begin{figure}[!h]
\center
\includegraphics[width=7cm]{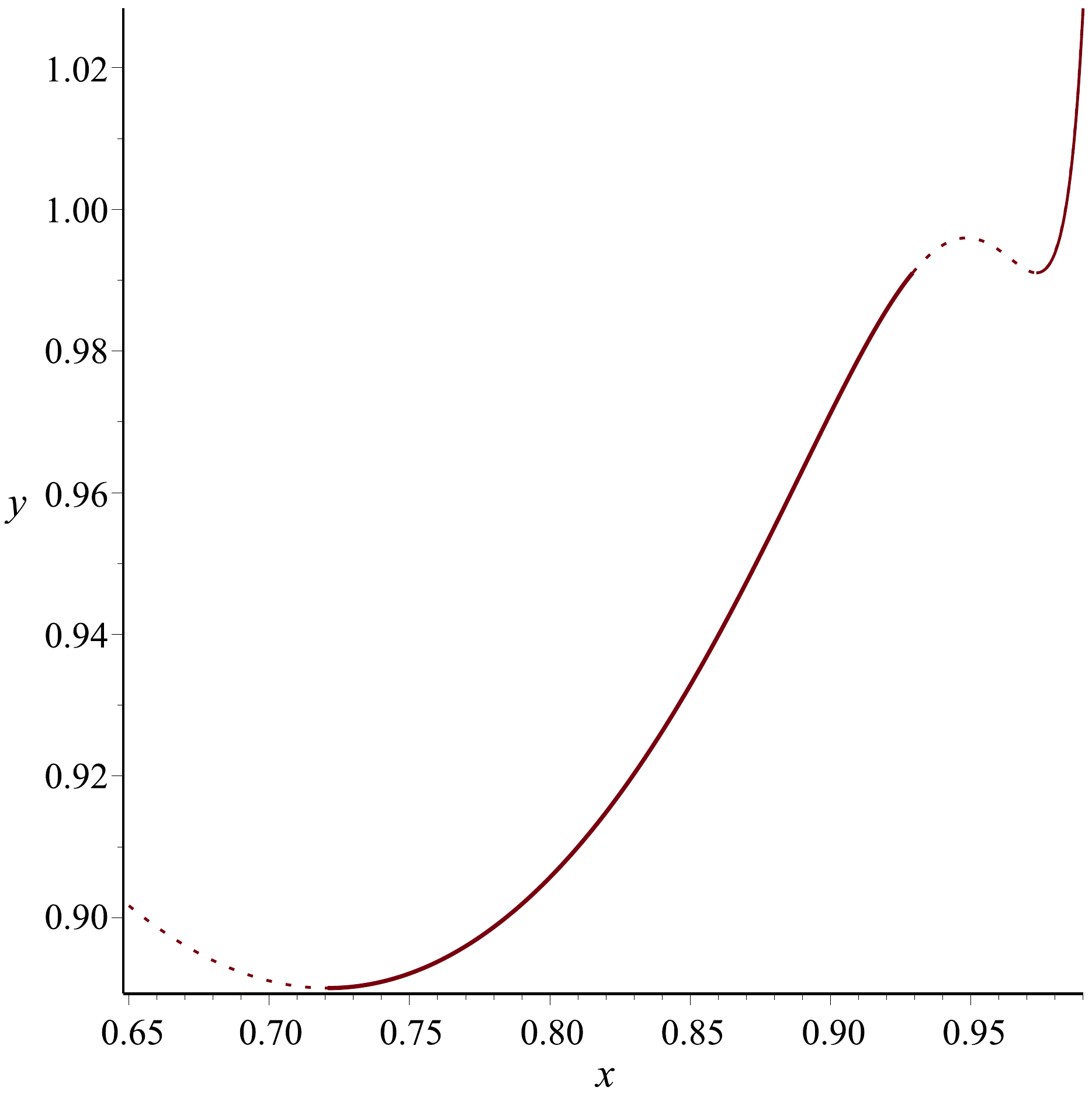} \quad
\includegraphics[width=7cm]{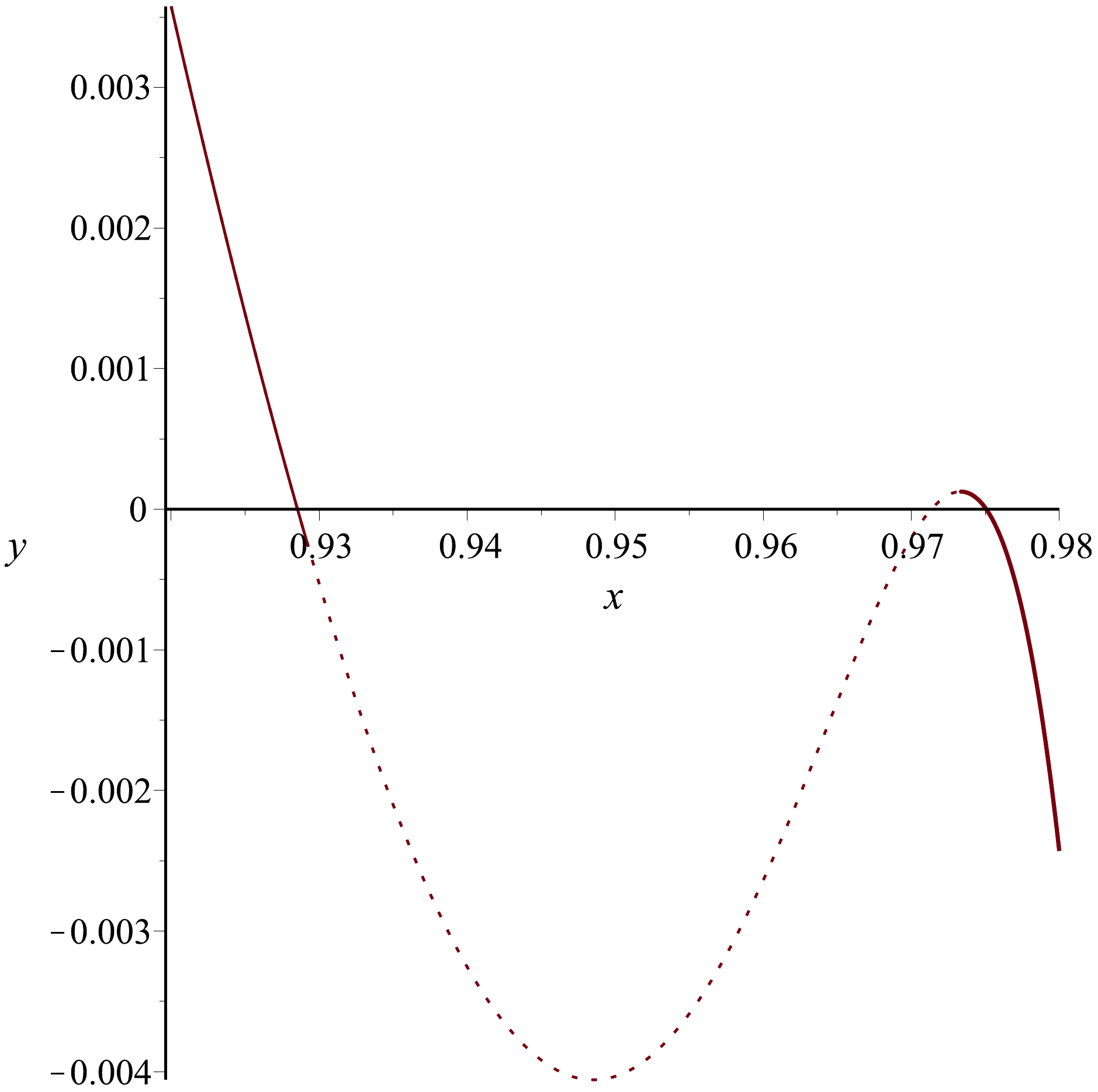}
\caption{Example with $\rho(s) = 0.9183 s^3+ 0.04 s^{19} + 0.0417 s^{41}$. The left plot shows parts of the curves $y= h(x)$ (all the line) and $x = g^*(y)$ (solid line).
The right plot shows parts of the curves 
$y=  \psi(x)$ (dashed line) and the
 locus of $(g^*(\alpha),\psi(g^*(\alpha)))$ (solid line).
Again $g^*(\alpha)$ has two discontinuities, one 
at $\alpha = \alpha_\rho^\sharp \approx 0.890061$ and one at
$\alpha \approx 0.991044$, with the first corresponding to a jump from $g^* =0$ to $g^* \approx 0.720793$ and the second
to a jump from $g^* \approx 0.929269$ to $g^* \approx 0.973325$. The right plot shows  the three positive roots of $\psi(x)=0$.
The two of these roots achieved by
 $\psi(g^*(\alpha))$ are at $x \approx  0.928538$ and $x \approx 0.975069$. The first corresponds to
$\alpha = \ubar \alpha_\rho \approx 0.990686$ and the second to
 $\alpha \approx 0.991185$. Hence as $\alpha$ ranges in $(0,1)$,
$\psi (g^*(\alpha))$ changes sign from positive, to negative, to positive, and finally to negative again.}
\label{fig2}
\end{figure}

\begin{proof}[Proof of Theorem \ref{c:alphabar}.]
By Corollary \ref{transfer}, a.s.\ we have a sequence of random matrices satisfying the
hypotheses of
 Theorem \ref{t:2core}.
If $\alpha < \alpha_\rho^\sharp$, then $g^*(\alpha) = 0$, and Theorem \ref{t:2core}
shows that the 2-core has $o(n)$ rows. So from now on suppose that $\alpha > \alpha_\rho^\sharp$,
 so $g^* = g^*(\alpha) >0$ (see Lemma \ref{gstarlem}).

For the statement (i), note that out of $m_n \sim \alpha n$ rows, a proportion $\rho(g^* )$ survives, by Theorem \ref{t:2core}(ii).
For (ii), the discussion around (\ref{2coredegs}) implies that
the proportion of the $n$ original vertices
whose degree in the $2$-core is non-zero is obtained by
subtracting from $1$ the mass that a Po$(\nu)$ random variable places on
$\{0,1\}$.

For (iii),
we compare the
limits in (i) and (ii).
Suppose that these limits
satisfy
 \begin{equation}
 \label{morerows}
 \alpha \rho (g^*)  > 1 - \re^{-\alpha \rho'(g^*)} (1+ \alpha \rho'(g^*)).
 \end{equation}
By our assumptions on $\alpha$ and $W$, we have
$\rho(0)=\rho'(0) =0$ and $g^* >0$,
 which implies that $\rho (g^*)$ and $\rho'(g^*)$ are both positive.
Then we may rewrite (\ref{morerows}) as
\begin{align*}
\alpha \rho(g^* )
& >   1 - (1 - g^* )(1 + \alpha \rho'(g^* ) ) \\
& = g^* +  (1-g^* ) \log (1 - g^* ) ,
\end{align*}
using the definition of $g^*$.
Now substituting in $\alpha = h (g^*(\alpha))$ 
for $\alpha$ on the left-hand side of the last display (given $\rho'( g^*) > 0$)
 we may rewrite the last inequality as $\psi (g^*) < 0$, where $\psi$ is defined by (\ref{psidef}).
 Similarly, $\psi ( g^*) > 0$ is equivalent to
 \begin{equation}
 \label{morecols}
 \alpha \rho (g^*)  < 1 - \re^{-\alpha \rho'(g^*)} (1+ \alpha \rho'(g^*)).
 \end{equation}

If $\psi (g^*) <0$, then (\ref{morerows}) holds and the limit in (i) in strictly greater than the limit in (ii),
which shows that the $2$-core eventually has more rows than occupied columns, and vice versa if $\psi (g^*) >0$ (so that
(\ref{morecols}) holds). 

Recall the definition of $\ubar \alpha_\rho$ from (\ref{ubardef}).
We know from Lemma \ref{gstarlem}(iii) that $g^*$ has only
finitely many discontinuities. Either $\ubar \alpha_\rho$
is a continuity point for $g^*(\alpha)$ (and hence for $\psi(g^*(\alpha))$), or
else $\ubar \alpha_\rho \in \DD_\rho$
with $\psi ( g^* ( \ubar \alpha_\rho ) ) <0$ and no other point of
$\DD_\rho$ is in a neighbourhood of $\ubar \alpha_\rho$. In either
case, $\psi (g^*(\alpha)) <0$ for $\alpha$ in an interval of the form $[\ubar \alpha_\rho , \ubar \alpha_\rho +\delta )$
with $\delta>0$.
The final statement in the theorem,
about the existence of a hypercycle, follows from Lemma \ref{hypercycle-2core}(iii).
\end{proof}

We next state a result giving an upper bound for $\ubar \alpha_\rho$.
\begin{proposition}
\label{lessthan1}
  Suppose that $\Pr [ W \geq 3]=1$ and $\Exp [ W^2 ] < \infty$. Then
$\ubar \alpha_\rho \leq 1$.
\end{proposition}
\begin{proof}
We know from Lemma \ref{gstarlem} that $\alpha_\rho^\sharp \leq 1$, so
if $\ubar \alpha_\rho \leq  \alpha_\rho^\sharp$ there is nothing
to prove. Hence we assume $\ubar \alpha_\rho > \alpha_\rho^\sharp$ from now on. 
First we show that
\begin{equation}
\label{psicross}
\textrm{  for any  } \eps >0 \textrm{ there exists  } \alpha \in
 (\ubar \alpha_\rho - \eps, \ubar \alpha_\rho),
\textrm{ such that  } 
\psi ( g^*(\alpha )) > 0.
\end{equation}
By the definition (\ref{ubardef}) 
of $\ubar \alpha_\rho$, and the assumption
$\ubar \alpha_\rho > \alpha_\rho^\sharp$,
 if (\ref{psicross}) fails then there exists
$\delta >0 $ such that
$\psi \circ g^*$ is identically zero on the interval
$I: = (\ubar \alpha_\rho - \delta, \ubar \alpha_\rho)$,
and by taking $\delta$ small enough we may assume the
interval $I$ contains no discontinuities of $g^*$.
But then the image $J : =  g^*(I)$ is also an open interval
because $g^*$ is continuous and strictly increasing on $I$.
So we would then have $\psi$ identically zero on $J$, which would
 contradict the fact that $\psi$ is analytic and non-constant on $(0,1)$. 
Thus (\ref{psicross}) must hold as asserted.

Observe next that every time the 2-core algorithm deletes a row,
 it has to create at least one column of degree zero, and possibly more. So the aspect ratio (i.e., number of rows divided by number of occupied columns)
  is nondecreasing at each step of the algorithm, provided the 
initial aspect ratio is at least $1$.
Hence the aspect ratio of a non-empty 2-core
  is at least as large as the aspect ratio of the original  incidence matrix to which the algorithm is applied, provided the latter
is at least $1$.

 So if  $m_n/n \to \alpha > 1$, then
  the aspect ratio of the original matrix exceeds $1$ for all $n$ large enough,
   and hence so does the aspect ratio of the $2$-core, assuming it exists.
  Suppose that $\ubar \alpha_\rho > 1$. Then by (\ref{psicross})
and the finiteness of $\DD_\rho$,
   there exists $\alpha' \in (1, \ubar \alpha_\rho) \setminus \DD_\rho$
   such that $\psi ( g^* (\alpha') ) > 0$.
   Then, by  Theorem \ref{c:alphabar}(iii), with $m_n / n \to \alpha = \alpha'$, the $2$-core has aspect ratio less than $1$ for all $n$ large enough,
 which  contradicts the previous conclusion that $\alpha >1$ implied the $2$-core having limiting aspect ratio greater than $1$.
 Hence $\ubar \alpha_\rho \leq 1$.
 \end{proof}

Next we give 
 more information on the key functions $h$ and $\psi$, which should clarify the situation in Theorem \ref{c:alphabar}(iii).
By a {\em root} of $\psi$, we mean
 any number $x $ with $\psi(x)=0$.

 \begin{proposition}
  \label{psiprop}
  Suppose that $\Pr [ W \geq 3]=1$ and $\Exp [ W^2 ] < \infty$. Then
  $0 < \alpha_\rho ^\sharp \leq \ubar \alpha_\rho   \leq 1$. The function
  $\psi$ has
  at least one root in $(0,1)$, and $h$ has at least one local minimum in $(0,1)$.
Suppose that   the following condition  holds:
\begin{itemize}
\item[(a)] $h$  has a single local minimum $x_\rho^\sharp$ in  $(0,1)$, with $h( x_\rho^\sharp) = \inf_{x \in (0,1)} h(x)$.
\end{itemize}
Then $x_\rho^\sharp$ is the location of the unique local maximum of $\psi$ in $(0,1)$,
$\psi ( x_\rho^\sharp ) >0$,
and the interval $(0, 1)$ contains exactly one root of $\psi$, denoted $x^*_\rho$,
which satisfies 
 $x_\rho^\sharp < x^*_\rho$.
Moreover, $\ubar \alpha_\rho   = h(x^*_\rho) > \alpha_\rho^\sharp$,  
and
\begin{equation}
\label{sharp}
\psi ( g^*(\alpha) ) \begin{cases}
> 0 & \textrm{ for all } \alpha \in ( \alpha_\rho^\sharp , \ubar \alpha_\rho) \\
< 0 & \textrm{ for all } \alpha > \ubar \alpha_\rho .
\end{cases}
\end{equation}
Finally, in the fixed row-weight case where where $W = r \geq 3$  a.s.,
condition  (a)   holds, and the unique positive  root of $\psi$
is $x^*_r \in (\frac{r-2}{r-1},1)$.
\end{proposition}

An important observation that helps to explain the  close connection between the functions $h$ and $\psi$ (apparent in Figure \ref{fig1}, for example) and will also form an ingredient in the
proof of Proposition \ref{psiprop} is the following result.

\begin{lemma}
\label{lem:hpsi}
  For all $x \in (0,1)$, $\psi' (x)$ has the same sign as $-h'(x)$, so, in particular, the locations of the local minima of $h$ correspond
exactly to the  locations of the local maxima of $\psi$ in $(0,1)$. Finally, if $\Exp [W^2]<\infty$, then
as $x\downarrow 0$ we have
\begin{align} \psi ( 1- x) & = 1 - h (1-x) - x - \frac{\Exp[ W(W-1)]}{2\Exp[W]} ( 1 +o(1)) x^2 \log x \nonumber\\
& = 1 - h(1-x) - x + o(x) .
\label{hpsi}
\end{align}
\end{lemma}
\begin{proof}
Differentiating (\ref{psidef}), we obtain
\begin{align}
\label{psidiff1}
\psi'(x)     =  - \frac{\rho (x)}{\rho'(x)} \left( \frac{1}{1-x} + \frac{\rho''(x)  }{\rho'(x)} \log (1-x) \right) .
\end{align}
On the other hand, from (\ref{hdef}), we have that, for $x \in (0,1)$,
\[ h' (x) = \frac{1}{\rho'(x)} \left( \frac{1}{1-x} + \frac{\rho''(x)}{\rho'(x)} \log (1-x) \right) = - \frac{1}{\rho(x)} \psi'(x) ,\]
by comparison with  (\ref{psidiff1}).
Finally,  (\ref{hpsi}) follows from a routine calculation.
\end{proof}

Before completing the proof of Proposition \ref{psiprop}, we make some further remarks and present some examples.
The main complication in the interpretation of Theorem \ref{c:alphabar}(iii) is due to
 the fact that $g^*$ has discontinuities,
so $\{ \psi (g^*(\alpha)) : \alpha \geq 0 \}$ is only a subset
of $\{ \psi (x) : x \in [0,1) \}$. Let
\[ \GG_\rho := \{ g^* (\alpha) : \alpha \geq \alpha_\rho^\sharp \} .\]
By Lemma \ref{gstarlem}(ii), $\GG_\rho$ is  a union of finitely many intervals $\GG_\rho = [g^-_1, g^+_1) \cup \cdots \cup [g^-_\ell, g^+_\ell)$
where $g^-_1 < g^+_1 < g^-_2 < \cdots < g^+_\ell$, and,
 for each $k$, $g_k^- = g^*(\alpha)$ for  $\alpha \in \DD_\rho$, and $h( g_k^-)$ is a local minimum.
Recall that $\alpha = h ( g^*(\alpha))$ and $\alpha \mapsto g^*(\alpha)$ is increasing for $\alpha > \alpha_\rho^\sharp$ (see Lemma \ref{gstarlem}), 
so $x \mapsto h (x)$
must be increasing on $\GG_\rho$. So in fact $\alpha_\rho^\sharp = h ( g_1^- ) < \cdots < h( g_\ell^-)$. The `curve'
$\psi ( g^*(\alpha) )$, $\alpha \geq \alpha_\rho^\sharp$ is then a (discontinuous) trace of $\psi (x)$, where $x$ runs over $\GG_\rho$,
piecewise continuously on intervals starting at $g_k^-$ which, by Lemma \ref{lem:hpsi}, correspond to local {\em maxima} of $\psi$.
Figures \ref{fig1} and \ref{fig2} give some illustrations of possible behaviour. Observe that $\psi (x)$ is not necessarily decreasing for all $x \in \GG_\rho$.

Note that condition  (a) in Proposition \ref{psiprop}
 is not necessary for the sharp transition property (\ref{sharp})
to hold. Two other relevant conditions are:
\begin{itemize}
\item[(b)] $\psi$ has a single   root in $(0,1)$;
\item[(c)] the global minimum of $h$ on $(0,1)$ is the rightmost local minimum.
\end{itemize}
If $\Pr [ W \geq 3] =1$ and $\Exp[W]<\infty$, then $h(x) \to \infty$ as $x \to 0$ and as $x \to 1$, so
  (a) $\Rightarrow$ (c), while in the course of the proof of Proposition \ref{psiprop} below, we show
   that (a) $\Rightarrow$ (b) as well. We mention  3 illustrative  examples.
   \begin{itemize}
   \item
An example   for which   conditions (a) and (b) do not hold but (c) does
is provided by $\rho (s) = 0.9s^3 + 0.1s^{38}$, for which $\psi$ has 3 positive roots
(see Figure \ref{fig8}).
\item
 An example for which condition (b) holds but conditions  (a) and (c)  do not
is $\rho(s) = 0.9s^3 + 0.1s^{24}$, for which $g^*$ has two discontinuities (see Figure \ref{fig1}).
\item
An example in which none of (a), (b) or (c) holds and where (\ref{sharp}) fails is provided by
 $\rho(s) = 0.9183 s^3+ 0.04 s^{19} + 0.0417 s^{41}$ (see Figure \ref{fig2}).
\end{itemize}

\begin{figure}[!h]
\center
\includegraphics[width=7cm]{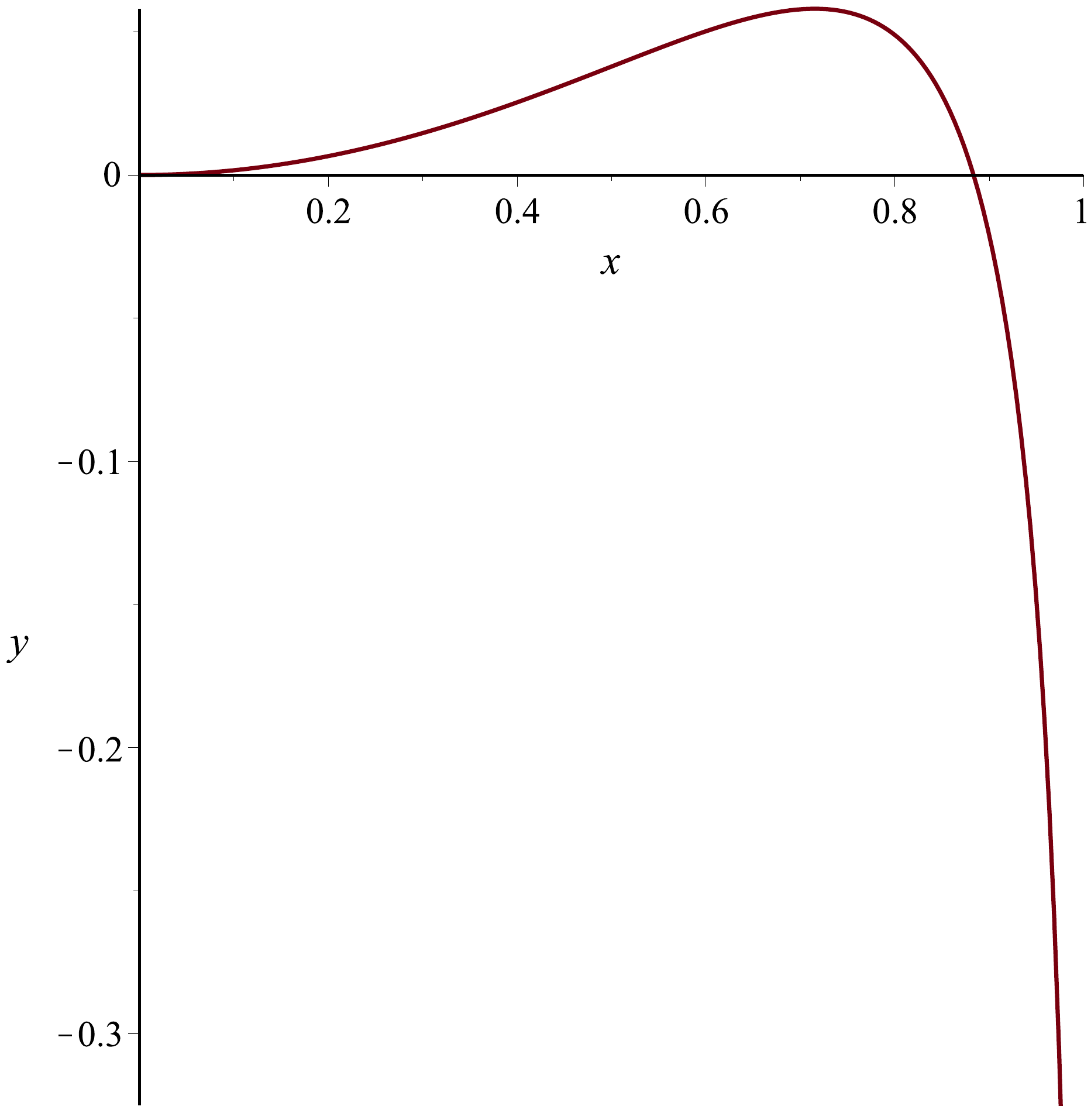} \quad
\includegraphics[width=7cm]{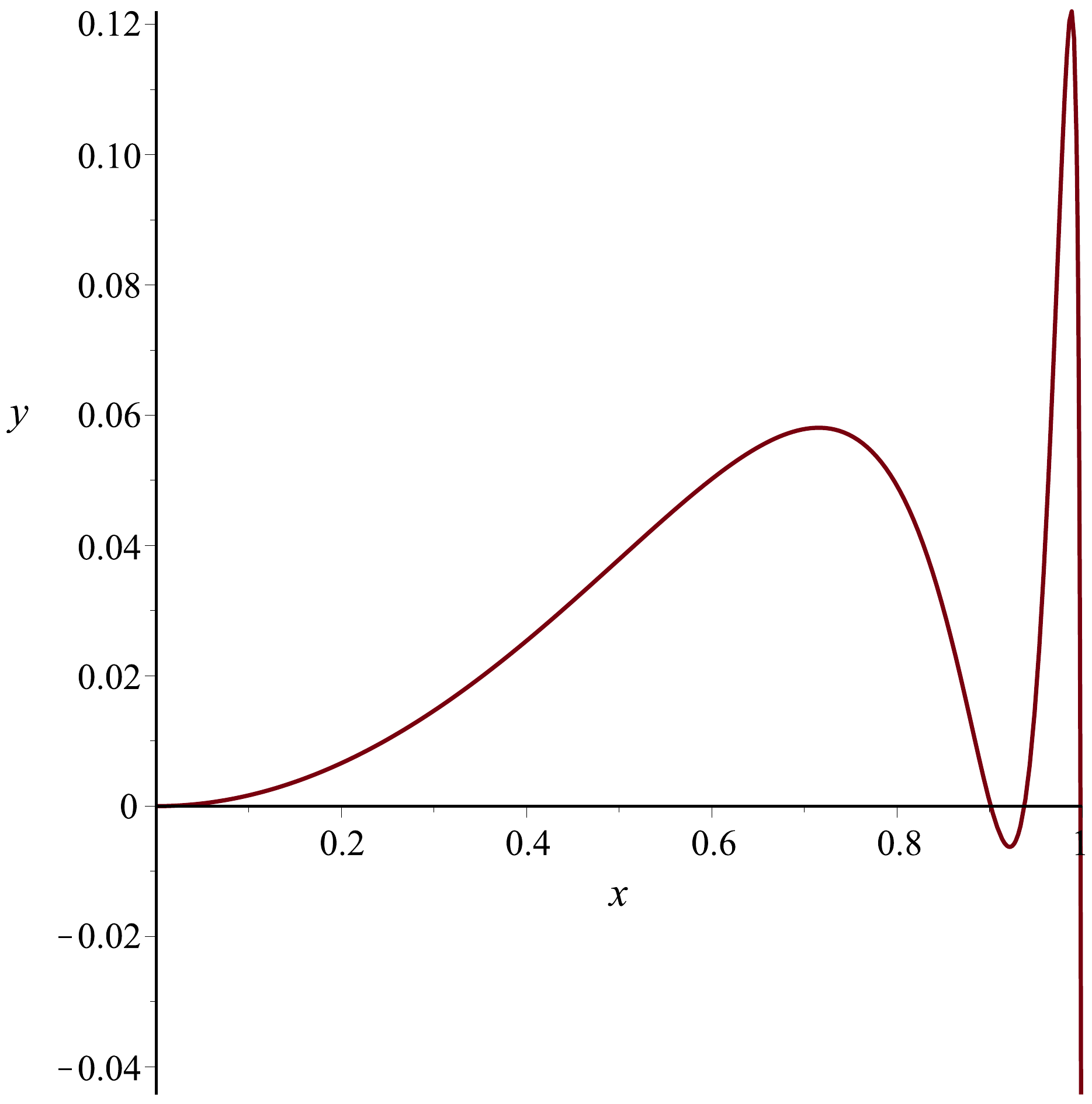}
\caption{Plots of $y= \psi(x)$  for $\rho(s) = s^3$ ({\em left})
and $\rho(s) = 0.9 s^3 +0.1 s^{38}$ ({\em right}). In the first case the only positive root is $x_1 \approx 0.883414$, while in
the second case the $3$ positive roots are $x_1 \approx 0.901174$, $x_2 \approx 0.937414$, and $x_3 \approx 0.997979$.
For the case on the right $\alpha_\rho^\sharp \approx 0.872923$ and $g^* ( \alpha_\rho^\sharp ) \approx 0.988192$,
and only the root $x_3$ exceeds  this value.
Proposition \ref{psiprop} gives $\ubar \alpha_\rho   \approx 0.917935$ for the case on the left
and   $\ubar \alpha_\rho   \approx 0.998263$ for the case on the right.}
\label{fig8}
\end{figure}

\begin{proof}[Proof of Proposition \ref{psiprop}.]
First we show that if $\Pr [ W \geq 3 ] =1$ and $\Exp[W]<\infty$, then $\psi$ has at least one root in $(0,1)$.
So suppose that there exists an integer $r \geq 3$ for which $\Pr [ W \geq r] = 1$
and $\Pr [ W = r] = p >0$. Then $\rho (s) \sim p s^r$ as $s \downarrow 0$. From (\ref{psidiff1}) we have
\begin{align}
\label{psidiff2}
\psi'' (x) & =  (1-x)^{-1} \left( \frac{2 \rho(x) \rho'' (x)}{\rho' (x)^2} - 1 \right) - (1-x)^{-2} \frac{\rho(x)}{\rho'(x)}   \nonumber\\
&~~{} - \left( \frac{\rho''(x)}{\rho'(x)} + \frac{ \rho(x ) \rho'''(x)}{\rho'(x)^2} - \frac{2 \rho(x) \rho''(x)^2}{\rho'(x)^3} \right) \log (1-x) .
\end{align}
Taking $x \downarrow 0$ in (\ref{psidiff1}) and (\ref{psidiff2}),
using   $\rho^{(k)} (x) \sim \frac{r!}{(r-k)!} p x^{r-k}$ for $k \leq 3$,
we obtain
\[ \psi' (0) =0 ; ~~~ \psi''( 0) = \frac{r -2}{r} > 0 ,\]
since $r \geq 3$.
Hence $\psi(0) =0$ is a local minimum, and $\psi (x) >0$ for $x>0$ small enough. But  $\psi(x) \to -\infty$ as $x \uparrow 1$,
so continuity implies that $\psi$ has at least one root in $(0,1)$.

Consider the condition  (a)   in the proposition.
Suppose that $h$ has a unique local minimum located at $x_\rho^\sharp \in (0,1)$, so $\ubar \alpha_\rho = h (x^\sharp_\rho)$. 
Then by Lemma
\ref{lem:hpsi}, $\psi$ has a unique local maximum at $x_\rho^\sharp$, and necessarily
$\psi ( x_\rho^\sharp) >0$. By continuity (and Rolle's theorem) it follows that $\psi$ has 
exactly one root  $x^*_\rho \in (x_\rho^\sharp, 1)$.
So (a) $\Rightarrow$ (b).
Moreover, it follows that  $\psi (x) >0$ for $x \in (0,x^*_\rho)$ and $\psi (x) <0$ for $x > x^*_\rho$.
Hence in this case $\GG_\rho =  [ x_\rho^\sharp, 1)$, and the claim (\ref{sharp}) follows.

Finally, we show that if $\rho(s) = s^r$ for some $r \geq 3$,
then (a) holds.
In the case $\rho(s) = s^r$ we obtain (cf (\ref{psidiff2}))
\[ \psi''(x) = \frac{1}{r(1-x)} \left( r -1 - \frac{1}{1-x} \right) .\]
Hence $\psi''(0) = \frac{r-2}{r} > 0$ and
 for $x \in (0,1)$
we have
 $\psi''(x) =0$
if and only if $x = \frac{r-2}{r-1}$ (an inflexion point).
So, by Rolle's theorem, $\psi'(x) =0$ for at most one $x \in (0,1)$, necessarily $x \in (\frac{r-2}{r-1},1)$,
and this must be a local maximum for $\psi$ since $\psi(x) \to -\infty$ as $x \uparrow 1$. Another application of
Rolle's theorem shows that $\psi$ has a single positive root, $x^*_r$ say,
which must be in $(\frac{r-2}{r-1},1)$.
By Lemma \ref{lem:hpsi}, the fact that $\psi$ has a single local maximum
 also shows that $h$ has a single local minimum in $(0,1)$, which is lies in
  $ (\frac{r-2}{r-1},1)$.
\end{proof}

Now we can complete the proof of   Theorem \ref{thm3}.

 \begin{proof}[Proof of Theorem \ref{thm3}.]
The expected number $\Expn  [ \NN (n,m) ]$
 of null vectors is at least one, and may be large even when
$\Prn [T_n \leq \alpha
 n ]$
 is small. Nevertheless we can derive bounds on $T_n/ n$
 by studying the asymptotics of
$\Expn [ \NN (n,m) ]$
because
\begin{equation}
\label{lowbound}
\Prn [ T_n \leq m ]
=\Prn [ \NN ( n,m)
\geq
2 ]
\leq \Expn [ \NN (n,m) ] -1,
\end{equation}
by Markov's inequality applied to the
nonnegative random variable $\NN(n,m) -1$.

Suppose that $m_n / n \to \alpha \in (0, \alpha^*_\rho)$. Then by  (\ref{lowbound}) with (\ref{polybound}),
$\Prn [ T_n \leq m_n ]
= O (n^{-1} )$.
It follows that, for any $\eps >0$,
$\Prn [ T_n \leq (\alpha^*_\rho -\eps) n ] \to 0$. On the other hand,
Theorem \ref{c:alphabar}(iii)
implies that there exists $\delta>0$ such that
 for any $\alpha \in ( \ubar \alpha_\rho, \ubar \alpha_\rho +\delta)$,
 $\Prn [ T_n \leq \alpha ] \to 1$.
Moreover, these results together show that
$\alpha^*_\rho \leq \ubar \alpha_\rho$, and
 $\ubar \alpha_\rho \leq 1$ by 
Proposition \ref{lessthan1}.
 \end{proof}

To conclude this section, we give the proof of Proposition \ref{threshasym}.

\begin{proof}[Proof of Proposition \ref{threshasym}.]
Take $\rho(s) = s^r$ for $r \geq 3$.
As already mentioned, the asymptotic for $\alpha^*_r$ is in \cite{calkin2}.
 Fix $\alpha >0$.
Then,
\[ h ( 1 - \re^{-\alpha r/2} ) = - \frac{\log ( \re^{-\alpha r/2} )}{r (
1 - \re^{ -\alpha r/2} )^{r-1}} = \frac{\alpha}{2} ( 1 + o(1) ) ,\]
as $r \to \infty$.
Hence $h (1 -\re^{- \alpha r/2} ) \leq \alpha$ for all $r$ sufficiently
large,
which by (\ref{alphasharp}) shows that $\limsup_{r \to \infty}
\alpha_r^\sharp \leq \alpha$. Since $\alpha >0$
was arbitrary, it follows that $\lim_{r \to \infty} \alpha_r^\sharp = 0$.
 
Finally, by Proposition \ref{psiprop}, (\ref{fixedbar}) holds.
Then with  (\ref{grsim}) and repeated Taylor expansions we obtain
\begin{align*} \ubar \alpha_r & = - \frac{ \log ( \re^{-r} + r^2
\re^{-2r} +O (r^4 \re^{-3r} ) ) }  {r (1 - \re^{-r} - r^2 \re^{-2r} + O
(r^4 \re^{-3r} ) )^{r-1} } \\
& =  \frac{ 1 - r^{-1} \log ( 1 + r^2 \re^{-r} +O (r^4 \re^{-2r} ) ) } 
{  (1 - \re^{-r} - r^2 \re^{-2r} + O (r^4 \re^{-3r} ) )^{r-1} } \\
& = \left( 1 - r \re^{-r} +O (r^3 \re^{-2r} ) \right) \left( 1 + (r-1)
\re^{-r} +
O (r^3 \re^{-2r} ) \right) \\
& = 1 - \re^{-r} + O (r^3 \re^{-2r} ) ,
\end{align*}
 completing the proof of (\ref{starsim}).
 \end{proof}

\section{Technical appendix}
\label{sec:appendix}

\subsection{Parity of random variables}

The following simple result, on
the probability that certain integer-valued
random variables take even values, will be used several times.
The formula in Lemma \ref{lem4}(i)
 for the probability that a binomial
variable is even may be found for example in \cite[pp.\ 277--278]{feller1}.

\begin{lemma}
\label{lem4}
 Let $X$ be a  $\ZP$-valued random variable
  with
probability
 generating function $\phi(s):=\Exp[s^X]$.
Then $\Pr [ X \in 2\Z ] = \frac{1}{2} (1 + \phi(-1))$.
In particular (i) if $X \sim \Bin (n,p)$,
then $\Pr[ X \in 2\Z ] = (1 + (1-2p)^n)/2$; and (ii)
if $X \sim \Po (\mu)$,
then $\Pr [ X \in 2 \Z ] = \re^{-\mu} \cosh \mu$.
\end{lemma}
\begin{proof}
For $X \in \ZP$,
$\1\{X \in 2\Z \} = \frac{1}{2} (1 + (-1)^X )$, yielding the first statement in the lemma.
For the $\Bin (n,p)$ case, $\phi (s ) = (ps + 1 -p)^n$, giving part (i),
while in the $\Po (\mu)$ case, $\phi (s) = \re^{\mu (s -1)}$, which gives
$\Pr [ X \in 2 \Z ] = (1+ \re^{-2\mu})/2 = \re^{-\mu} \cosh \mu$.
\end{proof}

\subsection{Parity of multinomial random variables}
\label{sec:multipar}

We saw in Lemma \ref{lem4} that the probability that a $\Bin (n,p)$
random variable is even is
$(1+(1-2p)^n)/2$. In this section
 we extend this formula to a more complicated multinomial setting
and more general congruence conditions modulo $r$.

Here is our probabilistic model.
We perform a sequence of $n$ independent trials.
Each trial is probabilistically identical,
and we are interested in the outcome of a trial
 described in terms
of
an arbitrary collection of $k$ events
$A_1 , \ldots, A_k$. Probabilities  $p(E_I)$ are
specified for each of the `elementary'  events $E_I$
defined as
\[
E_I :=  \left( \cap_{i \in I} A_i \right)
 \cap \left( \cap_{j \in [k] \setminus I} A_j\right).
\]
for each $I \subseteq [k]$ (here $[k]:= \{1,\ldots, k\}$).
We assume that
$p(E_I) >0$ for each $I$ and that $\sum_{I \subseteq [k]} p(E_I) =1$.
Use $\Pr_n$ to denote the probability measure associated
with the model consisting of $n$ trials
as just described.

Let  $N_i$ be the number of occurrences of event $A_i$.
For $L \subseteq J \subseteq [k]$ set
$$
E_{J,L} := \left( \cap_{i \in L} A_i \right) \cap
 \left( \cap_{i \in J \setminus L} A_i^c \right)
$$
and set $E_J:= E_{[k],J}$. Also, set
$
p_o(J) $ to be the probability for a single trial that an odd number
of outcomes $A_i, i \in J$ occur, and note that
\begin{align*}
1-2 p_o (J) = \sum_{L \subseteq J}(-1)^{|L|} p(E_{J,L}),
\end{align*}
where $|L|$ denotes the number of elements of $L$.
The next result gives a general
formula  for the probability in $n$ trials that
for each $i$ the $N_i$ falls into a particular congruence
class modulo $r$, and a specialization to
the event that $N_i$ is even for each $i$ in a specified subset $I$ of
$[k]$ and $N_i$ is odd for each $i \in [k] \setminus I$.
For positive integer $r$, define the complex number
$\omega := \re^{(2 \pi /r) i}$ (a  complex $r$th root of unity).

\begin{lemma}
\label{parity}
(i) Let $r \geq 2$ be an integer and
 $\bt = (t_1,\ldots, t_k)  \in \{0,1,\ldots,r-1\}^k$.
Then
\begin{align}
\Pr_n \left[ \cap_{i=1}^k \{ N_i \equiv t_i ~ (\mod r) \} \right]
= r^{-k}
\sum_{\bh \in \{0,1,\ldots,r-1\}^k}
 \omega^{- \bt \cdot  \bh }
\left(   \sum_{\bg \in \{0,1\}^k} \omega^{\bg \cdot \bh } p_{\bg} \right)^n,
\label{maineq}
\end{align}
where
 for $\bg = (g_1,\ldots,g_k) \in \{0,1\}^k$ we set $p_{\bg} = p(E_{\bg})$ with
the event
$$
E_\bg := \left( \cap_{i \in \{1,\ldots,k\}: g_i =1} A_i \right) \cap
\left( \cap_{i \in \{1,\ldots,k\}: g_i =0} A_i^c \right)
$$
defined in terms of a single trial, and
$\bt \cdot \bh := \sum_{i=1}^k t_i h_i $ and
$\bg \cdot \bh := \sum_{i=1}^k g_i h_i $.

(ii) For any $I \subseteq [k]$,
\begin{align}
\Pr \left[\left( \cap_{i \in I} \{ N_i \in 2 \Z \} \right)
\cap
\left( \cap_{j \in [k] \setminus I} \{ N_j \notin 2\Z \}
\right)
\right]
= 2^{-k} \sum_{J \subseteq [k]} (-1)^{|J \cap I|} (1 - 2 p_o(J))^n
.
\end{align}
\end{lemma}
\begin{proof}
 Part (ii) follows from the $r=2$ case of part (i)
on setting $t_i$ to be $1$ on $I$ and $0$ otherwise. Thus we
need to prove part (i).
We obtain (\ref{maineq}) by induction on $n$. First consider the case $n=0$.
In this case
the left side of (\ref{maineq}) is equal to 1 if $t_i =0$ for all $i$ and is equal to zero
otherwise, and the right side is equal to
\[
r^{-k} \sum_{\bh \in \{0,1, \ldots , r-1\}^k } \omega^{-\bt \cdot \bh} .\]
If $\bt = \0$ then  this expression is 1.
Otherwise, it is zero since if  $t_i\neq 0$ for some $i$ then
\[
\sum_{h_i=0}^{r-1}
\omega^{-t_i h_i} = \frac{ 1 - \omega^{- r t_i } }{ 1 - \omega^{-t_i} } = 0.
\]
Thus the inductive hypothesis holds for $n=0$. Suppose it holds for some $n$.
In the case of $n+1$ trials, conditioning on the outcome of the first
trial we obtain
\begin{align*}
\Pr_{n+1}  \left[ \cap_{i=1}^k \{ N_i \equiv t_i ~ (\mod r) \} \right]
= \sum_{\bg \in \{0,1\}^k} p_{\bg}
\Pr_{n} \left[ \cap_{i=1}^k \{ N_i \equiv t_i ~ (\mod r) \} | E_\bg \right]
\\
= \sum_{\bg \in \{0,1\}^k} p_{\bg}
 \Pr_{n+1} \left[ \cap_{i=1}^k \{ N_i \equiv t_i - g_i ~ (\mod r) \} \right]
\\
= r^{-k}
 \sum_{\bg \in \{0,1\}^k} p_{\bg}
\sum_{\bh \in \{0,1,\ldots,r-1\}^k}
 \omega^{- (\bt - \bg) \cdot  \bh }
\left(   \sum_{\bbf \in \{0,1\}^k}
\omega^{\bbf \cdot \bh }
 p_{\bbf}
 \right)^n
\\
= r^{-k}
\sum_{\bh \in \{0,1,\ldots,r-1\}^k}
 \omega^{- \bt  \cdot  \bh }
 \left( \sum_{\bg \in \{0,1\}^k} \omega^{\bg \cdot \bh} p_{\bg} \right)
 \left( \sum_{\bbf \in \{0,1\}^k} \omega^{\bbf \cdot \bh} p_{\bbf}  \right)^n
\\
= r^{-k}
\sum_{\bh \in \{0,1,\ldots,r-1\}^k}
 \omega^{- \bt  \cdot  \bh }
 \left( \sum_{\bg \in \{0,1\}^k} \omega^{\bg \cdot \bh} p_{\bg}  \right)^{n+1},
\end{align*}
which completes the induction. \end{proof}

\subsection{Generating function properties}

The next result collects some elementary properties of probability generation functions.

\begin{lemma}
\label{pgf}
Let $\phi(s) := \Exp [ s^X]$, $s \in [-1,1]$, for a $\ZP$-valued random variable $X$. Then $\phi(0) = \Pr[ X =0]$, $\phi(1) = 1$,
and $\phi(s)$ is infinitely differentiable at least for $s \in (-1,1)$; if $\Exp [X] < \infty$ then
$\phi'(s) = \frac{\ud}{\ud s} \phi (s)$ is   continuous in the closed interval $[-1,1]$.
Moreover,
\begin{itemize}
\item[(i)] Suppose that $\Pr [ X=0]=0$. Then  as $s \downarrow 0$,
\[ \phi (s) = s \Pr[ X=1] + O(s^2), ~\textrm{and}~ \phi'(s) = \Pr[ X=1] + O(s) .\]
\item[(ii)] If $\Exp [ X] < \infty$, then as $s \downarrow 0$,
\[ \phi (1-s) = 1 - s \Exp[ X] + o(s), ~\textrm{and}~ \phi'(1-s) = \Exp[ X] + o(1) .\]
\item[(iii)] For any $s \in [0,1]$, $| \phi (-s) | \leq \phi (s)$.
\end{itemize}
\end{lemma}
\begin{proof}
Apart perhaps from part (iii),
all of the properties stated in the lemma are well known: see for example \cite[pp.\ 264--266]{feller1}.
For part (iii), let $s \in [0,1]$. Then
\begin{align*}
| \phi (-s)|  \leq  \Exp [ \left|(-s)^X\right|  ]
= \Exp [ s^X ]  = \phi(s).
\end{align*}
\end{proof}

\subsection{Asymptotic estimates}

We shall use the following
bounds on the binomial coefficient $\binom{n}{k}$.
\begin{lemma}
Let $n \in \N$ and $k \in \{0,1,\ldots, n\}$. Then
\begin{align}
\label{chooseub2}
\binom{n}{k} & \leq
 \left( \left(\frac{k}{n}\right)^{k/n}
\left( 1 - \frac{k}{n} \right)^{1-(k/n)}
\right)^{-n} \leq n^k \re^k k^{-k}. \end{align}
On the other hand, if   $0 < k < n$,
\begin{align}
\label{chooselb}
\binom{n}{k} &  \geq
\left( \frac{n}{2 \pi k (n-k) } \right)^{1/2}
\re^{-1/6} \left( \left(\frac{k}{n}\right)^{k/n}
\left( \frac{n-k}{n} \right)^{(n-k)/n}
\right)^{-n}.
\end{align}
\end{lemma}
\begin{proof}
We apply
Robbins's
refinement of Stirling's formula (see e.g.~\cite[\S II.9]{feller1}),
which says that for any $n \geq 1$,
\[ n! = ( 2 \pi)^{1/2} n^{n+(1/2)} \re^{-n + \eps_n} ,\]
where $\frac{1}{12n+1} < \eps_n < \frac{1}{12n}$.
This yields the upper bound, for $n \geq 1$ and $k, n-k \geq 1$,
\begin{align}
\label{chooseub}
\binom{n}{k} & \leq
\left( \frac{n}{2 \pi k (n-k) } \right)^{1/2}
 \left( \left(\frac{k}{n}\right)^{k/n}
\left( \frac{n-k}{n} \right)^{(n-k)/n}
\right)^{-n} ,
\end{align}
where
we have used the fact that
\[ \eps_n - \eps_k - \eps_{n-k} \leq \frac{1}{12n} -
 \frac{12n +2}{144 k (n-k) +12n +1}
\leq \frac{1}{12n} -\frac{12n+2}{36n^2 +12n +1 } \leq 0
 ,
\]
since $k(n-k) \leq n^2/4$. By considering separately the cases
 (i) $k \in \{0,n\}$, and
(ii) $0<k <n$, 
using (\ref{chooseub})
 in case (ii), we obtain  the first inequality in (\ref{chooseub2}).
 The second inequality in (\ref{chooseub2}) follows from the fact that
 \[ \left( 1 - \frac{k}{n} \right)^{-(n-k)}  = \left( 1 + \frac{k}{n-k} \right)^{n-k} \leq \re^{k} .\]
 For the lower bound, another application of Robbins's bounds yields  (\ref{chooselb}),
where for the $\re^{-1/6}$ term we have used the fact that
$\eps_n - \eps_k - \eps_{n-k} \geq - \frac{1}{12k} - \frac{1}{12(n-k)} \geq -\frac{1}{6}$. \end{proof}

\section*{Acknowledgements}

The authors thank Julian West for pointing out the elementary application of
the pigeonhole principle in the proof of Lemma \ref{hypercycle-2core}, thus avoiding the use of linear algebra.

\end{document}